\documentclass[reqno]{amsart}

\usepackage{graphicx}

\usepackage{enumerate}

\makeatletter
\newcommand*\bigcdot{\mathpalette\bigcdot@{.5}}
\newcommand*\bigcdot@[2]{\mathbin{\vcenter{\hbox{\scalebox{#2}{$\m@th#1\bullet$}}}}}
\makeatother

\usepackage{amsmath}
\usepackage{amsthm}
\usepackage{amsfonts}
\usepackage{mathrsfs}
\usepackage{xcolor}

\usepackage{amssymb}
\usepackage{relsize}
\usepackage{mathtools}
\usepackage{bm}

\usepackage{latexsym,tikz-cd}
\usepackage[normalem]{ulem}
\usepackage[colorlinks=true,linkcolor=blue]{hyperref}

\numberwithin{equation}{section}

\newtheorem{theorem}{Theorem}[section]
\newtheorem{proposition}[theorem]{Proposition}
\newtheorem{corollary}[theorem]{Corollary}
\newtheorem{lemma}[theorem]{Lemma}

\newtheorem*{openquestion}{Open Question}
\theoremstyle{definition}
\newtheorem{definition}[theorem]{Definition}
\theoremstyle{remark}
\newtheorem{remark}[theorem]{Remark}

\newcommand{\nwc}{\newcommand}

\nwc{\Oph}{\operatorname{Op}_\hbar}

\renewcommand{\Re}{\operatorname{Re}}

\DeclareMathOperator{\Tr}{Tr}

\newcommand{\La}{\Lambda}

\renewcommand{\phi}{\varphi}

\nwc{\gl}{\langle}
\nwc{\gr}{\rangle}

\newcommand{\N}{{\mathbb N}}

\newcommand{\R}{{\mathbb R}}
\renewcommand{\H}{{\mathbb H}}
\newcommand{\C}{{\mathbb C}}

\newcommand{\Z}{{\mathbb Z}}

\newcommand{\D}{{\mathbb D}}

\nwc{\rest}{\restriction}

\nwc{\defeq}{\stackrel{\rm{def}}{=}}
\renewcommand{\d}{\partial}

\title[Steklov Determinant and Isospectral Compactness]{The Steklov Determinant and Compactness of Isospectral Planar Domains}

\date{}

\author{Yujun Jin, Zuoqin Wang}
\thanks{Partially supported by   NSFC no. 12571064.}
\address{School of Mathematical Sciences\\
	University of Science and Technology of China\\
	Hefei, 230026\\ P.R. China}
\email{byjyj@mail.ustc.edu.cn}

\address{School of Mathematical Sciences\\
	University of Science and Technology of China\\
	Hefei, 230026\\ P.R. China}
\email{wangzuoq@ustc.edu.cn}

\begin{document}

\begin{abstract}
We prove that every Steklov isospectral family of compact smooth planar domains is compact in the $C^\infty$ topology, answering an open question of Colbois, Girouard, Gordon, and Sher.
The proof has two main parts. First, a trace comparison principle for Dirichlet-to-Neumann operators yields monotonicity of negative Steklov zeta values and compactness within a fixed conformal class for genus-zero flat surfaces. Second, we analyze the normalized Steklov determinant on degenerating hyperbolic surfaces with geodesic boundary. Its asymptotics are expressed in terms of shrinking boundary components and small Neumann and Dirichlet eigenvalues, which in genus zero are compared with weighted graph Laplacians. This rules out degeneration under a bound on the total hyperbolic boundary length. We finally establish this bound for planar domains by geometric non-collapse estimates.
\end{abstract}
\maketitle
\vspace{-1.5em}
\tableofcontents
\section{Introduction}

Kac's question ``Can one hear the shape of a drum?'' \cite{kac}, is one of the classical starting points of inverse spectral
geometry. In its original form, it asks whether a planar domain is
determined by the spectrum of its Dirichlet Laplacian. The
counterexamples of Gordon, Webb, and Wolpert \cite{GWW92} showed that the
answer is negative. This is part of a broader non-uniqueness phenomenon
in inverse spectral geometry, illustrated earlier by Milnor's
isospectral manifolds \cite{Milnor1964}, and subsequently by the constructions of Gordon and Wilson
\cite{GordonWilson1984} and Sunada's
method \cite{Sunada1985}. More recently, Hu, Shi, and Tang
\cite{HuShiTang2026} constructed noncongruent Steklov-isospectral
strictly convex planar domains.

These non-uniqueness results naturally lead to a weaker question: even
when the spectrum does not determine the geometry uniquely, does fixing
the spectrum at least prevent the underlying geometry from degenerating?
In other words, are isospectral families compact?

Such a compactness question can be formulated for many different spectral problems. In the present paper, we study it in the setting of the Steklov spectrum.

Let $(\Omega,g)$ be a compact Riemannian surface with nonempty smooth
boundary $M=\partial\Omega$. The Steklov eigenvalue problem is given by
\[
\begin{cases}
\Delta_g u=0, & \text{in }\Omega,\\
\partial_\nu u=\sigma u, & \text{on }M,
\end{cases}
\]
where $\nu$ denotes the outward unit normal vector field along the boundary.
Equivalently, the Steklov spectrum is the spectrum of the
Dirichlet-to-Neumann operator on $M$
\[
\Lambda_{\Omega,g}f:=\frac{\partial \mathcal H_{\Omega}f}{\partial\nu}\bigg|_{M},
\]
where $\mathcal H_{\Omega}f$ denotes the harmonic extension of $f$ to
$\Omega$. It is well known that $\Lambda_{\Omega,g}$ is a self-adjoint elliptic pseudodifferential operator of order one \cite{LU1989}. Its eigenvalues form a discrete sequence
\[
0=\sigma_0(\Lambda_{\Omega,g})<\sigma_1(\Lambda_{\Omega,g})\leq \sigma_2(\Lambda_{\Omega,g})\leq \cdots,
\]
where each eigenvalue is repeated according to its multiplicity.

To formulate the compactness question precisely, we first specify the topologies in which convergence of surfaces will be understood. Throughout the paper, by a compact family of Riemannian surfaces we mean a family that is relatively compact modulo diffeomorphisms:
\begin{definition}
Let $\mathcal F$ be a family of compact smooth Riemannian surfaces with boundary.
We say that $\mathcal F$ is \emph{compact in the $C^\infty$ topology}
if every sequence $\{S_k\}\subset\mathcal F$ admits a subsequence
$\{S_{k_j}\}$ for which there exist a compact reference surface $\Omega$,
smooth metrics $g_{k_j}$ on $\Omega$, and isometries
\[
P_{k_j}:(\Omega,g_{k_j})\longrightarrow S_{k_j},
\]
such that $g_{k_j}$ converges in $C^\infty$ to a smooth metric $g$ on
$\Omega$. For $s\geq0$, \emph{compactness in the $H^s$ topology} is
defined analogously, with the convergence $g_{k_j}\to g$ understood in
$H^s(\Omega)$ with respect to any fixed smooth background metric on
$\Omega$.
\end{definition}

However, Steklov isospectral families need not be compact in this sense
for arbitrary Riemannian metrics. Indeed, if $\rho\in C^\infty(\Omega)$ is any positive function satisfying $\rho|M=1$, then the conformal invariance of harmonic functions in dimension two gives
\[
\Lambda_{\Omega,\rho g}=\Lambda_{\Omega,g}.
\]
Thus, the Steklov spectrum is insensitive to conformal deformations of the metric in the interior that leave the boundary metric unchanged. By choosing such conformal factors without any compactness control, one obtains Steklov isospectral families that fail to be compact in any of the above topologies, even within a fixed conformal class.

To eliminate this intrinsic conformal freedom, we restrict throughout the paper to flat metrics. Within this class, the preceding obstruction disappears. Indeed, if both $(\Omega,g)$ and $(\Omega,\rho g)$ are flat and $\rho=1$ on $M$, then the conformal change formula for the Gaussian curvature implies that $\log\rho$ is harmonic in $\Omega$. Since $\log\rho$ vanishes on $M$, the maximum principle therefore implies that $\rho\equiv1$ on $\Omega$.

This restriction also places our problem in direct analogy with the classical compactness theory for Laplace isospectral surfaces. Osgood, Phillips, and Sarnak \cite{OPSmodulispace} proved that bounded planar domains that are isospectral for the Dirichlet Laplacian form a $C^\infty$-compact family. Kim later extended this result to Dirichlet-isospectral flat metrics on compact orientable surfaces with boundary and negative Euler characteristic \cite{Kim2008}. See also the earlier work of Osgood, Phillips, and Sarnak on closed isospectral surfaces \cite{OSGOOD1988212}. In dimensions greater than two, compactness results are more limited and generally require additional geometric assumptions or restrictions on the class of metrics; see, for example, \cite{ChangYang,ChenXu1996,LiuWang2019}.

By contrast, substantially less is known about compactness for Steklov isospectral surfaces. For simply connected planar domains, Edward
\cite{Edward01011993} proved compactness in the $H^s$ topology for every
$s<\frac52$. This result was later strengthened by Jollivet and Sharafutdinov
\cite{JOLLIVET20181712} to compactness in the $C^\infty$ topology.
Motivated by the corresponding result for the Dirichlet Laplacian, Colbois, Girouard, Gordon, and Sher posed the following Steklov analog for multiply connected planar domains in their recent survey:
\begin{openquestion}[\cite{Colbois2024}, Open Question 8.5]
	Are families of multiply connected, compact, Steklov isospectral planar domains necessarily compact in the $C^\infty$ topology?
\end{openquestion}

One of the main results of the present paper answers this question affirmatively for planar domains of connectivity $n\geq 3$ (see Theorem~\ref{thm::planar}). Together with the simply connected case established by Jollivet and Sharafutdinov \cite{JOLLIVET20181712} and the doubly connected case proved in our previous work \cite{annular}, this yields the following complete result.
\begin{theorem}\label{main1}
    Any Steklov isospectral family of compact smooth planar domains is compact in the $C^\infty$ topology.
\end{theorem}

Our proof separates the compactness problem into two parts. First, within a fixed conformal class, we establish uniform control of the conformal factor. Second, when the conformal class varies, we rule out degeneration in the corresponding moduli space. The main objective of this paper is to resolve both issues for flat genus-zero Steklov isospectral surfaces.

The first ingredient is a trace comparison principle for the Dirichlet-to-Neumann operators associated with nested collections of boundary components. More precisely, after decomposing the boundary according to a partition $\mathcal P$ of the index set $I$, we compare $\Lambda_{\Omega_I,g}$ with the direct sum of the Dirichlet-to-Neumann operators on the corresponding filled-in surfaces. For every positive smooth function $c$ on the boundary and for every $s\geq1$, we prove
\[
\operatorname{Tr}\left((c\Lambda_{\Omega_I,g}c)^s
-(c\Lambda_{\mathcal P}c)^s\right)\geq0.
\]
Together with some trace positivity estimates proved in Section~\ref{sec::compactnessfixedconclass}, this comparison also gives monotonicity properties for the negative values of the Steklov spectral zeta function, as stated in Corollary~\ref{zetasum} and Corollary~\ref{cor::zetamonotonicity}.

We apply this trace comparison to Steklov isospectral flat surfaces in a fixed conformal class. By decomposing the spectral trace invariants into nonnegative contributions and applying the cylinder estimates from \cite{annular}, we obtain uniform boundary control of the conformal factor. Flatness then extends this control to the interior and yields $C^\infty$-compactness for genus-zero flat surfaces with connectivity $n\geq3$ (see Theorem~\ref{thm::maincompactnessflat}). Together with the simply connected case \cite{JOLLIVET20181712} and the annular case \cite{annular}, this gives the following main result.
\begin{theorem}
Any Steklov isospectral family of compact smooth genus-zero flat surfaces
lying in a fixed conformal class is compact in the $C^\infty$ topology.
\end{theorem}

The remaining task is to control the conformal class. For this purpose we
use the canonical hyperbolic representative of a conformal class. By the
uniformization theorem of Osgood, Phillips, and Sarnak \cite{OSGOOD1988148}, each
conformal class on a surface with negative Euler characteristic contains a
unique hyperbolic metric with geodesic boundary, up to isometry. Therefore, degeneration of conformal classes can be
studied through degeneration of hyperbolic surfaces with geodesic boundary.

A crucial step of this part is a uniform asymptotic analysis of the Steklov determinant on degenerating hyperbolic surfaces. In related work, Wentworth \cite{WentworthDN} studied the moduli space of genus-zero hyperbolic surfaces with at least three geodesic boundary components of prescribed lengths. He proved that the associated determinant height function is proper on this moduli space. Since the Steklov spectrum determines the number and the multiset of boundary lengths \cite{GIROUARD_PARNOVSKI_POLTEROVICH_SHER_2014}, this properness result implies that any Steklov isospectral family of such hyperbolic surfaces is compact in the $C^\infty$ topology.

Our analysis extends this approach in two directions: the boundary lengths may vary, and the hyperbolic surfaces may degenerate within any fixed topological type. Let $X$ be a compact orientable hyperbolic surface with geodesic boundary and negative Euler characteristic. The starting point is the formula of Guillarmou and Guillop\'e \cite[Theorem 1.3]{8180465},
\begin{equation*}
\frac{{\det}'\Lambda_X}{\ell(\partial X)}=-\frac{Z_G'(1)}{(Z_{G_0}(1))^2}
\frac{e^{\frac{\ell(\partial X)}{4}}}{2\pi\chi(X)},
\end{equation*}
where $Z_G$ is the Selberg zeta function of the double of $X$, and $Z_{G_0}$ is the Selberg zeta function associated with the bordered surface $X$ (see Section~\ref{sectiondegenration} for details). Thus, the degeneration of the normalized Steklov determinant can be studied through the asymptotic behavior of these two Selberg zeta functions.

To carry out this analysis, we distinguish three types of short geodesics: shrinking boundary components, short interior closed geodesics, and short reflected geodesics meeting the boundary orthogonally. Combining Wolpert's asymptotic formula for $Z_G'(1)$ with uniform estimates for $Z_{G_0}(1)$, we obtain the following key uniform
two-sided estimate, as stated in Theorem~\ref{thm::shortgeodesics}:
\begin{equation}\label{intro::eq1}
\frac{{\det}'\Lambda_X}{\ell(\partial X)}
\asymp_{n,\mathtt g}e^{\frac{\ell(\partial X)}{4}}
\prod_{b_i<r^*} b_i\frac{
\prod_{0<\lambda_k^N(X)<\frac14}\lambda_k^N(X)}
{\prod_{\lambda_k^D(X)<\frac14}\lambda_k^D(X)}.
\end{equation}
Here $b_i$ are the boundary lengths, while $\lambda_k^N(X)$ and $\lambda_k^D(X)$ denote the Neumann and Dirichlet eigenvalues of the hyperbolic Laplacian on $X$, respectively. The proof of this estimate is given in Section~\ref{sectiondegenration} and relies on the heat trace estimates established in Appendix~\ref{heattrace}.

The quotient of small Neumann and Dirichlet eigenvalues appearing in \eqref{intro::eq1} is then studied in the genus-zero case. Under a uniform upper bound on the total boundary length, degenerations caused by reflected short geodesics can be excluded. The remaining possible degenerations arise from pinching interior separating geodesics or from shrinking boundary components.
In Appendix~\ref{smalleigenvaluesDN}, we establish uniform two-sided comparisons between the small eigenvalues of $X$ and those of suitable weighted graph Laplacians. More precisely, the small Neumann spectrum is modeled by a graph Laplacian whose edge weights are determined by the pinched interior geodesics, whereas the small Dirichlet spectrum is modeled by a related graph problem incorporating additional weights associated with the boundary components. These graph comparisons are related to the classical theory
of small Laplace eigenvalues on boundaryless surfaces, including
the estimates of Schoen--Wolpert--Yau \cite{SWY} and the graph-comparison
results of Colbois \cite{Colbois1985},
Dodziuk--Pignataro--Randol--Sullivan
\cite{DodziukPignataroRandolSullivan1987}, and Burger
\cite{Burger1990}.

Combining the determinant asymptotics with these graph estimates, we obtain a properness result for the normalized Steklov determinant. More precisely, for every $T>0$, on the region of the moduli space defined by
$\ell(\partial X)\leq T$, the function
\[
X\longmapsto \frac{{\det}'\Lambda_X}{\ell(\partial X)}
\]
is proper when regarded as a map into $(0,\infty)$. In fact, its positive superlevel sets are compact. Indeed, we prove in Proposition~\ref{prop:ratio-tends-zero-upper-bound} that, along every
degenerating family in this region, this positive function tends to zero.
Since the normalized Steklov determinant
is conformally invariant by Guillarmou and Guillop\'e \cite{8180465},
\[
\frac{{\det}'\Lambda_{\Omega,g}}{L_g(\partial\Omega)}=\frac{{\det}'\Lambda_{X}}{\ell(\partial X)},
\]
this properness prevents degeneration of the
hyperbolic representatives associated with a Steklov isospectral family.

Combining this nondegeneration result with the compactness theorem within a fixed conformal class yields the following result (see Theorem~\ref{thm::utcompactness}).

\begin{theorem}\label{maincompactnesshyperbolic}
Fix $T>0$. Then any Steklov isospectral family of compact smooth genus-zero flat surfaces whose
associated hyperbolic representatives satisfy
\[
\ell(\partial X)\leq T
\]
is compact in the $C^\infty$ topology.
\end{theorem}

The final part of the paper establishes the required bound for planar domains and thereby proves Theorem~\ref{thm::planar}, which in turn completes the proof of Theorem~\ref{main1}. Let
$(\Omega,g_E)\subset\mathbb R^2$ be a smooth bounded planar domain with
$n\geq3$ boundary components. We prove two non-collapse estimates depending only
on the Steklov spectrum. First, distinct boundary components cannot approach
each other in the interior of $\Omega$. This follows from Corollary~\ref{cor::zetamonotonicity}, together with the cylinder estimates used in the proof of Theorem~\ref{thm::maincompactnessflat}. Second, a single boundary component cannot
develop an inner self-collapse. This is obtained by combining Corollary~\ref{cor::zetamonotonicity} with the compactness theorem of Jollivet and Sharafutdinov \cite{JOLLIVET20181712} for flat topological disks.
Together these estimates give a uniform embedded Euclidean collar around each
boundary component. The collar geometry then yields a uniform upper bound for
the total boundary length of the associated hyperbolic representative.
Applying Theorem~\ref{thm::utcompactness}, we obtain the main planar
compactness theorem.

We now briefly describe the organization of the paper. In
Section~\ref{traceineq}, we prove the zeta-function comparison theorem and its
trace-inequality consequences. In
Section~\ref{sec::compactnessfixedconclass}, we establish compactness for flat
genus-zero surfaces in a fixed conformal class.
Section~\ref{sectiondegenration} is devoted to the asymptotic formula for the
normalized Steklov determinant on degenerating hyperbolic surfaces, and the
heat-trace estimates needed there are proved in Appendix~\ref{heattrace}.
The graph approximation for small Neumann and Dirichlet eigenvalues is carried
out in Appendix~\ref{smalleigenvaluesDN}. These ingredients are combined in
Section~\ref{sec::compactnessflat} to prove
Theorem~\ref{thm::utcompactness}. The final section applies this result to
smooth planar domains and proves the planar compactness theorem.

\section{Comparison and monotonicity of the spectral zeta function}\label{traceineq}
Let $\mathcal{S}_{\mathtt{g}}$ be a connected closed surface of genus $\mathtt{g}$, and let $U_1,\ldots,U_n$ be open disks whose closures
$\bar U_1,\ldots,\bar U_n$ are pairwise disjoint smoothly embedded closed disks in $\mathcal{S}_{\mathtt g}$. For any nonempty index set $I \subset \{1,\dots,n\}$, define $\Omega_I\subset \mathcal{S}_{\mathtt{g}}$ to be the surface with boundary obtained by removing the disks $U_i$, $i \in I$, from $\mathcal{S}_{\mathtt{g}}$, that is,
\[
\Omega_I:=\mathcal{S}_{\mathtt{g}}\setminus\bigcup_{i\in I} U_i.
\]
The Euler characteristic of $\Omega_I$ is given by
\[
\chi(\Omega_I)=
\begin{cases}
2-2\mathtt{g}-\# I, & \text{if } \mathcal{S}_{\mathtt{g}} \text{ is orientable},\\
2-\mathtt{g}-\# I, & \text{if } \mathcal{S}_{\mathtt{g}} \text{ is non-orientable}.
\end{cases}
\]

Let $M_i := \partial U_i$ denote the boundary of $U_i$. Clearly, if $J \subset I \subset \{1,\dots,n\}$, then we have $\Omega_I \subset \Omega_J$ and $\partial \Omega_J \subset \partial \Omega_I$. In particular, we denote by
\[
\Omega:=\mathcal{S}_{\mathtt{g}}\setminus\bigcup_{i=1}^n U_i
\]
the surface obtained by removing all the disks. Its boundary is denoted by
\[
M=M_1\cup M_2\cup \cdots \cup M_n.
\]

Given any smooth Riemannian metric $g$ on $\mathcal{S}_\mathtt{g}$, we still denote by $g$ its restriction to each $\Omega_I$, and by $L_g(M_i)$ the length of $M_i$ with respect to the metric $g$.

For compact surfaces with or without boundary, Osgood, Phillips, and Sarnak \cite{OSGOOD1988148} studied the determinant of the Laplacian as a functional on a conformal class and showed that its extremals single out canonical metrics.
In particular, this yields the following uniformization result, stated in our notation.
\begin{theorem}[\cite{OSGOOD1988148}]\label{thm::OPS}
Let $(\mathcal{S}_{\mathtt{g}},g)$ be a closed surface endowed with a smooth metric $g$. Then the conformal class $[g]$ contains a metric of constant curvature, unique up to multiplication by a positive constant.

Moreover, for each compact domain $\Omega_I$ introduced above, the following canonical representatives exist uniquely in the restricted conformal class $[g]_{\Omega_I}$, up to multiplication by a positive constant:
\begin{enumerate}[(1)]
\item (Type I) a constant curvature metric with geodesic boundary;
\item (Type II) a flat metric whose boundary has constant geodesic curvature.
\end{enumerate}
Here, $[g]_{\Omega_I}$ denotes the restriction of the conformal class $[g]$ to $\Omega_I$, namely
\[
[g]_{\Omega_I}:=[g|_{\Omega_I}]=\{e^{2\phi}g|_{\Omega_I} \mid \phi\in C^{\infty}(\Omega_I;\R)\}.
\]
\end{theorem}

This uniformization theorem allows us to select canonical representatives of conformal classes on $\Omega_I$.
As an illustration, we give two examples here.
\begin{enumerate}
    \item If $\mathtt{g}=0$, and $\# I=2$, then $\chi(\Omega_I)=0$. By the Gauss--Bonnet theorem and Theorem~\ref{thm::OPS}, there exists a unique metric (of both Type I and Type II)
$g_0 \in [g]_{\Omega_I}$ such that $(\Omega_I, g_0)$ is isometric to a flat cylinder
\[
\mathcal C_l := S^1 \times [0,l], \qquad S^1=\mathbb R/(2\pi\mathbb Z),
\]
with the product metric $d\theta^2+dt^2$, for some $l>0$.  With this normalization, $l$ uniquely determines the conformal modulus. This example plays a crucial role in our earlier paper \cite{annular}. 
\item If the Euler characteristic $\chi(\Omega_I)<0$, then by the Gauss--Bonnet theorem and Theorem~\ref{thm::OPS}, there exists a unique metric (of Type I)
$\hat{g} \in [g]_{\Omega_I}$ such that $(\Omega_I, \hat{g})$ is isometric to a hyperbolic surface with geodesic boundary.
Consequently, conformal classes on $\Omega_I$ are in one-to-one correspondence with hyperbolic metrics with geodesic boundary.
\end{enumerate}

Let $\Lambda_{\Omega_I,g}$ denote the \emph{Dirichlet-to-Neumann operator} associated with $\Omega_I$ with respect to the metric $g$. 
Let $D_{\Omega_I,g}$ denote the tangential derivative operator along $\partial\Omega_I$, namely
\[
D_{\Omega_I,g}f:=\frac{1}{\sqrt{-1}}\partial_\tau f,
\]
where, on each oriented boundary component, $\tau$ is the positively oriented arc-length parameter with respect to $g$. Reversing the orientation changes $D_{\Omega_I,g}$ to its negative and therefore does not affect $|D_{\Omega_I,g}|$.
Suppose $I'\supset I$ is another set of indices. When no confusion can arise, we use the same notation $\Lambda_{\Omega_I,g}$ and $D_{\Omega_I,g}$ for their extensions to $\partial\Omega_{I'}$ obtained by setting them equal to zero outside $\partial\Omega_I$. More precisely, for $f\in C^\infty(\partial\Omega_{I'})$,
\[
\Lambda_{\Omega_I,g}f:=
\begin{cases}
\Lambda_{\Omega_I,g}(f|_{\partial\Omega_I}),
& \text{on } \partial\Omega_I,\\
0,& \text{on } \partial\Omega_{I'}\setminus\partial\Omega_I,
\end{cases}
\]
and
\[
D_{\Omega_I,g}f:=
\begin{cases}
D_{\Omega_I,g}(f|_{\partial\Omega_I}),
& \text{on } \partial\Omega_I,\\
0,& \text{on } \partial\Omega_{I'}\setminus\partial\Omega_I.
\end{cases}
\]

It is well known that both $\Lambda_{\Omega_I,g}$ and $D_{\Omega_I,g}$ are first-order elliptic pseudodifferential operators on $\partial\Omega_I$. Moreover, their full symbols are closely related in the following sense.
\begin{lemma}[\cite{annular}, Lemma 2.1]\label{symbol}
For any smooth function $c$ on $\partial\Omega_I$, the operators $(c\Lambda_{\Omega_I,g})^2$ and $(cD_{\Omega_I,g})^2$ have the same full symbol.
\end{lemma}
As a consequence, the full symbols of the positive operators $\Lambda_{\Omega_I,g}$ and $|D_{\Omega_I,g}|$ are identical. Their eigenvalues satisfy the asymptotic relation
\begin{equation}
\sigma_k(\Lambda_{\Omega_I,g})=\sigma_k(|D_{\Omega_I,g}|)+O(k^{-\infty}).
\end{equation}

We now introduce the \emph{spectral zeta function} associated with the Dirichlet-to-Neumann operator $\Lambda_{\Omega_I,g}$. It is defined by
\[
\zeta_{\Omega_I,g}(s):=\sum_{\sigma_k(\Lambda_{\Omega_I,g})>0}\sigma_k(\Lambda_{\Omega_I,g})^{-s}, \quad \Re s>1.
\]
By Weyl's law, the above series converges absolutely in the half-plane $\{\Re s>1\}$ and is therefore holomorphic there. Moreover, one has the following explicit formulas for the meromorphic continuation and special values of $\zeta_{\Omega_I,g}(s)$.
\begin{proposition}[\cite{annular}, Proposition 2.2]\label{prop::zetaexpress}
    The spectral zeta function $\zeta_{\Omega_I,g}(s)$ admits a meromorphic continuation to the whole complex plane, with a unique simple pole at $s=1$. Explicitly, for each $s\in\C$, we have
\begin{equation*}
\zeta_{\Omega_I,g}(s)=2\zeta_{R}(s)\sum_{i\in I}\left(\frac{2\pi}{L_g(M_i)}\right)^{-s}+\Tr (\Lambda_{\Omega_I,g}^{-s}-|D_{\Omega_I,g}|^{-s}),
\end{equation*}
where $\zeta_{R}(s)$ denotes the Riemann zeta function. Let $g_0=a^2g|_{\Omega_I},\ a\in C^{\infty}(\Omega_I,\R_+)$ be a smooth metric on $\Omega_I$ that belongs to $[g]_{\Omega_I}$. Then
\[
\Tr (\Lambda_{\Omega_I,g}^{-s}-|D_{\Omega_I,g}|^{-s})=\Tr \left(\left(a^{\frac12}\Lambda_{\Omega_I,g_0}a^{\frac12}\right)^{-s}-\left|a^{\frac12}D_{\Omega_I,g_0}a^{\frac12}\right|^{-s}\right).
\]
In particular, we have
\begin{equation}\label{generalzeta-2m}
	\zeta_{\Omega_I,g}(-2m)=\Tr\left((a\Lambda_{\Omega_I,g_0})^{2m}-(aD_{\Omega_I,g_0})^{2m} \right)
\end{equation} for each $m\in\Z_+$ and
\begin{equation}\label{generalzeta-1}
    \zeta_{\Omega_I,g}(-1)=-\frac{\pi}{3}\sum_{i\in I} \frac1{L_g(M_i)}+\Tr \left(a^{\frac12}\Lambda_{\Omega_I,g_0}a^{\frac12}-|a^{\frac12}D_{\Omega_I,g_0}a^{\frac12}|\right).
\end{equation}
\end{proposition}

The following standard consequence of the spectral theorem and Jensen's inequality will be used repeatedly in the proofs of the trace inequalities below.
\begin{lemma}\label{jensen}
Let $A$ be a self-adjoint operator on a Hilbert space $H$. Let $f$ be a convex function defined on an interval $I \subset \mathbb{R}$ containing the spectrum $\sigma(A)$. Then for every $u \in \operatorname{Dom}(A)\cap\operatorname{Dom}(f(A))$ with $\|u\|_{H}=1$, one has
\[
\langle f(A)u,u\rangle_H \geq f(\langle Au,u\rangle_H).
\]
\end{lemma}

The following lemma establishes a comparison relation between the positive operators $\Lambda_{\Omega_I,g}$ defined on different domains. This will be used to compare their spectral zeta functions.
\begin{lemma}\label{lem::positiveness}
Let $I,J\subset\{1,2,\dots,n\}$ be nonempty index sets with $J\subset I$. For every $f\in C^\infty(\partial\Omega_I)$ supported on $\partial\Omega_J$, one has
\begin{equation}\label{eq::ineq1st}
\left(\Lambda_{\Omega_I,g}f,f\right)_{L^2(\partial\Omega_I)}
\geq
\left(\Lambda_{\Omega_J,g}f,f\right)_{L^2(\partial\Omega_I)}.
\end{equation}
If $J\subsetneq I$, equality holds only for $f\equiv0$.
\end{lemma}
\begin{proof}
Let $u_I$ and $u_J$ be the harmonic extensions of $f$ to $\Omega_I$ and $\Omega_J$, respectively. Extend $u_I$ by zero across every disk $U_i$ with $i\in I\setminus J$. Because $u_I$ has zero trace on $M_i$ for such $i$, the resulting function, denoted by $\widetilde u_I$, belongs to $H^1(\Omega_J)$ and has boundary value $f$ on $\partial\Omega_J$.

By the Dirichlet principle,
\[
\int_{\Omega_J}|\nabla u_J|_g^2\,dV_g
\leq\int_{\Omega_J}|\nabla\widetilde u_I|_g^2\,dV_g=\int_{\Omega_I}|\nabla u_I|_g^2\,dV_g.
\]
Using the energy identity for the Dirichlet-to-Neumann operator gives \eqref{eq::ineq1st}.

Suppose now that $J\subsetneq I$ and equality holds. Then $\widetilde u_I$ is also an energy minimizer on $\Omega_J$, hence $\widetilde u_I=u_J$. Since $\widetilde u_I$ vanishes on the nonempty open set $\bigcup_{i\in I\setminus J}U_i$, unique continuation implies that $u_J\equiv0$. Therefore $f\equiv0$.
\end{proof}

Let $\mathcal P=\{J_k\}$ be a partition of $I$, so that
\[
I=\bigsqcup_{J_k\in\mathcal P}J_k.
\]
Every $f\in C^\infty(\partial\Omega_I)$ then decomposes uniquely as
\[
f=\sum_{J_k\in\mathcal P}f_k,
\]
where $f_k$ is supported on
\[
\partial\Omega_{J_k}=\bigcup_{i\in J_k}M_i.
\]
Define the operator
\[
\Lambda_{\mathcal P}:=\bigoplus_{J_k\in\mathcal P}\Lambda_{\Omega_{J_k},g}
\]
on $\partial\Omega_I$ by
\[
\Lambda_{\mathcal P}f:=\sum_{J_k\in\mathcal P}\Lambda_{\Omega_{J_k},g}f_k.
\]

The previous lemma immediately implies the following trace inequality.
\begin{proposition}\label{prop::inequfortrace}
For any partition $\mathcal P$ of $I$, any $c\in C^\infty(\partial\Omega_I;\R_+)$ and $s\ge1$,
\[
\Tr\left((c\Lambda_{\Omega_I,g}c)^s-(c\Lambda_{\mathcal P}c)^s\right)\geq 0.
\]
Equality holds if and only if $\mathcal{P}=\{I\}$.
\end{proposition}
\begin{proof}
Apply Lemma~\ref{symbol} with $c^2$ in place of $c$. Since
\[
(c\Lambda_{\Omega_I,g}c)^2=c^{-1}(c^2\Lambda_{\Omega_I,g})^2c,\qquad(cD_{\Omega_I,g}c)^2=c^{-1}(c^2D_{\Omega_I,g})^2c,
\]
the operators $(c\Lambda_{\Omega_I,g}c)^2$ and $(cD_{\Omega_I,g}c)^2$ have the same full symbol. The same argument applies to every block of $\Lambda_{\mathcal P}$. The pseudodifferential functional calculus for positive operators then shows that $(c\Lambda_{\Omega_I,g}c)^s$ and $(c\Lambda_{\mathcal P}c)^s$ have the same full symbol. Their difference is therefore smoothing and, in particular, trace class.

For each $J_k\in\mathcal P$, choose an orthonormal eigenbasis $\{\eta_{k,j}\}_j$ of $c\Lambda_{\Omega_{J_k},g}c$ in $L^2(\partial\Omega_{J_k})$, and extend every eigenfunction by zero to $\partial\Omega_I$. The union of these bases is an orthonormal basis of $L^2(\partial\Omega_I)$. If
\[
c\Lambda_{\Omega_{J_k},g}c\,\eta_{k,j}=\lambda_{k,j}\eta_{k,j},
\]
then Jensen's inequality and Lemma~\ref{lem::positiveness} give
\begin{align*}
\left((c\Lambda_{\Omega_I,g}c)^s\eta_{k,j},\eta_{k,j}\right)
&\geq\left(c\Lambda_{\Omega_I,g}c\,\eta_{k,j},\eta_{k,j}\right)^s\\
&\geq\left(c\Lambda_{\Omega_{J_k},g}c\,\eta_{k,j},\eta_{k,j}\right)^s\\
&=\lambda_{k,j}^s
=\left((c\Lambda_{\mathcal P}c)^s\eta_{k,j},\eta_{k,j}\right).
\end{align*}
Summing over $k$ and $j$ proves the trace inequality. If $\mathcal P\neq\{I\}$, then every block $J_k$ is a proper subset of $I$, and Lemma~\ref{lem::positiveness} makes the second inequality strict for every nonzero $c\eta_{k,j}$. Hence equality occurs precisely when $\mathcal P=\{I\}$.
\end{proof}

Combining Proposition~\ref{prop::zetaexpress} and Proposition~\ref{prop::inequfortrace} and setting $c=1$, we obtain the following inequality for spectral zeta functions.
\begin{corollary}\label{zetasum}
Let $(\mathcal{S}_{\mathtt{g}},g)$ be a closed surface, and let $\Omega_I$ be one of the domains defined above. For any partition $\mathcal{P}=\{J_k\}$ of $I$ and any $s\leq -1$, we have
    \[
\zeta_{\Omega_I,g}(s)\geq \sum_{J_k\in \mathcal{P}}\zeta_{\Omega_{J_k},g}(s).
    \]
Equality holds if and only if $\mathcal{P}=\{I\}$.
\end{corollary}

\section{Compactness theorem for genus-zero flat surfaces within a fixed conformal class}
\label{sec::compactnessfixedconclass}

In this section, we focus on the genus-zero case and prove a compactness
theorem for Steklov isospectral families within a fixed conformal class, using
the spectral zeta function associated with the Dirichlet-to-Neumann operator.

Let $(\mathcal{S}_0,g)$ be a connected closed surface equipped with the metric $g$. Recall that $\Omega\subset \mathcal{S}_0$ is a domain obtained by removing all disks $U_i$, $i=1,2,\dots,n$. When $n=2$, the surface $(\Omega,g)$ is an annular surface in the sense of \cite{annular}, and the compactness of the family of Steklov isospectral flat metrics in this setting is proved in \cite{annular}. Therefore, the main interest of this section is the case $n\geq3$. In contrast to the annular case, domains with multiple boundary components exhibit substantially greater complexity. In particular, it is not known whether the conformal class is determined by the Steklov spectrum in this setting. For this reason, throughout this section we restrict our attention to families of Steklov isospectral flat metrics within a fixed conformal class.

We continue to write
\[
M=\partial\Omega=M_1\cup M_2\cup\cdots\cup M_n.
\]
For each two-element subset $I\subset\{1,2,\dots,n\}$,
Theorem~\ref{thm::OPS} provides a flat cylinder
$\mathcal C_I:=S^1\times[0,l_I]$ with its canonical product metric $g_I$, a
conformal map $F_I:\mathcal C_I\to\Omega_I$, and a smooth function
$\kappa_I\in C^\infty(\mathcal C_I;\R)$ such that
\[
F_I^*g=e^{2\kappa_I}g_I.
\]

Let $\La_{\mathcal{C}_I}$ denote the Dirichlet-to-Neumann operator on $\partial\mathcal{C}_I$ with respect to the metric $g_I$ and let $D_{\mathcal{C}_I}$ denote the tangential derivative operator $\frac{1}{\sqrt{-1}}\d_\theta$ along $\partial\mathcal{C}_I$ with respect to the same metric. By the proof of Lemma 5.2 in \cite{annular}, for any nonzero $u\in C^\infty( \partial\mathcal{C}_I)$ supported on a single connected component of $\partial\mathcal{C}_I$, one has
\begin{equation}\label{lemma2.5}
    (\La_{\mathcal{C}_I} u,u)_{L^2(\partial\mathcal{C}_I)}> (|D_{\mathcal{C}_I}|u,u)_{L^2(\partial\mathcal{C}_I)}.
\end{equation}

Setting $c_I=e^{-\kappa_I}|_{\partial\mathcal{C}_I}$, the conformal covariance of the Dirichlet-to-Neumann operator gives
\begin{equation}\label{eq::ladrelation}
    \La_{\Omega_I,g}=(F_I^{-1})^*\circ( c_I\La_{\mathcal{C}_I})\circ F_I^*
\end{equation}
while the tangential derivative operator satisfies
\begin{equation}\label{eq::ladrelation2}
    D_{\Omega_I,g}=(F_I^{-1})^*\circ( c_I D_{\mathcal{C}_I})\circ F_I^*.
\end{equation}
These identities imply the following comparison.
\begin{lemma}\label{lem::positiveness2}
Suppose that $I\subset \{1,2,\dots ,n\}$ and $\# I=2$. Then for any nonzero function $f\in C^\infty(M)$ supported on $M_i$, $i\in I$, the following inequality holds:
\begin{equation}
    \left(\La_{\Omega_I,g} f, f\right)_{L^2(M)} > |\left(D_{\Omega_I,g} f, f\right)_{L^2(M)}|.
\end{equation}
\end{lemma}
\begin{proof}
By \eqref{eq::ladrelation} we have
\begin{align*}
    (\La_{\Omega_I,g}f,f)_{L^2(M,g)}&=\left((F_I^{-1})^*c_I\La_{\mathcal{C}_I} F_I^*f,f\right)_{L^2(\partial\Omega_I,g)}\\
    &=\left(c_I\La_{\mathcal{C}_I} F_I^*f,F_I^*f\right)_{L^2(\partial\mathcal{C}_I,c_I^{-2}g_I)}\\
    &=(\La_{\mathcal{C}_I} F_I^*f,F_I^*f)_{L^2(\partial\mathcal{C}_I,g_I)}.
\end{align*}
Similarly, by \eqref{eq::ladrelation2} we have
\begin{align*}
	\left|(D_{\Omega_I,g}f,f)_{L^2(M,g)}\right|
	=|(D_{\mathcal{C}_I}F^*_If,F^*_If)_{L^2(\partial\mathcal{C}_I,g_I)}|.
\end{align*}
It follows from \eqref{lemma2.5} and Lemma~\ref{jensen} that
\begin{align*}
  (\La_{\mathcal{C}_I} F_I^*f,F_I^*f)_{L^2(\partial\mathcal{C}_I,g_I)}>
    (|D_{\mathcal{C}_I}|F^*_If,F^*_If)_{L^2(\partial\mathcal{C}_I,g_I)}
    \geq\left|(D_{\mathcal{C}_I}F^*_If,F^*_If)_{L^2(\partial\mathcal{C}_I,g_I)}\right|.
\end{align*}
This completes the proof.
\end{proof}
The next lemma records the corresponding simply connected case. In contrast
with the annular case above, one only obtains a non-strict inequality.
\begin{lemma}\label{positivenesslemmsimply}
    Suppose that $I\subset \{1,2,\dots ,n\}$ and $\# I=1$. Then for any function $f\in C^\infty(M)$ supported on $\partial\Omega_I$, the following inequality holds:
\begin{equation*}
    \left(\La_{\Omega_I,g} f, f\right)_{L^2(M)} \geq |\left(D_{\Omega_I,g} f, f\right)_{L^2(M)}|.
\end{equation*}
\end{lemma}
\begin{proof}
Since the domain $\Omega_I$ is simply connected, it is conformal to the unit disk $\D$. Let $\Lambda_\D$ and $D_\D$ denote the Dirichlet-to-Neumann operator and the tangential derivative operator on $\partial\D$, respectively, with respect to the canonical Euclidean metric. Arguing similarly as in the proof of Lemma~\ref{lem::positiveness2}, it suffices to show that for any smooth function $f'\in C^\infty(\partial\D)$,
    \[
    \left(\La_\D f',f'\right)_{L^2(\partial\D)}\geq\left|\left(D_\D f',f'\right)_{L^2(\partial\D)}\right|.
    \]
The identity $\Lambda_\D=|D_\D|$, followed by Lemma~\ref{jensen}, proves the claim.
\end{proof}
The preceding operator inequalities imply the following trace estimate.
\begin{proposition}\label{proppositivenessforzetafunction}
    Suppose that $I\subset \{1,2,\dots ,n\}$ and $\# I\leq2$. Then for any positive smooth function $c\in C^\infty(M)$ and $s\geq 1$, we have
    \[
    \Tr \left((c\La_{\Omega_I,g}c)^s-|cD_{\Omega_I,g}c|^s\right)\geq 0.
    \]
\end{proposition}
\begin{proof}
    Let $\{\eta_{I,k}\}$ be the normalized eigenfunctions of the operator $cD_{\Omega_I,g}c$, chosen so that each $\eta_{I,k}$ is supported on a single boundary component of $\Omega_I$. These functions form an orthonormal basis for $L^2(\partial\Omega_I)$. By Lemma~\ref{jensen}, Lemma~\ref{lem::positiveness2} and Lemma~\ref{positivenesslemmsimply}, we obtain
    \begin{align*}
        \left((c\La_{\Omega_I,g}c)^s\eta_{I,k},\eta_{I,k}\right)_{L^2(\partial\Omega_I)}&\geq \left(\La_{\Omega_I,g}c\eta_{I,k},c\eta_{I,k}\right)^s_{L^2(\partial\Omega_I)}\\
        & \geq |\left(D_{\Omega_I,g}c\eta_{I,k},c\eta_{I,k}\right)|^s_{L^2(\partial\Omega_I)}\\
        & =\left(|cD_{\Omega_I,g}c|^s\eta_{I,k},\eta_{I,k}\right)_{L^2(\partial\Omega_I)},
    \end{align*}
which completes the proof.
\end{proof}

As a consequence, we obtain the following monotonicity property for negative zeta values.
\begin{corollary}\label{cor::zetamonotonicity}
Let $\varnothing\neq J\subset I\subset \{1,2,\dots,n\}$. 
Then, for every $s\leq -1$,
\[
\zeta_{\Omega_I,g}(s)\geq \zeta_{\Omega_J,g}(s)+2\zeta_R(s)\sum_{i\in I\setminus J}
\left(\frac{2\pi}{L_g(M_i)}\right)^{-s}.
\]
\end{corollary}

\begin{proof}
If $I=J$, the assertion is immediate. We therefore assume that $I\setminus J$ is nonempty. By Proposition~\ref{prop::zetaexpress}, we have
\[
\zeta_{\Omega_I,g}(s)=2\zeta_R(s)\sum_{i\in I}
\left(\frac{2\pi}{L_g(M_i)}\right)^{-s}+\Tr\bigl(\Lambda_{\Omega_I,g}^{-s}-|D_{\Omega_I,g}|^{-s}\bigr).
\]
Hence it suffices to prove that
\[
\Tr\bigl(\Lambda_{\Omega_I,g}^{-s}-|D_{\Omega_I,g}|^{-s}\bigr)\geq\Tr\bigl(\Lambda_{\Omega_J,g}^{-s}-|D_{\Omega_J,g}|^{-s}\bigr).
\]

To this end, choose a partition
\[
\mathcal P_{IJ}=\{J,J_1,\dots,J_k\}
\]
of $I$, where $J_1,\dots,J_k$ form a partition of $I\setminus J$ and each $J_\alpha$ satisfies $\# J_\alpha\leq 2$. By Proposition~\ref{prop::inequfortrace} and Proposition~\ref{proppositivenessforzetafunction}, for $s\leq -1$ we obtain
\begin{align*}
\Tr\bigl(\Lambda_{\Omega_I,g}^{-s}-|D_{\Omega_I,g}|^{-s}\bigr)
&=\Tr\bigl(\Lambda_{\Omega_I,g}^{-s}-\Lambda_{\mathcal P_{IJ}}^{-s}\bigr)
+\Tr\bigl(\Lambda_{\mathcal P_{IJ}}^{-s}-|D_{\Omega_I,g}|^{-s}\bigr) \\
&\geq\Tr\bigl(\Lambda_{\Omega_J,g}^{-s}-|D_{\Omega_J,g}|^{-s}\bigr)+\sum_{i=1}^k\Tr\bigl(\Lambda_{\Omega_{J_i},g}^{-s}-|D_{\Omega_{J_i},g}|^{-s}\bigr)\\
&\geq\Tr\bigl(\Lambda_{\Omega_J,g}^{-s}-|D_{\Omega_J,g}|^{-s}\bigr).
\end{align*}
This proves the corollary.
\end{proof}

We also require the following estimates on flat cylinders, which will play a crucial role in controlling the conformal factor and proving the compactness theorem.
\begin{proposition}[\cite{annular}, Proposition 5.3]\label{prop::ineq1}
    Let $\mathcal C_l=S^1\times [0,l]$ be a flat cylinder with the canonical product metric. Then there exist constants $d_{m,l}, d_m'>0$ such that for any positive function $a\in C^\infty (\partial\mathcal C_l;\R_+)$ and $m\in \Z_+$, we have
    \[
    \|a^m\|^2_{H^{m+{\frac12}}(\partial\mathcal C_l)}\leq d_{m,l}\left(\Tr \left((a\Lambda_{\mathcal C_l})^2-(aD_{\mathcal C_l})^2\right)\right)^{m}+d_m' \Tr\left((a\Lambda_{\mathcal C_l})^{2m}-(aD_{\mathcal C_l})^{2m}\right).
    \]
\end{proposition}
\begin{proposition}[\cite{annular}, Proposition 5.4]\label{prop::ineq2}
 Let $\mathcal C_l=S^1\times [0,l]$ be a flat cylinder with the canonical product metric. For any positive function $a\in C^\infty (\partial\mathcal C_l;\R_+)$, we have
\[
a\geq K_l\left(\int_{\partial\mathcal C_l} a^{-1}\,d\theta,\,\Tr (a^{\frac12}\La_{\mathcal C_l} a^{\frac12}-|a^{\frac12}D_{\mathcal C_l} a^{\frac12}|),\, \Tr \left((a\Lambda_{\mathcal C_l})^2-(aD_{\mathcal C_l})^2\right)\right).
\]
Here $K_l$ is monotonically decreasing in each variable
$x,y,z$ and is given by
\[K_l(x,y,z):=\frac{2\pi}{x}\exp\left[-\sqrt{2\pi x}\left(6y+\frac{2}{\sqrt{B_l(0)}}\sqrt{z}\right)^{\frac12}\right],\]
where
\[
 B_l(0):=\sum_{k\in\Z\setminus\{0\} }\left(k^2\tanh^2\frac{kl}{2}+k^2\coth^2\frac{kl}{2}-2k^2\right)+\frac{4}{l^2}.
\]
\end{proposition}

With these preparations in place, we are now ready to prove the compactness theorem.
\begin{theorem}\label{thm::maincompactnessflat}
Let $(\Omega,g)$ be a fixed surface of genus zero with $n\geq 3$ smooth boundary components. Then any Steklov isospectral family of flat Riemannian surfaces which are conformal to $(\Omega,g)$ is compact in the $C^\infty$ topology.
\end{theorem}
\begin{proof}
We regard $(\Omega,g)$ as a domain in a fixed closed genus-zero surface
$(\mathcal S_0,g)$. Denote by $(\mathcal S_0,\Omega,g)$ the embedding data, namely the ambient surface
$(\mathcal S_0,g)$ together with the embedding $\Omega\hookrightarrow
\mathcal S_0$. Let $\mathcal H_g$ be a Steklov isospectral family of flat
metrics in the conformal class of $[g]_\Omega$. Thus every $g'\in\mathcal H_g$
can be written as
\[
g'=e^{2\phi}g|_\Omega
\]
for some $\phi\in C^\infty(\Omega;\mathbb R)$.

It is enough to prove the compactness of $\mathcal H_g$. Equivalently, it suffices to prove uniform compactness for the corresponding conformal factors. More precisely, we first obtain uniform Sobolev bounds for their boundary traces, and then use elliptic regularity to pass these bounds to the interior. The main point is that the trace terms appearing in the zeta identities can be decomposed into nonnegative contributions associated with subsurfaces of $\mathcal S_0$ having at most two boundary components. This allows us to apply the estimates from the previous subsection to each such contribution.

Set $a:=e^{-\phi}|_M$, which is a positive smooth function on $M$. Then
\[
\Lambda_{\Omega,g'}=a\Lambda_{\Omega,g},\qquad D_{\Omega,g'}=aD_{\Omega,g}.
\]

Fix an index set $I\subset\{1,\dots,n\}$ with $\# I=2$, and let $\mathcal P_{[n]}:=\{I, J_1, \cdots, J_{n-2}\}$ be the partition of $\{1,\dots,n\}$ consisting of $I$ together with the remaining singleton sets. By Proposition~\ref{prop::zetaexpress}, for any $s\geq 1$,
\begin{equation}\label{spectralinvzeta}
    \zeta_{\Omega,g'}(-s)-2\zeta_{R}(-s)\sum_{i=1}^n\left(\frac{2\pi}{L_{g'}(M_i)}\right)^{s}=\Tr\left(\left(a^{\frac12}\La_{\Omega,g}a^{\frac12}\right)^{s}-\left|a^{\frac12}D_{\Omega,g}a^{\frac12}\right|^{s} \right).
\end{equation}
We may then decompose the trace term as
\begin{align*}
    \Tr\left(\left(a^{\frac12}\La_{\Omega,g}a^{\frac12}\right)^{s}-\left|a^{\frac12}D_{\Omega,g}a^{\frac12}\right|^{s} \right)= &\Tr\left(\left(a^{\frac12}\La_{\Omega,g}a^{\frac12}\right)^{s}-\left(a^{\frac12}\La_{\mathcal P_{[n]}}a^{\frac12}\right)^{s}\right)\notag\\
    &+\Tr\left(\left(a^{\frac12}\La_{\mathcal P_{[n]}}a^{\frac12}\right)^{s}-\left|a^{\frac12}D_{\Omega,g}a^{\frac12}\right|^{s} \right).
\end{align*}
By Proposition~\ref{prop::inequfortrace}, the first term on the right-hand side is strictly positive,
\[
\Tr\left(\left(a^{\frac12}\La_{\Omega,g}a^{\frac12}\right)^{s}-\left(a^{\frac12}\La_{\mathcal P_{[n]}}a^{\frac12}\right)^{s}\right)>0.
\]
Consequently, the second term is bounded above by the spectral quantity appearing on the left-hand side of \eqref{spectralinvzeta}. By the definition of $\Lambda_{\mathcal P_{[n]}}$, the second term further decomposes according to the partition.
\begin{align}
\Tr\left(\left(a^{\frac12}\La_{\mathcal P_{[n]}}a^{\frac12}\right)^{s}-\left|a^{\frac12}D_{\Omega,g}a^{\frac12}\right|^{s} \right)=&
    \sum_{k=1}^{n-2}\Tr\left(\left(a^{\frac12}\La_{\Omega_{J_k},g}a^{\frac12}\right)^{s}-\left|a^{\frac12}D_{\Omega_{J_k},g}a^{\frac12}\right|^{s} \right)\notag\\
    &+\Tr\left(\left(a^{\frac12}\La_{\Omega_{I},g}a^{\frac12}\right)^{s}-\left|a^{\frac12}D_{\Omega_{I},g}a^{\frac12}\right|^{s} \right).\notag
\end{align}
Each of these terms is nonnegative by Proposition~\ref{proppositivenessforzetafunction}. Therefore each individual trace is controlled by the same spectral quantity. In particular, we obtain the estimate
\begin{equation}\label{nodd}
    \Tr\left(\left(a^{\frac12}\La_{\Omega_{I},g}a^{\frac12}\right)^{s}-\left|a^{\frac12}D_{\Omega_{I},g}a^{\frac12}\right|^{s} \right)\leq \zeta_{\Omega,g'}(-s)-2\zeta_{R}(-s)\sum_{i=1}^n\left(\frac{2\pi}{L_{g'}(M_i)}\right)^{s},
\end{equation}
for every two-element index set $I$.

We first apply this estimate with $s=2m$, where $m\in\mathbb Z_+$.
For each such $I$, the surface $(\Omega_I,g)$ is conformal to the flat
cylinder $\mathcal C_I:=S^1\times[0,l_I]$. Because the embedding data
$(\mathcal S_0,\Omega,g)$ are fixed, so are the conformal map
$F_I:\mathcal C_I\to\Omega_I$, the cylinder height $l_I$, and the positive
factor $c_I$ appearing in \eqref{eq::ladrelation} and
\eqref{eq::ladrelation2}.
Therefore, the corresponding trace term may be expressed in cylinder coordinates as
\begin{align}\label{eq::zeta-2formul}
    \Tr \left((a\La_{\Omega_{I},g})^{2m}-(aD_{\Omega_{I},g})^{2m} \right)
    =\Tr\left(\left((F_{I}^*a)c_{I}\La_{\mathcal C_I}\right)^{2m}-\left((F_{I}^*a)c_{I}D_{\mathcal C_I}\right)^{2m} \right).
\end{align}
It follows from Proposition~\ref{prop::ineq1} that, for each $m\in \Z_+$,
\begin{align*}
    \|(F_{I}^*a)^mc_{I}^m\|^2_{H^{m+\frac12}(\partial\mathcal C_I)}\leq &d_{m,l_{I}}\left(\Tr \left((a\La_{\Omega_{I},g})^{2}-(aD_{\Omega_{I},g})^{2} \right)\right)^{m}\\
    &+d'_m\Tr \left((a\La_{\Omega_{I},g})^{2m}-(aD_{\Omega_{I},g})^{2m} \right).
\end{align*}
Moreover, for $m\geq1$, the space $H^{m+\frac12}(\partial\mathcal C_I)$ is a Banach algebra (see, for example, \cite[Theorem 1.4.4.2]{Grisvard1985}). The following Sobolev multiplication estimate holds for some positive constant $R_m$:
\begin{equation*}
    \|(F_{I}^*a)^m\|_{H^{m+\frac12}(\partial\mathcal C_I)}\leq R_m \|(F_{I}^*a)^mc_{I}^m\|_{H^{m+\frac12}(\partial\mathcal C_I)}\|c_{I}^{-m}\|_{H^{m+\frac12}(\partial\mathcal C_I)}.
\end{equation*}
Thus, for each $m\in\Z_+$,
\begin{equation}\label{eq::con1}
    \|(F_{I}^*a)^m\|^2_{H^{m+\frac12}(\partial\mathcal C_I)}\leq R_m^2\|c_{I}^{-m}\|^2_{H^{m+\frac12}(\partial\mathcal C_I)}\left(d_{m,l_{I}}\left(\zeta_{\Omega,g'}(-2)\right)^m+d'_m\zeta_{\Omega,g'}(-2m)\right).
\end{equation}
In particular, taking $m=1$ and using the embedding $H^{\frac{3}{2}}(\partial\mathcal C_I) \subset C^0(\partial\mathcal C_I)$, we obtain a uniform upper bound for $a$ on $\partial\Omega_I$, namely, there exists some positive constant $R_{I}'$ such that
\begin{equation*}
    a|_{\partial\Omega_{I}}\leq R'_{I} \|c_{I}^{-1}\|_{H^{\frac32}(\partial\mathcal C_I)}\left(\zeta_{\Omega,g'}(-2)\right)^{\frac12}.
\end{equation*}
Letting $I$ range over all pairs of boundary components, these estimates cover all of $M$, and hence give the global upper bound
\begin{equation}\label{upperboundforamul}
    a\leq \max _{\# I=2} R'_{I} \|c_{I}^{-1}\|_{H^{\frac32}(\partial\mathcal C_I)}\left(\zeta_{\Omega,g'}(-2)\right)^{\frac12}.
\end{equation}

We next derive a uniform positive lower bound for $a$. Taking $s=1$ in \eqref{nodd}, we use the spectral quantities $L_{g'}(M)$, $\zeta_{\Omega,g'}(-1)$, and $\zeta_{\Omega,g'}(-2)$.
The total length of the boundary of $\Omega$ under the metric $g'$ is given by
\[
L_{g'}(M)=\int_{M}a^{-1}\,d\tau,
\]
where $\tau$ is the length parameter under the metric $g$. We have, for each two-element index set $I$,
\[
\int_{\partial\Omega_I}a^{-1}\,d\tau=\int_{\partial\mathcal{C}_I}F_I^*a^{-1} c_I^{-1}\, d\theta\leq L_{g'}(M).
\]
We write the corresponding trace term in \eqref{nodd} as
\begin{align*}
\Tr&\left(a^{\frac12}\La_{\Omega_{I},g}a^{\frac12}-\left|a^{\frac12}D_{\Omega_{I},g}a^{\frac12}\right| \right)\\
 &=\Tr \left((F^*_{I}a^{\frac12})c_{I}^{\frac12}\La_{\mathcal C_I}c_{I}^{\frac12}(F^*_{I}a^{\frac12})-\left|(F^*_{I}a^{\frac12})c_{I}^{\frac12}D_{\mathcal C_I}c_{I}^{\frac12}(F^*_{I}a^{\frac12})\right|\right).
\end{align*}
Therefore Proposition~\ref{prop::ineq2}, together with the estimate \eqref{nodd}, yields a pointwise lower bound for $a$,
\begin{align*}
    F_{I}^*(a)c_{I}&\geq K_{l_I}\Bigg(\int_{\partial \mathcal{C}_I} (F_{I}^*(a)c_{I})^{-1}\, d\theta,\Tr\left(a^{\frac12}\La_{\Omega_{I},g}a^{\frac12}-\left|a^{\frac12}D_{\Omega_{I},g}a^{\frac12}\right| \right),\\
    &\phantom{\geq K_{l_I}}\Tr \left((a\La_{\Omega_{I},g})^{2}-(aD_{\Omega_{I},g})^{2} \right)\Bigg)\\
    &\geq K_{l_I}\left(L
    _{g'}(M),\zeta_{\Omega,g'}(-1)+\frac{\pi}{3}\sum_{i=1}^n \frac1{L_{g'}(M_i)},\zeta_{\Omega,g'}(-2)\right).
\end{align*}
Hence,
\begin{equation}\label{eq::zeta-1ieq}
\begin{split}
a\geq \min_{\# I=2}\Bigg\{ 
\left(\inf_{\partial\mathcal C_I}c_I^{-1}\right)
K_{l_I}\Bigg(
L_{g'}(M),
\zeta_{\Omega,g'}(-1)+\frac{\pi}{3}\sum_{i=1}^n\frac1{L_{g'}(M_i)}, \zeta_{\Omega,g'}(-2)
\Bigg)\Bigg\}>0.
\end{split}
\end{equation}

Combining the upper bound \eqref{upperboundforamul} and the lower bound \eqref{eq::zeta-1ieq}, we may use the equivalence of the Sobolev weights $1+\|\cdot\|_{H^m}$ and $1+\|\log(\cdot)\|_{H^m}$ for positive smooth functions bounded above and below away from zero. We deduce that there exist positive constants $C_I$ and $C_m$, independent of $g'$, such that
\begin{align}
    \|\phi\|_{H^m(M)} =\|\log a\|_{H^m(M)} 
    &\leq\sum_{\# I=2}C_I\|\log F_I^*a\|_{H^m(\partial\mathcal{C}_I)}\notag\\
  &= \frac{1}{m} \sum_{\# I=2}C_I\|\log (F_I^*a)^m\|_{H^m(\partial\mathcal{C}_I)}\notag\\
   &\leq C_m \sum_{\# I=2}C_I\left(1+\| (F_I^*a)^m\|_{H^m(\partial\mathcal{C}_I)}\right).\label{bdphiest}
\end{align}
The right-hand side is uniformly bounded for each $m\in\mathbb Z_+$. Indeed, combining \eqref{bdphiest} with \eqref{eq::con1}, we find constants $\Theta_{(\mathcal{S}_0,\Omega,g),m}$ depending only on the fixed background data $(\mathcal{S}_0,\Omega,g)$ and on $m$, such that
\begin{equation}\label{bdfinalestimate}
    \|\phi\|_{H^m(M)}\leq \Theta_{(\mathcal{S}_0,\Omega,g),m}(1+(\zeta_{\Omega,g'}(-2))^m+\zeta_{\Omega,g'}(-2m)).
\end{equation}

Finally, since $g'$ is flat, the conformal factor $\phi$ satisfies
\[
-\Delta_g\phi=-K_g \quad \text{in }\Omega,
\]
where $K_g$ denotes the Gaussian curvature of $(\Omega,g)$.
Therefore, standard elliptic regularity estimates imply that for every $m\geq 2$, there exists a positive constant $d_{(\Omega,g),m}$, depending only on $(\Omega,g)$ and $m$ such that
\begin{equation}\label{finalestimate}
    \|\phi\|_{H^m(\Omega)}\leq d_{(\Omega,g),m}\left(\|K_g\|_{H^{m-2}(\Omega)}+\|\phi\|_{H^{m}(M)}\right).
\end{equation}
The curvature term $\|K_g\|_{H^{m-2}(\Omega)}$ is fixed, and the boundary term has already been uniformly bounded.
Consequently, for every $m\in\mathbb Z_+$, the Sobolev norm $\|\phi\|_{H^m(\Omega)}$ is uniformly bounded by quantities depending only on $(\mathcal S_0,\Omega,g)$ and the Steklov spectrum.
A diagonal application of the Rellich compactness theorem, followed by Sobolev embedding, yields a subsequence converging in $C^\infty$. This proves the theorem.
\end{proof}
\begin{remark}
    In the proof of Theorem~\ref{thm::maincompactnessflat}, we show that for a genus-zero surface $(\Omega,g)$, the trace
    \[
    \Tr\left(\left(a^{\frac12}\La_{\Omega,g}a^{\frac12}\right)^{s}-\left|a^{\frac12}D_{\Omega,g}a^{\frac12}\right|^{s} \right)
    \]
    is non-negative for all $s\geq 1$. Combined with Proposition~\ref{prop::zetaexpress}, this provides a lower bound for the spectral zeta function
    \[
    \zeta_{\Omega,g}(-s)\geq 2\zeta_{R}(-s)\sum_{i=1}^n\left(\frac{2\pi}{L_{g}(M_i)}\right)^{s},
    \]
    for all $s\geq1$. This extends the corresponding result of Jollivet--Sharafutdinov \cite{Sharafutdinov18} to genus-zero surfaces.
\end{remark}

\section{Asymptotic formula for the Steklov determinant on hyperbolic surfaces}\label{sectiondegenration}
In this section, we consider compact hyperbolic surfaces with geodesic boundary. We study the asymptotic behavior of the determinant associated with the Steklov spectrum under degeneration of the surface. This analysis will play a crucial role in our later study of the Steklov isospectral compactness problem for surfaces with negative Euler characteristic without fixing the conformal class.

\subsection{Hyperbolic setup and determinant formula}

Let $\Omega$ be a compact, connected, orientable surface of genus $\mathtt{g}$ whose boundary $M=M_1\cup M_2\cup\cdots\cup M_n$ has $n$ connected components. Assume that its Euler characteristic
\[
\chi(\Omega)=2-2\mathtt g-n
\]
is negative. Let $M_H(\Omega)$ denote the space of metrics $\hat g$ on $\Omega$ such
that $(\Omega,\hat g)$ is a compact hyperbolic surface with geodesic
boundary, endowed with the $C^\infty$ topology. Define the moduli space to be
\begin{equation*}
\mathcal M(\Omega):=
M_H(\Omega)/\operatorname{Diff}(\Omega),
\end{equation*}
where $\operatorname{Diff}(\Omega)$ acts by pullback, and equip
$\mathcal M(\Omega)$ with the quotient topology. Equivalently, this is
the usual topology described by Fenchel--Nielsen coordinates. Its real
dimension is
$
6\mathtt g-6+3n.
$

For any $\hat g\in M_H(\Omega)$, choose a discrete subgroup
$G_0\subset\operatorname{Isom}(\H^2)$ such that $(\Omega,\hat g)$ is
isometric to the hyperbolic surface $X:=G_0\backslash\H^2$. Denote by
\[
\partial X=\Gamma_1\cup \Gamma_2\cup\cdots\cup \Gamma_n,\quad b_i:=\ell(\Gamma_i)=L_{\hat{g}}(M_i),\quad i=1,2,\dots,n,
\]
the lengths of boundary components of $X$, and let
\[
\ell(\partial X):=\sum_{i=1}^n b_i
\]
be the total length of $\partial X$.
Let $G=G_0\cap \operatorname{Isom}^+(\H^2)$ be the index-$2$ subgroup of orientation-preserving isometries of $G_0$. Then the double of $X$ is the closed hyperbolic surface $\widetilde{X}:=G\backslash\H^2$. We say that the hyperbolic surface $X$ degenerates if its double $\widetilde X$ degenerates. By Mumford's compactness theorem
\cite{Mumford1971} (see also \cite{Buser_geometry_2010}), this is equivalent to the existence
of a pinching closed geodesic on $\widetilde X$ whose length tends to zero.
 After passing to the quotient by the reflection involution, such degeneration can be described in terms of degenerating geodesic data on $X$, such as shrinking interior closed geodesics or shrinking geodesic boundary components.
 
We now fix our notation for \emph{(zeta-regularized) determinants}.
Let $A$ be a nonnegative self-adjoint elliptic operator with discrete
spectrum. Its spectral zeta function is defined by
\[
\zeta_A(s):=\sum_{\lambda\in\operatorname{Spec}(A)\setminus\{0\}}
\lambda^{-s},
\]
where eigenvalues are counted with multiplicity, initially for
$\operatorname{Re}s$ sufficiently large and then by meromorphic
continuation. If $\zeta_A$ is regular at $s=0$, the zeta-regularized
determinant of $A$ is defined to be 
\[
{\det}' A:=\exp\bigl(-\zeta_A'(0)\bigr).
\]
The prime indicates that the zero eigenvalues are omitted. If $A$ is
invertible, we simply write
\[
\det A:=\exp\bigl(-\zeta_A'(0)\bigr).
\]

Let $\Lambda_X$ denote the Dirichlet-to-Neumann operator associated with
$X$. A determinant formula of Guillarmou--Guillop\'e
\cite[Theorem 1.3]{8180465} expresses the normalized zeta-regularized
determinant of $\Lambda_X$ directly in terms of the Selberg zeta functions. This formula allows us to
study the degeneration of $\det'\Lambda_X$ through the corresponding
closed-geodesic data. More precisely,
\begin{equation}\label{eq::3.31}
	\frac{{\det}'\Lambda_{X}}{{\ell}(\partial X)}=-\frac{Z_G'(1)}{(Z_{G_0}(1))^2}\frac{e^{\frac{\ell(\partial X)}{4}}}{2\pi\chi(X)},
\end{equation}
where $Z_G$ and $Z_{G_0}$ are the Selberg zeta functions associated with $\widetilde{X}$ and $X$, respectively. Their definitions are as follows. For the closed hyperbolic surface $\widetilde{X}$, we recall that there exists a bijection between the set of primitive conjugacy classes of the Fuchsian group $G$ and the set $[\Gamma]$ of primitive closed oriented geodesics $\widetilde\gamma$ (whose length is denoted by $r_{\widetilde\gamma}$). The dynamical Ruelle zeta function of $\widetilde{X}$ is defined by the formula
\[
R_G(s):=\prod_{\widetilde\gamma\in[\Gamma]}\left(1-e^{-sr_{\widetilde\gamma}}\right),
\]
where we adopt Ruelle’s original convention (see \cite{Ruelle1976GeneralizedZeta,Ruelle1976ZetaFunctions}), under which the Ruelle zeta function is the inverse of that frequently used in the modern literature.
The Selberg zeta function of $\widetilde{X}$ is given by
\[
Z_G(s):=\prod_{k\in\N_0}R_G(s+k).
\]
Guillop\'e \cite{10.1215/S0012-7094-86-05345-7} extended it to a hyperbolic surface $X$ with geodesic boundary as follows.\footnote{Following \cite{8180465}, the zeta functions used here are the
square roots of the corresponding zeta functions originally defined by
Guillop\'e. In the definition of $R_{\partial X}$ in
\cite{8180465}, the factor $e^{-(s+1)b_i}$ is misprinted with $s$
in place of $s+1$.}
\begin{equation}\label{SelbZetaFunc}
	Z_{G_0}(s):=\prod_{k\in \N_0} R_{\partial X}(s+2k)R_{G_0}(s+2k),
\end{equation}
where
\[ R_{\partial X}(s):=\prod_{i=1}^{n}\left(1-e^{-(s+1)b_i}\right)^2\]
counts the boundary geodesic contribution, and
\[ R_{G_0}(s):=\prod_{\gamma\in[C]} \left(1-(-1)^{n_\gamma}e^{-s r_\gamma}\right)\left(1-e^{-(s+1)r_\gamma}\right). \]
Here, $[C]$ denotes the set of primitive oriented closed geodesics $\gamma$ on $X$ that are not contained in $\partial X$, and for each $\gamma \in [C]$ we let $r_\gamma$ denote its length, and $n_\gamma\ge 0$ the number of geometric reflections of $\gamma$ on the boundary $\partial X$.
These zeta functions converge absolutely in the half-plane $\{\Re s>1\}$ and admit meromorphic continuations to $\C$.

\subsection{Short geodesics and collar geometry}

By the collar lemma, we may choose a universal constant
\[
0<r^*<\frac{\pi}{4}
\]
with the following property for every closed hyperbolic surface $\widetilde X$. There are at most $3\widetilde{\mathtt g}-3$ primitive closed unoriented geodesics on $\widetilde X$ of length less than $r^*$. For each such geodesic $\widetilde\gamma$, the standard collar
\[
C(\widetilde\gamma):=\left\{
x\in\widetilde X:d_{\widetilde X}(x,\widetilde\gamma)<
w(\widetilde\gamma)
\right\}
\]
is embedded, where
\[
w(\widetilde\gamma):=\operatorname{arsinh}
\left(
\frac{1}{\sinh(r_{\widetilde\gamma}/2)}
\right)
=\log\frac{4}{r_{\widetilde\gamma}}+
O(r_{\widetilde\gamma}^2)
\qquad
\text{as }r_{\widetilde\gamma}\to0.
\]
Moreover, the standard collars associated with such geodesics are pairwise disjoint.

Throughout the paper, we use $r^*$ as a smallness threshold and,
whenever necessary, decrease it without changing the notation.
After such reductions, $r^*$ is allowed to depend at most on
$n$ and $\mathtt g$.

Inside $C(\widetilde\gamma)$, we use conformal collar coordinates $(u,v)$ in which
\[
u\in\mathbb R/r_{\widetilde\gamma}\mathbb Z
\]
and the hyperbolic metric takes the form
\[
\frac{du^2+dv^2}{\sin^2 v}.
\]
We then define the standard subcollar
\[
SC(\widetilde\gamma):=
\left\{
(u,v)\in C(\widetilde\gamma):
2r_{\widetilde\gamma}<v<\pi-2r_{\widetilde\gamma}
\right\}.
\]

The following standard consequence of Wolpert's thick-part analysis will be used repeatedly; see \cite[Sections~2.6 and~2.7]{Wolpert1987}, \cite[Section~3]{Kim2008}.
\begin{lemma}[\cite{Wolpert1987}]\label{lem::thickgeometry}
For fixed topology, the complements of the standard subcollars form a precompact family in the $C^\infty$ topology, modulo diffeomorphisms. In particular, they have uniformly bounded geometry.
\end{lemma}

\begin{definition}\label{def::shortgeodesics}
A primitive closed geodesic on $\widetilde X$ is called \emph{short}
if its length is less than $r^*$. A primitive closed geodesic on $X$
is called \emph{short} if its corresponding closed geodesic on
$\widetilde X$ is short. A boundary component $\Gamma_i$ is called
\emph{short} if $b_i<r^*$.

For the short oriented geodesics in $[C]$, set
\begin{enumerate}
    \item $\mathscr C_1:=\{\gamma:\gamma \text{ is short and }\gamma\cap\partial X=\varnothing\}$;
    \item $\mathscr C_2:=\{\gamma:\gamma \text{ is short and meets }\partial X\text{ transversely}\}$.
\end{enumerate}
\end{definition}

If $\gamma \in \mathscr{C}_1$, then the geometric reflections $n_\gamma=0$. If $\gamma\in \mathscr{C}_2$, then by the collar lemma we have $n_\gamma=2$, and $\gamma$ intersects the boundary $\partial X$ orthogonally at the two reflection points. In this case, the length of $\gamma$ is twice the length of the corresponding orthogeodesic.

It follows from the above discussion that any oriented primitive closed geodesic in $[C]$ with $n_\gamma\neq 0,2$ must have length greater than $r^*$, and therefore cannot be short.

For later use, we distinguish oriented and unoriented representatives. We introduce an equivalence relation on each $\mathscr C_j$, $j=1,2$, by declaring $\gamma_1\sim \gamma_2$ if and only if $\gamma_1=\gamma_2^{-1}$. The quotient sets
\[
\mathscr{C}_j/\sim, \quad j=1,2
\]
therefore represent the corresponding sets of short unoriented closed geodesics in $[C]$.
For geodesics in $\mathscr{C}_1$, the two orientations give two distinct elements. Hence $\mathscr{C}_1$ contains twice as many elements as $\mathscr{C}_1/\sim$. In contrast, $\mathscr{C}_2$ and $\mathscr{C}_2/\sim$ are naturally identified.

It is classical that the number of primitive closed geodesics on a compact hyperbolic surface grows at most exponentially with respect to the length.We shall use the following form of this estimate. For a closed hyperbolic surface Y, Buser's estimates \cite[Theorems 4.1.6 and 6.6.4]{Buser_geometry_2010}, as summarized in \cite[Corollary 15]{WuXue}, give
\begin{equation}\label{eq::closedgeoestimate}
    \pi_Y(r)\leq (\mathtt{g}_Y-1) e^{r+7},
\end{equation}
where $\pi_Y(r)$ denotes the number of primitive closed geodesics on $Y$ of length less than $r$, and $\mathtt{g}_Y$ denotes the genus of $Y$.
 We emphasize that this is a uniform topological estimate, rather than a consequence of the prime geodesic theorem, whose error terms are not uniform on the moduli space.

For the bordered case, the doubling argument separates the reflected closed geodesics according to the parity of $n_\gamma$. If $n_\gamma$ is even, then $\gamma$ unfolds to a closed geodesic on $\widetilde X$ of the same length $r_\gamma$, and the closed-surface counting estimate \eqref{eq::closedgeoestimate} immediately gives an $O(e^r)$ bound. If $n_\gamma$ is odd, however, one period of $\gamma$ does not close up on $\widetilde X$. It joins a point to its image under the deck involution. Hence only after two periods does it give a closed geodesic on $\widetilde X$, now of length $2r_\gamma$. A direct use of \eqref{eq::closedgeoestimate} therefore yields only the rough bound $O(e^{2r})$ for the odd-reflection contribution, which is too crude for our purposes. We thus need the following estimate.

\begin{lemma}\label{lem:rough-count-boundary} Let $\pi_X(r)$ denote the number of primitive closed geodesics on $X=G_0\backslash\H^2$ of length less than $r$, including both those reflecting at the boundary and those contained in the interior of $X$. Then there exists a constant $B>0$, depending only on the genus and the number of boundary components of $X$, such that
\begin{equation}\label{roughestimate}
\pi_X(r)\leq B e^r.
\end{equation}
\end{lemma}

\begin{proof}
Let $\iota$ be the involution of $\widetilde X$ induced by reflection across
$\partial X$. By the previous discussion, it suffices to count the primitive closed geodesics with odd $n_\gamma$.

By the collar lemma, outside the standard subcollars the injectivity radius of $\widetilde X$ is at least $\frac{r^*}{2}$.
Choose a maximal set of points $p_1,\ldots,p_k$ outside the standard subcollars such that
\begin{equation*}
d(p_i,p_j)\geq \frac{r^*}{2},\qquad i\neq j.
\end{equation*}
Then the balls $B(p_i,\frac{r^*}{4})$ are pairwise disjoint, and every point outside
the standard subcollars is within distance $\frac{r^*}{2}$ from some $p_i$. Since
\begin{equation*}
\operatorname{Area}(\widetilde X)=4\pi(\widetilde{\mathtt g}-1),
\end{equation*}
we have
\begin{equation*}
k\leq
\frac{2(\widetilde{\mathtt g}-1)}
{\cosh(\frac{r^*}{4})-1}.
\end{equation*}

Let $\gamma$ be a primitive reflected closed geodesic on $X$ with odd
$n_\gamma$ and $r_\gamma\leq r$. By the preceding observation, such a geodesic
cannot be entirely contained in a standard subcollar. Hence its unfolding to $\widetilde X$
meets the complement of the standard subcollars. Choose a point $x$ on the unfolded
segment outside the standard subcollars, and choose $p_i$ such that
\begin{equation*}
d(x,p_i)\leq \frac{r^*}{2}.
\end{equation*}
Since $n_\gamma$ is odd, after cyclically shifting the period, the unfolded
segment may be viewed as a geodesic segment from $x$ to $\iota(x)$ of length
$r_\gamma$. 

Choose a path $\alpha$ from $p_i$ to $x$ of length at most $\frac{r^*}{2}$.
Let $\sigma_\gamma$ denote the unfolded geodesic segment from $x$ to
$\iota(x)$ corresponding to one period of $\gamma$. Consider the path
\[
\beta_\gamma:=\alpha \cdot \sigma_\gamma  \cdot\iota(\alpha)^{-1}
\]
from $p_i$ to $\iota(p_i)$. Its length is at most
\[
r_\gamma+r^*< r+r^*.
\]
Hence the geodesic representative of its relative homotopy class also has
length less than $r+r^*$.

We claim that, for fixed $p_i$, the relative homotopy class
$[\beta_\gamma]$ determines $\gamma$ uniquely, up to orientation.
Fix a lift $\widetilde p_i\in\mathbb H^2$ and a lift
$\widetilde\iota$ of $\iota$. Lifting $\beta_\gamma$ from
$\widetilde p_i$, its endpoint is uniquely of the form
\[
g\widetilde\iota(\widetilde p_i),
\qquad g\in G.
\]
Thus the relative homotopy class determines the orientation-reversing
isometry
\[
h:=g\widetilde\iota.
\]

To see how $h$ determines the original reflected geodesic, let
$\widetilde\alpha$ be the lift of $\alpha$ starting at
$\widetilde p_i$, and write $\widetilde x$ for its endpoint. Let
$\widetilde\sigma_\gamma$ be the lift of $\sigma_\gamma$ starting at
$\widetilde x$. Since the lift of $\beta_\gamma$ ends at
$h\widetilde p_i$, the uniqueness of path lifting applied to the final
segment $\iota(\alpha)^{-1}$ shows that
$\widetilde\sigma_\gamma$ ends at $h\widetilde x$.

We next compare the tangent vectors at the two endpoints. Since
$\sigma_\gamma$ is obtained by unfolding one period of the closed
reflected geodesic, folding the terminal point back by $\iota$ identifies
its terminal tangent vector with its initial tangent vector. Thus,
writing $r_\gamma$ for the length of $\sigma_\gamma$, we have
\begin{equation*}
d\iota_{\iota(x)}
\bigl(\dot\sigma_\gamma(r_\gamma)\bigr)=\dot\sigma_\gamma(0),
\end{equation*}
or equivalently,
\begin{equation*}
\dot\sigma_\gamma(r_\gamma)=
d\iota_x\bigl(\dot\sigma_\gamma(0)\bigr).
\end{equation*}
Let $\Pi:\mathbb H^2\to\widetilde X$ be the covering map. Since
$h=g\widetilde\iota$ is a lift of $\iota$, we have
$\Pi\circ h=\iota\circ \Pi$. Using
$\widetilde\sigma_\gamma(r_\gamma)=h\widetilde x$ and differentiating
$\Pi\circ h=\iota\circ \Pi$, we obtain
\begin{equation*}
\begin{aligned}
d\Pi_{h\widetilde x}
\bigl(\dot{\widetilde\sigma}_\gamma(r_\gamma)\bigr) =
\dot\sigma_\gamma(r_\gamma)=d\iota_x\bigl(\dot\sigma_\gamma(0)\bigr)  =
d\Pi_{h\widetilde x}\,
dh_{\widetilde x}
\bigl(\dot{\widetilde\sigma}_\gamma(0)\bigr).
\end{aligned}
\end{equation*}
Since $\Pi$ is a covering map, $d\Pi_{h\widetilde x}$ is an isomorphism, and
therefore
\begin{equation*}
\dot{\widetilde\sigma}_\gamma(r_\gamma)=dh_{\widetilde x}
\bigl(\dot{\widetilde\sigma}_\gamma(0)\bigr).
\end{equation*}

Let $\widetilde\gamma:\mathbb R\to\mathbb H^2$ denote the complete
geodesic extension of $\widetilde\sigma_\gamma$. The two geodesics
$t\mapsto\widetilde\gamma(t+r_\gamma)$ and
$t\mapsto h\widetilde\gamma(t)$ have the same position and tangent vector
at $t=0$. By uniqueness of the geodesic with prescribed initial data,
\begin{equation}
\widetilde\gamma(t+r_\gamma)
=h\widetilde\gamma(t)
\qquad
\text{for all } t\in\mathbb R.
\end{equation}
Hence the image of $\widetilde\gamma$ is an invariant geodesic of $h$.

Since $h$ is orientation reversing and has positive translation length, it is a glide reflection. Such an isometry has a unique invariant geodesic, namely its axis. It follows that the relative homotopy class $[\beta_\gamma]$, which uniquely
determines $h$, can correspond to at most one reflected closed geodesic on $X$,
up to orientation.
Therefore, the number of primitive odd-reflection geodesics is bounded by the total number of such relative homotopy classes, summed over the finitely many possible choices of $p_i$.

It remains to count these path classes. Fix $p_i$ and choose a lift
$\widetilde p_i\in\mathbb H^2$.
Since $\iota(p_i)$ also lies outside the standard subcollars, the injectivity radius
there is at least $\frac{r^*}{2}$. Therefore for different $g\in G$, these endpoints $g\widetilde\iota(\widetilde p_i)$ are pairwise separated by
at least $r^*$, and the balls of radius $\frac{r^*}{2}$ centered at them are pairwise
disjoint.

If the corresponding path has length at most $r+r^*$, then the endpoint
$g\widetilde\iota(\widetilde p_i)$ lies in
$B_{\mathbb H^2}(\widetilde p_i,r+r^*)$. Hence the disjoint balls of radius
$\frac{r^*}{2}$ centered at these endpoints are contained in
\begin{equation*}
B_{\mathbb H^2}\left(\widetilde p_i,r+\frac{3r^*}{2}\right).
\end{equation*}
It follows that, for each fixed $p_i$,
\begin{equation*}
\#\{\text{such path classes}\}\leq\frac{\cosh(r+\frac{3r^*}{2})-1}
{\cosh(\frac{r^*}{2})-1}
\leq C(r^*)e^r.
\end{equation*}

Summing over the possible choices of $p_i$, and using
$\widetilde{\mathtt g}=2\mathtt g+n-1$, we obtain an
$O_{\mathtt g,n}(e^r)$ bound for the odd-reflection contribution. Combining
this with the even-reflection and interior contributions gives
\begin{equation*}
\pi_X(r)\leq B e^r,
\end{equation*}
where $B>0$ depends only on $\mathtt g$ and $n$.
\end{proof}

\subsection{Selberg zeta estimates}

We now analyze the asymptotic behavior of \eqref{eq::3.31} as the surface $X$ degenerates. Throughout the sequel, for two positive functions $f$ and $g$, we write
\[
f\asymp_c g
\]
to mean that $f$ and $g$ are uniformly comparable up to positive constants depending only on $c$.

The following estimate of Wolpert expresses $Z'_G(1)$ in terms of the short geodesics and the small Laplace eigenvalues.
\begin{proposition}[\cite{Wolpert1987}]\label{eq::3.32}
Let $\widetilde{X}$ be a closed hyperbolic surface of genus $\widetilde{\mathtt g}$, and suppose that $\widetilde{X}$ possesses short closed geodesics of lengths $r_j$, $j=1,\dots,m$. Then
    \begin{equation}
    {Z_G'(1)}\asymp_{\widetilde{\mathtt g}}\ \prod_{j=1}^m e^{-\frac{\pi^2}{3r_j}}r_j^{-1}\!\!\prod_{0<\lambda_k(\widetilde{X})<\frac14}\lambda_k(\widetilde{X}),
\end{equation}
where $\lambda_k(\widetilde X)$ are the Laplace eigenvalues of $\widetilde X$.
\end{proposition}

To calculate the contribution to $Z_{G_0}(s)$ of a closed geodesic $\gamma$ on $X$ of length $r_\gamma$, we define the function
\begin{equation}
    \mathcal{Z}(s,r_\gamma):=\prod_{k=0}^\infty \left(1-e^{-(s+k)r_\gamma}\right).
\end{equation}
The following lemma shows that the contribution of long closed geodesics to the value of $Z_{G_0}(s)$ is uniformly bounded. In other words, it suffices to analyze the effect of those short geodesics as their lengths tend to zero.
\begin{lemma}\label{estimatetail}
For every compact subset $K\subset\{\Re s>1\}$, one has
\begin{equation}\label{eq::3.34}
  {| Z_{G_0}(s)|}\asymp_{K,n,\mathtt{g}}\ {\left|\prod_{\gamma\in \mathscr{C}_1\cup\mathscr{C}_2}\mathcal{Z}(s,r_\gamma)\prod_{b_i<r^*}\left(\mathcal{Z}\Bigl(\frac{s+1}{2},2b_i\Bigr)\right)^2\right|}.
\end{equation}
\end{lemma}
\begin{proof}

We denote by $Z_{G_0}^{(\geq r^*)}(s)$ and $ Z_{G_0}^{(< r^*)}(s)$ the contributions to $Z_{G_0}(s)$, as defined in \eqref{SelbZetaFunc}, of the primitive closed geodesics with length $r_\gamma\geq r^*$ and $r_\gamma< r^*$, respectively. Then
\[
\log Z_{G_0}(s)=\log Z_{G_0}^{(\geq r^*)}(s)+\log Z_{G_0}^{(< r^*)}(s).
\]

The infinite product in \eqref{SelbZetaFunc} is absolutely convergent in the half-plane $\{\Re s>1\}$. Fix an $s\in K$ and write $x=\Re s>1$. For each primitive closed geodesic $\gamma$ with $r_\gamma \ge r^*$ and each $k\geq 0$, we have the estimate
\[
\left|\log (1\pm e^{-(s+k)r_\gamma})\right| \leq \frac{|e^{-(s+k)r_\gamma}|}{1-|e^{-(s+k)r_\gamma}|}\leq C_Ke^{-(x+k)r_\gamma},
\]
where the constant $C_K$ is taken to be
\[C_K=\max_{s\in K} \frac{1}{1-e^{-\Re sr^*}}.\]
Therefore, we have
\begin{align*}
 |\log Z_{G_0}^{(\geq r^*)}(s)|&\leq \sum_{\substack{\gamma\in[C]\\ r_\gamma\geq r^*}}\sum_{k=0}^\infty \left|\log (1\pm e^{-(s+k)r_\gamma})\right|+2\sum_{b_i\geq r^*}\sum_{k=0}^\infty \left|\log (1-e^{-(s+2k+1)b_i})\right| \notag \\
 &\leq C_K\sum_{\substack{\gamma\in[C]\\ r_\gamma\geq r^*}}\sum_{k=0}^\infty e^{-(x+k)r_\gamma}+2C_K \sum_{b_i\geq r^*}\sum_{k=0}^\infty e^{-(x+2k+1)b_i}\notag \\
 &=C_K\sum_{\substack{\gamma\in[C]\\ r_\gamma\geq r^*}}\frac{e^{-xr_\gamma}}{1-e^{-r_\gamma}}+2C_K\sum_{b_i\geq r^*}\frac{e^{-(x+1)b_i}}{1-e^{-2b_i}}\notag\\
 &\leq\frac{C_K}{1-e^{-r^*}}\sum_{\substack{\gamma\in[C]\\ r_\gamma>r^*}}e^{-xr_\gamma}+\frac{2C_K}{1-e^{-2r^*}}\sum_{b_i\geq r^*}e^{-(x+1)b_i}.
\end{align*}
The second sum is uniformly bounded for $s\in K$ since the number of boundary components is finite.
By \eqref{roughestimate}, for any $x>1$ we have the estimate
\[
\sum_{\substack{\gamma\in[C]\\ r_\gamma>r^*}}e^{-xr_\gamma}\lesssim \int_{r^*} ^\infty e^{-(x-1)r} dr = \frac{e^{-(x-1)r^*}}{x-1}.
\]
This estimate immediately implies that
\[\left|\log Z_{G_0}^{(\geq r^*)}(s)\right| \lesssim\frac{e^{-(\Re s-1)r^*}}{\Re s-1},\]
uniformly for $s\in K$. Consequently, the construction of $Z_{G_0}^{(< r^*)}(s)$ gives \eqref{eq::3.34}.
\end{proof}
The asymptotic behavior of the function $\mathcal{Z}(s,r)$ as $r\to0$ is described by the following estimate.
\begin{lemma}[\cite{Wolpert1987}, Lemma 5.1]
     Given a compact subset $K$ of $\{\Re s>0\}$, we have
     \begin{equation}\label{eq::3.35}
        {|\mathcal{Z}(s,r)|}\asymp_K {\left|e^{-\frac{\pi^2}{6r}}r^{\frac12-s}\right|},
     \end{equation}
     for any $s\in K$ and $0<r<1$.
\end{lemma}

 Combining \eqref{eq::3.34} and \eqref{eq::3.35}, we have the following estimate.

 \begin{proposition}
 For any compact subset $K\subset\{\Re s>1\}$ and for each $s\in K$, the following holds:
 \begin{equation}\label{eq::Z_g0}
      {|Z_{G_0}(s)|}\asymp_{K,n,\mathtt{g}}{\left|\prod_{\gamma\in\mathscr{C}_1\cup \mathscr{C}_2}e^{-\frac{\pi^2}{6r_\gamma}}r_\gamma^{\frac12-s}\prod_{b_i<r^*}e^{-\frac{\pi^2}{6b_i}}b_i^{-s}\right|}.
 \end{equation}
 \end{proposition}

Next, we need to evaluate the value of $Z_{G_0}(1)$ appearing in \eqref{eq::3.31}. However, the definition of $Z_{G_0}(s)$ in \eqref{SelbZetaFunc}, which is given by an infinite Euler product, does not converge at $s=1$. To overcome this difficulty, we first compute the value of $Z_{G_0}(s)$ at a point $s=s_0>1$ where the Euler product converges. We then recover the value at $s=1$ by integrating the logarithmic derivative $Z_{G_0}'(s)/Z_{G_0}(s)$ along the real axis from $s=s_0$ to $s=1$. This approach allows us to calculate $Z_{G_0}(1)$ via analytic continuation.

The following relation between the Selberg zeta function $Z_{G_0}(s)$ and the determinant of the Dirichlet Laplacian $\Delta_X^D$ on $X$ was established in \cite[Proposition 3.2]{8180465}:
\[
\det (\Delta_X^D-s(1-s))=Z_{G_0}(s)\left(e^{\eta-\frac{\ell(\partial X)}{8\chi(X)}(1-2s)+s(1-s)}\frac{(2\pi)^{s-1}}{\Theta(s)^2\Gamma(s)}\right)^{-\chi(X)},
\]
where $\eta=2\zeta_R'(-1)-\frac14+\frac12\log(2\pi)$ and $\Theta(s)$ denotes the Barnes function. For real $s\geq1$, one has $s(1-s)\leq0$, so $s(1-s)$ is not an eigenvalue of the positive Dirichlet Laplacian $\Delta_X^D$. Since $\Gamma(s)$ and $\Theta(s)$ are analytic and nonzero on the positive real axis, $Z_{G_0}(s)$ is analytic and nonvanishing on $[1,s_0]$ for every $s_0>1$. We will make essential use of the following McKean-type formula, which plays a crucial role in extending the values of $Z_{G_0}(s)$ to regions where the Euler product representation does not converge.
\begin{proposition}[\cite{10.1215/S0012-7094-86-05345-7}, Proposition 3.1] For any $s \in \mathbb C$ away from the poles of the resolvent,
\begin{align}
    \frac{Z'_{G_0}(s)}{(2s-1)Z_{G_0}(s)} = &\Tr(R_{{X}}(s) - R_{{X}}(s_0)) + \frac{Z'_{G_0}(s_0)}{(2s_0-1)Z_{G_0}(s_0)}\notag \\& - \chi({X})\left(\frac{\Gamma'(s)}{\Gamma(s)} - \frac{\Gamma'(s_0)}{\Gamma(s_0)}\right) + \frac{\ell(\partial X)}{4}\left(\frac{1}{2s-1} - \frac{1}{2s_0-1}\right),\notag
\end{align}
where $R_X(s)=(\Delta_X^D-s(1-s))^{-1}$.
\end{proposition}
Applying this identity and integrating with respect to $s$ from $s_0$ to $1$, we decompose
\begin{equation}\label{eq::z1zs_0}
    \log Z_{G_0}(1)-\log Z_{G_0}(s_0)=\int_{s_0}^1 \frac{Z_{G_0}'(s)}{Z_{G_0}(s)}ds=I_{tr}+I_{s_0}+I_{\Gamma}+I_{\partial},
\end{equation}
where the individual contributions are defined by
\begin{align*}
I_{tr}&:=\int_{s_0}^1 (2s-1)\Tr(R_{{X}}(s) - R_{{X}}(s_0))ds,\\
I_{s_0}&:=\int_{s_0}^1\frac{(2s-1)Z'_{G_0}(s_0)}{(2s_0-1)Z_{G_0}(s_0)}ds=\frac{s_0(1-s_0)}{2s_0-1}\frac{Z'_{G_0}(s_0)}{Z_{G_0}(s_0)},\\
I_{\Gamma}&:=- \chi({X})\int_{s_0}^1(2s-1)\left(\frac{\Gamma'(s)}{\Gamma(s)} - \frac{\Gamma'(s_0)}{\Gamma(s_0)}\right)ds,\\
I_{\partial}&:=\frac{\ell(\partial X)}{4}\int_{s_0}^1\left(1 - \frac{2s-1}{2s_0-1}\right)ds=-\frac{\ell(\partial X)}{4}\frac{(s_0-1)^2}{2s_0-1}.
\end{align*}

We first estimate the contribution of the term $I_{s_0}$.
A direct computation shows (see also \cite[Proposition 3.1]{10.1215/S0012-7094-86-05345-7}) that
\begin{align*}
\frac{Z'_{G_0}(s_0)}{Z_{G_0}(s_0)} =& \sum_{i=1}^n b_i\sum_{k=1}^\infty\frac{e^{-kb_is_0}}{\sinh(kb_i)}\\
 & +\sum_{\substack{\gamma\in [C]\\ n_\gamma\text{ is even}}}\!\frac{r_\gamma}{2}\sum_{k=1}^\infty\frac{e^{-kr_\gamma(s_0-\frac12)}}{\sinh \big(\frac{kr_\gamma}{2}\big)} \\
 &+\sum_{\substack{\gamma\in [C]\\ n_\gamma\text{ is odd}}}\frac{r_\gamma}{2}\sum_{k=1}^\infty\left\{\frac{e^{-2kr_\gamma(s_0-\frac12)}}{\sinh(kr_\gamma)}-\frac{e^{-(2k-1)r_\gamma(s_0-\frac12)}}{\cosh((k-\frac{1}{2})r_\gamma)}\right\}.
\end{align*}

The rough estimate for the number of closed geodesics \eqref{roughestimate} implies that, for fixed $s_0>1$, the contributions from closed geodesics not short in the above sums are uniformly bounded. Indeed, for any $r\geq r^*$, $c_1,c_2>0$, we have
\begin{equation*}
    \sum_{k=1}^\infty\frac{e^{-k rc_1}}{\sinh(krc_2)}\leq\frac{2e^{-r(c_1+c_2)}}{(1-e^{-2c_2r^*})(1-e^{-r^*(c_1+c_2)})},
\end{equation*}
and
\begin{equation*}
    \sum_{k=1}^\infty \frac{e^{-(k-\frac12)rc_1}}{\cosh((k-\frac12)rc_2)}\leq \frac{2e^{-r\frac{c_1+c_2}{2}}}{1-e^{-r^*(c_1+c_2)}}.
\end{equation*}
Using these controls together with \eqref{roughestimate}, we obtain uniform bounds for each sum,
\begin{align*}
    \sum_{b_i\geq r^*} b_i\sum_{k=1}^\infty\frac{e^{-kb_is_0}}{\sinh(kb_i)}&\leq \frac{2}{(1-e^{-2r^*})(1-e^{-r^*(s_0+1)})} \sum_{b_i\geq r^*} b_i e^{-b_i(s_0+1)}\\
    &\leq \frac{2n}{(1-e^{-2r^*})(1-e^{-r^*(s_0+1)})}\sup_{r>0} re^{-r(s_0+1)},
\end{align*}
\begin{align*}
     \sum_{\substack{\gamma\in [C]\\ n_\gamma\text{ is even}\\ r_\gamma\geq r^*}}\!\frac{r_\gamma}{2}\sum_{k=1}^\infty\frac{e^{-kr_\gamma(s_0-\frac12)}}{\sinh \big(\frac{kr_\gamma}{2}\big)}
     &\leq \frac{1}{(1-e^{-r^*})(1-e^{-r^*s_0})} \sum_{\substack{\gamma\in [C]\\ n_\gamma\text{ is even}\\ r_\gamma\geq r^*}}r_\gamma{e^{-r_\gamma s_0} }\\
    &\lesssim \frac{1}{(1-e^{-r^*})(1-e^{-r^* s_0})}\int_{r^*}^\infty re^{-r(s_0-1)}dr,
\end{align*}
 similarly,
\begin{equation*}
     \sum_{\substack{\gamma\in [C]\\ n_\gamma\text{ is odd}}}\frac{r_\gamma}{2}\sum_{k=1}^\infty\frac{e^{-2kr_\gamma(s_0-\frac12)}}{\sinh (kr_\gamma)}
    \lesssim \frac{1}{(1-e^{-2r^*})(1-e^{-2r^* s_0})}\int_{r^*}^\infty re^{-r(2s_0-1)}dr,
\end{equation*}
and
\begin{align*}
     \sum_{\substack{\gamma\in [C]\\ n_\gamma\text{ is odd}}}\frac{r_\gamma}{2}\sum_{k=1}^\infty\frac{e^{-(2k-1)r_\gamma(s_0-\frac12)}}{\cosh ((k-\frac12)r_\gamma)}&\leq \frac{1}{1-e^{-2r^*s_0}} \sum_{\substack{\gamma\in [C]\\ n_\gamma\text{ is odd}}}r_\gamma{e^{-r_\gamma s_0} }\\
    &\lesssim \frac{1}{1-e^{-2r^* s_0}}\int_{r^*}^\infty re^{-r(s_0-1)}dr.
\end{align*}
As a result, there exists a constant $B_\zeta(s_0,n,\mathtt{g})$ that depends only on the selected point $s_0>1$ and the topology $(n,\mathtt{g})$ of $X$, such that
 \begin{align}
\Bigg| \frac{Z'_{G_0}(s_0)}{Z_{G_0}(s_0)} -\Big(\sum_{b_i<r^*}& b_i\sum_{k=1}^\infty\frac{e^{-kb_is_0}}{\sinh(kb_i)} \notag
 \\ &+\sum_{\substack{\gamma\in [C]\\ n_\gamma\text{ is even}\\r_\gamma<r^*}}\frac{r_\gamma}{2}\sum_{k=1}^\infty\frac{e^{-kr_\gamma(s_0-\frac12)}}{\sinh \big(\frac{kr_\gamma}{2}\big)}\Big)\Bigg|\leq B_\zeta(s_0,n,\mathtt{g}).\label{zs_0equ}
\end{align}

Therefore, we only need to consider the short geodesics $\gamma$ that either lie in the boundary or belong to $[C]$ with $n_{\gamma}=0,2$ (that is, $\gamma\in\mathscr{C}_1\cup\mathscr{C}_2$). The asymptotic behavior of the corresponding series is described by the following lemma.
\begin{lemma}\label{lem:small-r-sinh-sum}
Let $s>0$ be fixed. As $r\to 0^+$,
\[
S(r,s):=r\sum_{k=1}^{\infty}\frac{e^{-krs}}{\sinh(kr)}=-\log(2r)-\frac{\Gamma'}{\Gamma}\left(\frac{s+1}{2}\right)+O(r).
\]
\end{lemma}
\begin{proof}
Write
\[
S(r,s)=r\sum_{k=1}^{\infty}f_s(kr)+\sum_{k=1}^{\infty}\frac{e^{-kr}}{k},\qquad
f_s(x):=\frac{e^{-sx}}{\sinh x}-\frac{e^{-x}}{x}.
\]
The function $f_s$ is smooth at $x=0$ and decays exponentially at infinity. Hence Euler--Maclaurin formula gives
\[
r\sum_{k=1}^{\infty}f_s(kr)=\int_0^\infty f_s(x)\,dx+O(r).
\]
Also,
\[
\sum_{k=1}^{\infty}\frac{e^{-kr}}{k}=-\log(1-e^{-r})=-\log r+O(r).
\]
Therefore
\[
S(r,s)=-\log r+\int_0^\infty\left(\frac{e^{-sx}}{\sinh x}-\frac{e^{-x}}{x}\right)\,dx+O(r).
\]
The lemma then follows from the integral representation of $\Gamma'/\Gamma$,
\[
\int_0^\infty
\left(
\frac{e^{-sx}}{\sinh x}-\frac{e^{-x}}{x}
\right)dx=-\frac{\Gamma'}{\Gamma}\left(\frac{s+1}{2}\right)-\log 2.
\]
\end{proof}

Combining \eqref{zs_0equ} and Lemma~\ref{lem:small-r-sinh-sum}, we get the estimate for $e^{I_{s_0}}$,
\begin{equation}\label{Is0}
\exp\left(\frac{s_0(1-s_0)}{2s_0-1}\frac{Z'_{G_0}(s_0)}{Z_{G_0}(s_0)}\right)\asymp_{s_0,n,\mathtt{g}}\prod_{\gamma\in\mathscr{C}_1\cup\mathscr{C}_2}r_\gamma^{-\frac{s_0(1-s_0)}{2s_0-1}}\prod_{b_i<r^*}b_i^{- \frac{s_0(1-s_0)}{2s_0-1}}.
\end{equation}

We now estimate $I_{tr}$ and $I_\partial$. We first record a uniform
bound on the number of small Laplace eigenvalues that will also be used
later.

\begin{lemma}\label{lem:number-small-eigenvalues}
Let $X$ be a compact orientable hyperbolic surface of genus $\mathtt g$
with $n$ geodesic boundary components. Then the numbers of Neumann and
Dirichlet eigenvalues of $X$ in $[0,\frac14]$, counted with multiplicity,
are uniformly bounded in terms of $(n,\mathtt g)$.
\end{lemma}

\begin{proof}
Let $\widetilde X$ be the double of $X$. Otal and Rosas
\cite{OtalRosas2009} proved that $\widetilde X$ has at most
$2\widetilde{\mathtt g}-2$ eigenvalues in $[0,\frac14]$, counted with
multiplicity. Since the reflection involution commutes with the Laplacian,
the spectrum of $\widetilde X$ decomposes into its even and odd parts,
which correspond respectively to the Neumann and Dirichlet spectra of
$X$. Since $\widetilde{\mathtt g}$ is determined by $(n,\mathtt g)$,
the assertion follows.
\end{proof}

By symbol calculus, the operator $R_{{X}}(s) - R_{{X}}(s_0)$ is a trace class operator. If we denote the $k$-th Dirichlet eigenvalue of the Laplacian $\Delta_X^D$ by $\lambda_k^D(X)$ and set $\eta=s_0(1-s_0)<0$, then
\begin{align}
   I_{tr}&=\int_{s_0}^1 (2s-1) \sum_{k=1}^\infty \left(\frac{1}{\lambda_k^D(X)-s(1-s)}-\frac{1}{\lambda_k^D(X)-\eta}\right)ds \notag\\
   &=\sum_{k=1}^\infty \left(\log \frac{\lambda_k^D(X)}{\lambda_k^D(X)-\eta}-\frac{\eta}{\lambda_k^D(X)-\eta}\right).\notag
\end{align}

To separate the contribution of the small eigenvalues from the rest of the
spectrum, we introduce the following notation. For technical convenience, set
\[
\theta(t;X):=\sum_{k=1}^\infty e^{-\lambda_k^D(X)t},\qquad\theta_{\geq\frac14}(t;X):=\sum_{\lambda_k^D(X)\geq \frac14}e^{-\lambda_k^D(X)t},
\]
which denotes the heat trace and the truncated heat trace, respectively. We also
denote by
\[
I_{s}:=\sum_{\lambda_k^D(X)<\frac14} \left(\log \frac{\lambda_k^D(X)}{\lambda_k^D(X)-\eta}-\frac{\eta}{\lambda_k^D(X)-\eta}\right)
\]
the contribution of the small Dirichlet eigenvalues.

Using the identity
\[
\log \frac{\lambda_k^D(X)}{\lambda_k^D(X)-\eta}=\int_0^\infty \frac{e^{-(\lambda_k^D(X)-\eta)t}-e^{-\lambda_k^D(X)t}}{t}dt,
\]
we can express the remaining part of the trace term in terms of the truncated
heat trace. Indeed,
\begin{align}
    I_{tr}-I_s&=\sum_{\lambda_k^D(X)\geq \frac14} \int_0^\infty \Big(\frac{e^{-(\lambda_k^D(X)-\eta)t}-e^{-\lambda_k^D(X)t}}{t}-\eta e^{-(\lambda_k^D(X)-\eta)t}\Big)dt\notag\\
    &=\int_0^\infty \left(\frac{e^{\eta t}-1}{t}-\eta e^{\eta t}\right)\theta_{\geq\frac14}(t;X) dt.\label{trexpression}
\end{align}
Therefore, in order to estimate $I_{tr}$, it suffices to obtain suitable uniform bounds for $\theta_{\geq\frac14}(t;X)$. The following proposition provides the required estimate, whose proof will be given in Appendix~\ref{heattrace}.
\begin{proposition}\label{propheattrace}
There exist positive constants $B_{tr}$ and $b_{tr}$ that depend only on the topology of $X$ such that for all $t>0$,
\begin{align*}
    \bigg|\theta(t;X)-\Bigg(&\sum_{\lambda_k^D(X)<\frac14}e^{-\lambda_k^D(X)t}-2\pi \chi (X) P(t;\mathbb{H}^2)-\frac{e^{-\frac{t}{4}}}{8\sqrt{\pi t}}\ell(\partial X)\\ &+\frac{e^{-\frac{t}{4}}}{2\sqrt{\pi t}}\sum_{b_i<r^*} b_i\sum_{k=1}^\infty\frac{e^{-\frac{(kb_i)^2}{4t}- \frac{kb_i}{2}}}{\sinh(kb_i)}\\
  &+\frac{e^{-\frac{t}{4}}}{4\sqrt{\pi t}}\sum_{\gamma\in\mathscr{C}_1\cup\mathscr{C}_2}r_\gamma\sum_{k=1}^\infty\frac{e^{-\frac{(kr_\gamma)^2}{4t}}}{\sinh\big(\frac{kr_\gamma}{2}\big)} \Bigg)\Bigg|\leq B_{tr}e^{-b_{tr}t},
\end{align*}
where $P(t;\mathbb H^2)$ is the on-diagonal heat kernel of the hyperbolic plane:
\[
P(t;\mathbb H^2)
:=
\frac{e^{-t/4}}{(4\pi t)^{3/2}}
\int_0^\infty \frac{q}{\sinh(q/2)}e^{-\frac{q^2}{4t}}\,dq.
\]
\end{proposition}

Next we substitute this estimate into \eqref{trexpression}.
To streamline the computation, fix $\eta_0=s_0(1-s_0)<0$. For $u\geq0$, define
\begin{align*}
\mathcal{I}(u,\eta_0):&=\int_0^\infty \frac{e^{-\frac{t}{4}}}{4\sqrt{\pi t}} e^{-\frac{u^2}{t}}\left(\frac{e^{\eta_0 t}-1}{t}-\eta_0 e^{\eta_0 t}\right)dt\\
&=\frac{e^{-(2s_0-1)u}-e^{-u}}{4u}-\frac{s_0(1-s_0)}{2(2s_0-1)}e^{-(2s_0-1)u}.
\end{align*}
With the continuous extension at $u=0$, this function satisfies
\begin{align*}
    \mathcal{I}(u,\eta_0)=-\frac{1}{2}\frac{(s_0-1)^2}{2s_0-1}+O(u),\quad u\to 0^+.
\end{align*}
Consequently, the contribution of the total boundary length term to the integral $I_{tr}$ is given by
\[
-\frac{\ell(\partial X)}{2} \mathcal{I}(0,\eta_0)=\frac{\ell(\partial X)}{4}\frac{(s_0-1)^2}{2s_0-1}.
\]
This term is precisely canceled by the boundary term $I_\partial$.
The contribution of the term $-2\pi \chi (X) P(t;\mathbb{H}^2)$ to the integral $I_{tr}$ is independent of the degeneration of $X$. Hence we do not need to compute it.

In summary, there exists a constant $B_{tr}(s_0,n,\mathtt{g})$, such that
\begin{align*}
\Bigg| I_{tr}+I_\partial-\Bigl(I_{s}
+\sum_{b_i<r^*}2b_i\sum_{k=1}^\infty
\frac{\mathcal{I}(\frac{kb_i}{2},\eta_0)e^{-\frac{kb_i}{2}}}{\sinh(kb_i)} \\
+\sum_{\gamma\in\mathscr{C}_1\cup\mathscr{C}_2}
r_\gamma&\sum_{k=1}^\infty
\frac{\mathcal{I}(\frac{kr_\gamma}{2},\eta_0)}{\sinh(\frac{kr_\gamma}{2})}
\Bigr)\Bigg|
\leq B_{tr}(s_0,n,\mathtt{g}).
\end{align*}
Since $\mathcal{I}(u,\eta_0)$ decays exponentially as $u\to +\infty$, Lemma~\ref{lem:small-r-sinh-sum} gives, uniformly for $b_i<r^*$,
\begin{equation*}
    2b_i\sum_{k=1}^\infty
\frac{\mathcal{I}(\frac{kb_i}{2},\eta_0)e^{-\frac{kb_i}{2}}}{\sinh(kb_i)}=2b_i\sum_{k=1}^\infty \frac{\mathcal{I}(0,\eta_0)e^{-kb_is_0}}{\sinh(kb_i)}+O(1)=\frac{(s_0-1)^2}{2s_0-1}\log b_i+O(1).
\end{equation*}
Similarly, uniformly for $r_\gamma<r^*$,
\begin{equation*}
    r_\gamma \sum_{k=1}^\infty
\frac{\mathcal{I}(\frac{kr_\gamma}{2},\eta_0)}{\sinh(\frac{kr_\gamma}{2})}=\frac{(s_0-1)^2}{2s_0-1} \log r_\gamma +O(1).
\end{equation*}
It follows that
\begin{equation}\label{Itr+Id}
    {e^{I_{tr}+I_\partial}}\asymp_{s_0,n,\mathtt{g}} \prod_{\lambda_k^D(X)<\frac14}\lambda_k^D(X) \prod_{\gamma\in\mathscr{C}_1\cup\mathscr{C}_2}r_\gamma^{\frac{(s_0-1)^2}{2s_0-1}}\prod_{b_i<r^*}b_i^{\frac{(s_0-1)^2}{2s_0-1}}.
\end{equation}

The term $I_\Gamma$ depends only on $s_0$ and the topology, so it can be
absorbed into the comparison constant. Combining \eqref{eq::Z_g0},
\eqref{eq::z1zs_0}, \eqref{Is0}, and \eqref{Itr+Id}, the exponent of each
short boundary length is
\[
-s_0+\frac{s_0(s_0-1)}{2s_0-1}
+\frac{(s_0-1)^2}{2s_0-1}=-1,
\]
whereas the exponent of each short interior geodesic length is
\[
\frac12-s_0+\frac{s_0(s_0-1)}{2s_0-1}+\frac{(s_0-1)^2}{2s_0-1}=-\frac12.
\]
Thus all dependence on the auxiliary parameter $s_0$ cancels, and we obtain
the following bordered analogue of Wolpert's estimate in
Proposition~\ref{eq::3.32}.
\begin{proposition}\label{eq::Z_G01}
Let $X$ be a connected orientable compact hyperbolic surface of genus $\mathtt g$ with $n$ geodesic boundary components of lengths $b_1,\ldots,b_n$. Then
\[
  {Z_{G_0}(1)}\asymp_{n,\mathtt{g}}\prod_{\lambda_k^D(X)<\frac14}\lambda_k^D(X) \prod_{\gamma\in\mathscr{C}_1\cup\mathscr{C}_2}e^{-\frac{\pi^2}{6r_\gamma}}r_\gamma^{-\frac12}\prod_{b_i<r^*}e^{-\frac{\pi^2}{6b_i}}b_i^{-1}.\]
\end{proposition}

\subsection{Determinant asymptotics}

Having established this uniform control, we now compare the short oriented primitive closed geodesics on $X$ with their counterparts on $\widetilde X$. Let $\pi:\widetilde X\to X$ be the natural projection. For each short oriented closed geodesic $\gamma$ on $X$, we distinguish the following cases.
\begin{itemize}
	\item If $\gamma\in \mathscr{C}_1$, then $\gamma^{-1}\in \mathscr{C}_1$ is another short oriented closed geodesic with the opposite orientation. So in this case there exist four short oriented closed geodesics $\widetilde\gamma_1, \widetilde\gamma_2,\widetilde\gamma_1^{-1}, \widetilde\gamma_2^{-1}$ on $\widetilde{X}$ with the same length $r_\gamma$ as $\gamma$, such that $\pi(\widetilde\gamma_1)=\pi(\widetilde\gamma_2)=\gamma$ and $\pi(\widetilde\gamma_1^{-1})=\pi(\widetilde\gamma_2^{-1})=\gamma^{-1}$.
	\item If $\gamma\in \mathscr{C}_2$, we have $\gamma=\gamma^{-1}$. So in this case there are two corresponding oriented closed geodesics $\widetilde\gamma$ and $\widetilde\gamma^{-1}$ on $\widetilde{X}$, both of length $r_\gamma$, such that $\pi(\widetilde\gamma)=\pi(\widetilde\gamma^{-1})=\gamma$.
\end{itemize}
Each short boundary component of $X$ becomes a short closed geodesic of the same length on $\widetilde X$.

The above discussion, together with \eqref{eq::3.31}, Proposition~\ref{eq::3.32} and Proposition~\ref{eq::Z_G01}, yields the following uniform estimate for $\frac{{\det}'\Lambda_{X}}{\ell(\partial X)}$.
\begin{theorem}\label{thm::shortgeodesics}
	Let $X$ be a connected orientable compact hyperbolic surface with geodesic boundary. Then the following uniform comparison holds:
\begin{equation}
 \frac{{\det}'\Lambda_{X}}{\ell(\partial X)}\asymp_{n,\mathtt{g}} {e^{\frac{\ell(\partial X)}{4}}\prod_{b_i<r^*}{b_i}\frac{\prod_{0<\lambda^N_k(X)<\frac14}\lambda^N_k(X)}{\prod_{\lambda^D_k(X)<\frac14}\lambda^D_k(X)}}.
    \end{equation}
\end{theorem}
\begin{remark}
   Using Lemma~\ref{lem::LC3}, the estimate can equivalently be written as
\begin{equation*}
 \frac{{\det}'\Lambda_{X}}{\ell(\partial X)}\asymp_{n,\mathtt{g}} \frac{\prod_{b_i<r^*}b_i}{\prod_{\gamma\in\mathscr{C}_2/\sim}r_\gamma}\frac{\prod_{0<\lambda^N_k(X)<\frac14}\lambda^N_k(X)}{\prod_{\lambda^D_k(X)<\frac14}\lambda^D_k(X)}.
    \end{equation*}
\end{remark}
\begin{corollary}\label{cor::laplacian}
Let $X$ be a connected orientable compact hyperbolic surface with geodesic boundary. Then we have
\[
{\det}\, \Delta_X^D\asymp_{\mathtt{g},n} \prod_{b_i<r^*}e^{-\frac{\pi^2}{6b_i}}b_i^{-1}\prod_{\gamma\in\mathscr{C}_1/\sim}e^{-\frac{\pi^2}{3r_\gamma}}r_\gamma^{-1}\prod_{\gamma'\in\mathscr{C}_2/\sim}e^{-\frac{\pi^2}{6r_{\gamma'}}}\prod_{\lambda_k^D(X)<\frac14}\lambda_k^D(X),
\]
and
\[
{\det}' \Delta_X^N\asymp_{\mathtt{g},n} \prod_{b_i<r^*}e^{-\frac{\pi^2}{6b_i}}\prod_{\gamma\in\mathscr{C}_1/\sim}e^{-\frac{\pi^2}{3r_\gamma}}r_\gamma^{-1}\prod_{\gamma'\in\mathscr{C}_2/\sim}e^{-\frac{\pi^2}{6r_{\gamma'}}}r_{\gamma'}^{-1}\prod_{0<\lambda_k^N(X)<\frac14}\lambda_k^N(X).
\]
\end{corollary}
\begin{proof}
The Mayer--Vietoris formula for determinants by Burghelea, Friedlander, and Kappeler \cite[Theorem $\text{B}^*$]{BurgheleaFriedlanderKappeler1992} gives
\begin{equation}\label{eq::bfk}
    \frac{{\det}'\Lambda_X}{\ell(\partial X)}=\frac{1}{-2\pi \chi(X)}\frac{{\det} '\Delta_{\widetilde{X}}}{({\det}\, \Delta_X^D)^2}.
\end{equation}
By Sarnak \cite{Sarnak1987Determinants} (see also D'Hoker--Phong \cite{DHokerPhong1986Determinants}, Voros \cite{Voros1987Spectral}), we have
\begin{equation*}
    {\det} '\Delta_{\widetilde{X}}=Z_G'(1)e^{-2\eta\chi(X)},\quad \eta=2\zeta'_R(-1)-\frac14+\frac12\log(2\pi).
\end{equation*}
Combining this with Proposition~\ref{eq::3.32} and \eqref{eq::bfk}, we obtain
\begin{align*}
({\det}\, \Delta_X^D)^2&\asymp_{n,\mathtt{g}}\frac{\prod_{\gamma\in\mathscr{C}_2/\sim}r_\gamma}{\prod_{b_i<r^*}b_i}\frac{\prod_{\lambda^D_k(X)<\frac14}\lambda^D_k(X)}{\prod_{0<\lambda^N_k(X)<\frac14}\lambda^N_k(X)}\prod_{j=1}^me^{-\frac{\pi^2}{3r_j}}r_j^{-1}\prod_{0<\lambda_k(\widetilde{X})<\frac14}\lambda_k(\widetilde{X})\\
&=\prod_{b_i<r^*}e^{-\frac{\pi^2}{3b_i}}b_i^{-2}\prod_{\gamma\in\mathscr{C}_1/\sim}e^{-\frac{2\pi^2}{3r_\gamma}}r_\gamma^{-2}\prod_{\gamma'\in\mathscr{C}_2/\sim}e^{-\frac{\pi^2}{3r_{\gamma'}}}\left(\prod_{\lambda_k^D(X)<\frac14}\lambda_k^D(X)\right)^2.
\end{align*}
In the second line, we used the correspondence between short geodesics on
$X$ and $\widetilde X$, together with the decomposition of the small spectrum
of $\widetilde X$ into the Neumann and Dirichlet spectra of $X$.
All quantities are positive, so taking square roots gives the stated estimate for $\det\Delta_X^D$.

The reflection involution on $\widetilde X$ splits the nonzero spectrum into its even and odd parts, which are the Neumann and Dirichlet spectra on $X$, respectively. Consequently,
\[
{\det}'\Delta_{\widetilde X}={\det}'\Delta_X^N\,\det\Delta_X^D.
\]
Dividing the estimate for $\det'\Delta_{\widetilde X}$ by the one just obtained for $\det\Delta_X^D$ yields the asserted formula for $\det'\Delta_X^N$.
\end{proof}
\begin{remark}
Wentworth \cite{WentworthDN} established the corresponding determinant estimate for genus-zero hyperbolic surfaces with geodesic boundary, under the assumption that the length of each boundary component is fixed. Theorem~\ref{thm::shortgeodesics} extends his result to arbitrary compact orientable hyperbolic surfaces with geodesic boundary, without fixing the boundary lengths. Likewise, Corollary~\ref{cor::laplacian} extends the corresponding results of Wentworth \cite{WentworthDN} and Kim \cite{Kim2008} to this more general setting.
In particular, when the boundary lengths are fixed, the estimates in Appendix~\ref{smalleigenvaluesDN} uniformly control the additional small-eigenvalue factors in Theorem~\ref{thm::shortgeodesics} and Corollary~\ref{cor::laplacian}, so that Wentworth's estimates are recovered as a special case.
\end{remark}

\section{Compactness theorem for general genus-zero flat surfaces}
\label{sec::compactnessflat}

In this section, we prove a compactness result for genus-zero flat surfaces
with the same Steklov spectrum. 

The argument uses the
Osgood--Phillips--Sarnak uniformization theorem
(Theorem~\ref{thm::OPS}) to pass from flat metrics to hyperbolic metrics
with geodesic boundary. The key point is to rule out degeneration of the corresponding hyperbolic
surfaces. By Theorem~\ref{thm::shortgeodesics}, the normalized Steklov
determinant is controlled by the boundary lengths together with the small
Dirichlet and Neumann eigenvalues. We first introduce finite-dimensional
graph models for these small eigenvalues and then use them to rule out
degeneration.

\subsection{Graph models for the small eigenvalues}

The use of weighted graphs to describe small eigenvalues is motivated by
classical graph-comparison results of Colbois \cite{Colbois1985},
Dodziuk--Pignataro--Randol--Sullivan
\cite{DodziukPignataroRandolSullivan1987}, and Burger
\cite{Burger1990} in the boundaryless setting. The present setting is more delicate, as the geodesic boundary may degenerate and uniform control is required for both the Neumann and Dirichlet spectra.
In the Dirichlet case, the boundary components give rise to additional
weights, represented by a potential on the graph. 

Fix $T>0$, and let $X$ be a connected hyperbolic surface of genus zero
with $n\geq3$ geodesic boundary components
\begin{equation*}
    \partial X=\Gamma_1\cup\cdots\cup\Gamma_n,\qquad b_i:=\ell(\Gamma_i),
\end{equation*}
such that $b_i\leq T$ for every $i$.

Under this assumption, degenerations of $\mathscr C_2$-type are excluded.
Indeed, by the hyperbolic right-angled hexagon formula, pinching a closed
geodesic of $\mathscr C_2$-type forces the length of a boundary component
orthogonal to it to tend to infinity. Hence the only possible degenerating
parameters are the lengths of the interior short closed geodesics in
$\mathscr C_1/\sim$ and the boundary lengths $b_i$.

Write
\begin{equation*}
    \mathscr C_1/\sim=\{\gamma_1,\ldots,\gamma_m\},
\end{equation*}
and denote by $r_{\gamma_i}$ the length of $\gamma_i$. Since $X$ has genus
zero, every interior simple closed geodesic is separating. Thus cutting $X$
along the corresponding standard subcollars gives $m+1$ connected thick
pieces. More precisely, after also removing the standard boundary
subcollars corresponding to those $\Gamma_i$ with $b_i<r^*$, write
\begin{equation*}
    R=
    X\setminus
    \left(
        \bigcup_{i=1}^m SC_1(\gamma_i)\cup\bigcup_{b_i<r^*}SC_b(\Gamma_i)
    \right)
    =
    \bigcup_{\alpha\in\mathcal V}P_\alpha.
\end{equation*}

We associate to this decomposition the adjacency graph
\begin{equation*}
    \mathcal G_X=(\mathcal V,E),
\end{equation*}
whose vertices correspond to the pieces $P_\alpha$ and whose edges
correspond to the interior collars $SC_1(\gamma_i)$. The edge corresponding
to $\gamma_i$ joins the two pieces adjacent to that collar. Since $X$ has
genus zero and every $\gamma_i$ is separating, $\mathcal G_X$ is a tree. Moreover, $m\leq n-3$, so $\mathcal G_X$ has at most $n-2$ vertices.

For the edge corresponding to $\gamma_i$, set
\begin{equation*}
    q_i:=\frac{r_{\gamma_i}}{\pi-4r_{\gamma_i}}.
\end{equation*}
In particular,
\begin{equation*}
    q_i\asymp r_{\gamma_i}.
\end{equation*}
For vertex functions $s=(s_\alpha)_{\alpha\in\mathcal V}$ and
$t=(t_\alpha)_{\alpha\in\mathcal V}$, define
\begin{equation*}
    Q(s,t):= \sum_{i=1}^m q_i
    \bigl(s_{\alpha(i)}-s_{\beta(i)}\bigr)
    \overline{\bigl(t_{\alpha(i)}-t_{\beta(i)}\bigr)},
\end{equation*}
where $P_{\alpha(i)}$ and $P_{\beta(i)}$ are the two pieces adjacent to
$SC_1(\gamma_i)$.

Set
\begin{equation*}
    A_\alpha:=\operatorname{Area}(P_\alpha),
\end{equation*}
and equip the vertex space with the weighted inner product
\begin{equation*}
    \langle s,t\rangle_{\mathcal V} := \sum_{\alpha\in\mathcal V}  A_\alpha s_\alpha\overline{t_\alpha}.
\end{equation*}
Let
\begin{equation*}
    0=\mu_0(\mathcal G_X)  <
    \mu_1(\mathcal G_X) \leq\cdots\leq
    \mu_m(\mathcal G_X)
\end{equation*}
be the generalized eigenvalues of the problem 
\begin{equation*}
    Q(s,t)=\mu\langle s,t\rangle_{\mathcal V}.
\end{equation*}

The following result compares this graph spectrum with the small Neumann
spectrum of $X$. We defer the proof to Appendix~\ref{smalleigenvaluesDN}. 

\begin{proposition}
\label{prop:small-Neumann-product}
Fix $T>0$. Let $X$ be as above, and let
\begin{equation*}
    0=\lambda_0^N(X)  < \lambda_1^N(X) \leq \lambda_2^N(X) \leq\cdots
\end{equation*}
be the Neumann eigenvalues of $-\Delta_X^N$. Then
\begin{equation*}
    \lambda_k^N(X) \asymp_{n,T} \mu_k(\mathcal G_X), \qquad  1\leq k\leq m.
\end{equation*}
Moreover, there exists a constant $c>0$, depending only on $n$ and $T$,
such that
\begin{equation*}
    \lambda_{m+1}^N(X)\geq c.
\end{equation*}
\end{proposition}

Since $\mathcal G_X$ is a tree, the product of its nonzero eigenvalues can
be computed by the matrix-tree theorem. In Appendix~\ref{smalleigenvaluesDN} we also prove the following consequence.

\begin{corollary}\label{corprodN}
Under the assumptions of Proposition~\ref{prop:small-Neumann-product},
\begin{equation*}
\prod_{0<\lambda_j^N(X)<\frac14}\lambda_j^N(X)  \asymp_{n,T} \prod_{j=1}^m\lambda_j^N(X) \asymp_{n,T}\prod_{i=1}^m q_i \asymp \prod_{i=1}^m r_{\gamma_i}.
\end{equation*}
\end{corollary}

For the Dirichlet spectrum, shrinking boundary components also contribute
to the small eigenvalues. We therefore add a potential term to the graph
form. For each boundary component $\Gamma_i$, let $P_{\alpha_b(i)}$ denote
the adjacent thick piece. If $b_i<r^*$, this is the piece adjacent to the
standard boundary subcollar $SC_b(\Gamma_i)$; if $b_i\geq r^*$, it is the
piece whose boundary contains $\Gamma_i$.

Define
\begin{equation}\label{eq:boundary-weight-main}
 p_i := \begin{cases}
 \dfrac{b_i}{\frac{\pi}{2}-2b_i}, & b_i<r^*,\\[6pt] 1,
    & b_i\geq r^*.  \end{cases}
\end{equation}
Thus
\begin{equation*}
    p_i\asymp_{n,T}\min\{1,b_i\}.
\end{equation*}

Let $\mathsf Q$ denote the matrix of the form $Q$ in the standard basis of
vertex functions, and define the diagonal matrix $\mathsf B$ by
\begin{equation*}
    \mathsf B_{\alpha\alpha} :=
    \sum_{\alpha_b(i)=\alpha}p_i.
\end{equation*}
Equivalently, define the Dirichlet graph form
\begin{equation*}
    Q_D(s,t) :=Q(s,t)+\sum_{i=1}^n
    p_i s_{\alpha_b(i)}\overline{t_{\alpha_b(i)}}.
\end{equation*}
Let
\begin{equation*}
    0< \nu_1(\mathcal G_X) \leq\cdots\leq\nu_{m+1}(\mathcal G_X)
\end{equation*}
be the eigenvalues of the problem 
\begin{equation*}
    Q_D(s,t)=\nu\langle s,t\rangle_{\mathcal V}.
\end{equation*}

\begin{proposition}
\label{prop:small-Dirichlet-product}
Fix $T>0$. Let $X$ be as above, and let
\begin{equation*}
    0<\lambda_1^D(X) \leq \lambda_2^D(X) \leq\cdots
\end{equation*}
be the Dirichlet eigenvalues of $-\Delta_X^D$. Then
\begin{equation*}
    \lambda_j^D(X) \asymp_{n,T} \nu_j(\mathcal G_X), \qquad 1\leq j\leq m+1.
\end{equation*}
Moreover, there exists a constant $c>0$, depending only on $n$ and $T$,
such that
\begin{equation*}
    \lambda_{m+2}^D(X)\geq c.
\end{equation*}
Consequently,
\begin{equation*}
    \prod_{\lambda_j^D(X)<\frac14}\lambda_j^D(X) \asymp_{n,T} \prod_{j=1}^{m+1}\lambda_j^D(X) \asymp_{n,T} \det(\mathsf Q+\mathsf B).
\end{equation*}
\end{proposition}
Again the proof is deferred to Appendix~\ref{smalleigenvaluesDN}.

We now use these estimates to show that the small-eigenvalue contribution
in Theorem~\ref{thm::shortgeodesics} forces the normalized Steklov
determinant to vanish along every degenerating family with uniformly
bounded boundary lengths.

\begin{proposition}
\label{prop:ratio-tends-zero-upper-bound}
Fix $T>0$. Let $X$ be a connected hyperbolic surface of genus zero
with $n\geq3$ geodesic boundary components, and assume that every boundary
component has length at most $T$. Then, along any degenerating family in
this class,
\begin{equation*}
    \prod_{b_i<r^*} b_i\,\frac{ \prod_{0<\lambda_k^N(X)<\frac14}\lambda_k^N(X) }{ \prod_{\lambda_k^D(X)<\frac14}\lambda_k^D(X) }
    \longrightarrow0.
\end{equation*}
\end{proposition}

\begin{proof}
By Corollary~\ref{corprodN} and
Proposition~\ref{prop:small-Dirichlet-product},
\begin{equation*}
    \prod_{0<\lambda_k^N(X)<\frac14}\lambda_k^N(X) \asymp_{n,T}\prod_{j=1}^m q_j
\end{equation*}
and
\begin{equation*}
    \prod_{\lambda_k^D(X)<\frac14}\lambda_k^D(X)\asymp_{n,T} \det(\mathsf Q+\mathsf B).
\end{equation*}
Moreover, by \eqref{eq:boundary-weight-main},
\begin{equation*}
    \prod_{b_i<r^*}b_i \asymp_{n,T} \prod_{i=1}^n p_i.
\end{equation*}
Therefore it suffices to prove that
\begin{equation}\label{eq:graph-ratio-main}
    \frac{ \left(\prod_{i=1}^n p_i\right)
        \left(\prod_{j=1}^m q_j\right) }{ \det(\mathsf Q+\mathsf B) }
    \longrightarrow0.
\end{equation}

By Wentworth's matrix-tree formula with potential
\cite[Theorem~6.1]{WentworthDN}, applied to the weighted tree
$\mathcal G_X$ with potential $\mathsf B_{\alpha\alpha}$ at the vertex
$\alpha$, we have
\begin{equation}\label{eq:forest-expansion-main}
    \det(\mathsf Q+\mathsf B)=\sum_F\left(\prod_{e\in F}q_e\right) \left(
        \prod_{C\in\operatorname{Con}(F)}
        \sum_{\alpha\in C}\mathsf B_{\alpha\alpha} \right),
\end{equation}
where the sum runs over all spanning forests $F$ of $\mathcal G_X$ and
$\operatorname{Con}(F)$ denotes the set of connected components of $F$.
Since $\mathcal G_X$ is a tree, all terms in this expansion are
nonnegative.

Taking $F=\mathcal G_X$ gives
\begin{equation*}
    \det(\mathsf Q+\mathsf B)\geq  \left(\prod_{j=1}^m q_j\right) \left(\sum_{i=1}^n p_i\right).
\end{equation*}
Hence
\begin{equation}\label{eq:ratio-bound-boundary}
    \frac{  \left(\prod_{i=1}^n p_i\right)  \left(\prod_{j=1}^m q_j\right)}{
        \det(\mathsf Q+\mathsf B)}
    \leq\frac{\prod_{i=1}^n p_i}{\sum_{i=1}^n p_i}.
\end{equation}

Next, fix an edge $e_j$ corresponding to $\gamma_j$ and take
\begin{equation*}
    F=\mathcal G_X\setminus\{e_j\}.
\end{equation*}
Since $X$ has genus zero, the two components of
$\mathcal G_X\setminus\{e_j\}$ each contain at least one boundary
component. Denote them by $C_j^+$ and $C_j^-$. The corresponding term in
\eqref{eq:forest-expansion-main} gives
\begin{equation*}
    \det(\mathsf Q+\mathsf B)\geq \left(\prod_{\ell\neq j}q_\ell\right)
    \left(\sum_{\alpha\in C_j^+}\mathsf B_{\alpha\alpha}\right)
    \left(\sum_{\alpha\in C_j^-}\mathsf B_{\alpha\alpha}\right).
\end{equation*}
Since each component contains at least one boundary component and all
$p_i$ are uniformly bounded above in terms of $n$ and $T$, we obtain
\begin{equation*}
    \frac{ \prod_{i=1}^n p_i}{
        \left(\sum_{\alpha\in C_j^+}\mathsf B_{\alpha\alpha}\right)
        \left(\sum_{\alpha\in C_j^-}\mathsf B_{\alpha\alpha}\right) } \leq C.
\end{equation*}
It follows that
\begin{equation}\label{eq:ratio-bound-interior}
    \frac{
        \left(\prod_{i=1}^n p_i\right)
        \left(\prod_{\ell=1}^m q_\ell\right) }{
        \det(\mathsf Q+\mathsf B) } \leq Cq_j.
\end{equation}

Combining \eqref{eq:ratio-bound-boundary} and
\eqref{eq:ratio-bound-interior}, we obtain
\begin{equation}\label{eq:ratio-final-bound}
    \frac{
        \left(\prod_{i=1}^n p_i\right) \left(\prod_{j=1}^m q_j\right)}{
        \det(\mathsf Q+\mathsf B) }
    \leq C\min\left\{
        \frac{\prod_{i=1}^n p_i}{\sum_{i=1}^n p_i},  \min_{1\leq j\leq m}q_j
    \right\},
\end{equation}
with the convention that the second minimum is omitted if $m=0$.

Under the assumption $b_i\leq T$, degeneration means that either some
boundary length tends to zero or some interior short closed geodesic length
tends to zero. In the first case, some $p_i\to0$, and since all $p_i$ are
uniformly bounded above,
\begin{equation*}
    \frac{\prod_{i=1}^n p_i}{\sum_{i=1}^n p_i}
    \leq C\min_i p_i \longrightarrow0.
\end{equation*}
In the second case, some $q_j\to0$. Thus the right-hand side of
\eqref{eq:ratio-final-bound} tends to zero in every degenerating family.
This proves \eqref{eq:graph-ratio-main}, and hence the proposition.
\end{proof}

\subsection{Properness of the normalized determinant and compactness}

We now apply the preceding graph estimates to the compactness problem.
Let $\Omega$ be a compact connected genus-zero surface with $n\geq3$ boundary components. Then $\chi(\Omega)=2-n<0$.

For every flat metric $g$ on $\Omega$, the uniformization theorem, Theorem~\ref{thm::OPS}, asserts that each conformal class $[g]$ contains a unique hyperbolic metric $\hat g$ of constant curvature $-1$ whose boundary is geodesic. Consequently, when the Steklov isospectral flat metrics $g$ under consideration are allowed to vary among different conformal classes, the essential issue is to control the possible degeneration of the associated hyperbolic metrics
$\hat g$. In other words, the compactness problem for such flat metrics can be reduced to a compactness problem in the moduli space of hyperbolic surfaces with geodesic boundary.

As before, we denote by $X$ the hyperbolic surface that is isometric to $(\Omega,\hat g)$, and by $\widetilde{X}$ the double of $X$. The strategy is to use the spectral quantities of $(\Omega,g)$ to prevent the pinching of the closed geodesics on $X$. We need the following conformal invariance of the determinant.
\begin{proposition}[\cite{8180465}]\label{prop::zetaconformalinv}
    Let $W$ be a compact connected surface with smooth boundary. For any two conformally related metrics $h_1$ and $h_2$ on $W$, we have
    \begin{equation*}
	\frac{{\det}' \La_{W,h_1}}{L_{h_1}(\partial W)}=\frac{{\det}' \La_{W,h_2}}{L_{h_2}(\partial W)}.
\end{equation*}
\end{proposition}
Combining this with Theorem~\ref{thm::shortgeodesics}, we have the following relation.
\begin{equation}\label{maindet}
\frac{{\det}'\Lambda_{\Omega,g}}{L_{g}(M)}=\frac{{\det}'\Lambda_{X}}{\ell(\partial X)}\asymp_{n,\mathtt{g}} {e^{\frac{\ell(\partial X)}{4}}\prod_{b_i<r^*}{b_i}\frac{\prod_{0<\lambda^N_k(X)<\frac14}\lambda^N_k(X)}{\prod_{\lambda^D_k(X)<\frac14}\lambda^D_k(X)}}.
\end{equation}
Hence the key step in studying compactness of Steklov isospectral metrics on $\Omega$ is to analyze the properness of the function $\frac{{\det}' \Lambda_X}{\ell(\partial X)}$ on the moduli space $\mathcal{M}(\Omega)$.
Direct analysis via \eqref{maindet} is complicated, since different types of short closed geodesics may occur and the degeneration behaviors of the quotient of small eigenvalues are not clear.

To simplify the situation, for any positive constant $T$, we consider the set ${U}_{T}(\Omega)\subset{M}_H(\Omega)$ consisting of hyperbolic metrics whose total boundary length is no more than $T$, i.e.
\[
U_T(\Omega):=\{g_H\in M_H(\Omega):L_{g_H}(M)\leq T\},
\]
equipped with the $C^\infty$ topology. 
This induces the closed subspace
\begin{equation*}
	\mathcal{U}_{T}(\Omega):={U}_{T}(\Omega)/\operatorname{Diff}(\Omega).
\end{equation*}
\begin{lemma}\label{lem::properness}
For every $T>0$, the function
\[
X\longmapsto\frac{\det'\Lambda_X}{\ell(\partial X)}
\]
is proper on $\mathcal U_T(\Omega)$ as a map to $(0,\infty)$.
\end{lemma}
\begin{proof}
By Proposition~\ref{prop:ratio-tends-zero-upper-bound} and \eqref{maindet}, the normalized determinant tends to zero along every degenerating sequence in $\mathcal U_T(\Omega)$. Let $K\subset(0,\infty)$ be compact. Any sequence in the inverse image of $K$ has its normalized determinant bounded below by a positive constant, and therefore cannot degenerate. Mumford compactness then provides a convergent subsequence in $\mathcal U_T(\Omega)$. Since the normalized determinant is continuous, the inverse image of $K$ is closed and hence compact.
\end{proof}

Let $\widetilde{\mathcal U}_T$ denote the set of flat metrics whose canonical hyperbolic representatives lie in $U_T(\Omega)$:
\[
\widetilde{\mathcal U}_T:=\{(\Omega,g)\mid g\text{ is flat and its canonical representative }\hat g\in U_T(\Omega)\}.
\]
We can now state and prove the following theorem, which is equivalent to Theorem~\ref{maincompactnesshyperbolic}.
\begin{theorem}\label{thm::utcompactness}
   For every fixed $T>0$, any family of Steklov isospectral flat surfaces
contained in $\widetilde{\mathcal U}_T$ is compact in the $C^\infty$
topology.
\end{theorem}
\begin{proof}
Let $(\Omega,g_k)$ be a sequence of Steklov isospectral surfaces in
$\widetilde{\mathcal U}_{T}$. By Theorem~\ref{thm::OPS}, for each $k$
there exists a unique canonical hyperbolic metric
$\hat g_k\in U_T(\Omega)$ such that
\begin{equation*}
g_k=e^{2\phi_k}\hat g_k,
\end{equation*}
where $\phi_k\in C^\infty(\Omega;\mathbb R)$.

Both ${\det}'\Lambda_{\Omega,g_k}$ and $L_{g_k}(M)$ are Steklov spectral
invariants. Hence, by \eqref{maindet} and Lemma~\ref{lem::properness}, the
isometry classes of the hyperbolic metrics $\hat g_k$ remain in a compact
subset of $\mathcal M(\Omega)$. Therefore there exist diffeomorphisms $F_k:\Omega\to\Omega$ such that
the pullback metrics $F_k^*\hat g_k$ form a precompact family in the
$C^\infty$ topology. Since $F_k^*g_k$ is isometric to $g_k$, it suffices
to prove that the family $\{F_k^*g_k\}$ is precompact in the
$C^\infty$ topology. Moreover,
\begin{equation*}
F_k^*g_k=e^{2F_k^*\phi_k}F_k^*\hat g_k.
\end{equation*}

Embed $\Omega$ into a closed surface $\mathcal S_0$ diffeomorphic to $S^2$.
Using a fixed smooth extension procedure, we may extend the metrics
$F_k^*\hat g_k$ to $\mathcal S_0$ so that the resulting family, still
denoted by $F_k^*\hat g_k$, remains precompact in the $C^\infty$ topology
on $\mathcal S_0$.

Applying \eqref{bdfinalestimate} and \eqref{finalestimate} with background
metric $F_k^*\hat g_k$, for every $m\geq 1$ we obtain
\begin{equation*}
\|F_k^*\phi_k\|_{H^m(\Omega,F_k^*\hat g_k)}
\leq
C_m\Bigl(
1+\bigl(\zeta_{\Omega,F_k^*g_k}(-2)\bigr)^m+\zeta_{\Omega,F_k^*g_k}(-2m)
\Bigr),
\end{equation*}
where $C_m$ can be chosen independently of $k$. Indeed, after passing to the compact closure of $\{F_k^*\hat g_k\}$, the conformal cylinders used in \eqref{bdfinalestimate} have heights in compact subsets of $(0,\infty)$, and their conformal maps and boundary factors vary in bounded $C^\infty$ families. Thus the Sobolev multiplication and composition constants in the boundary estimates are uniform. The Sobolev norms associated with $F_k^*\hat g_k$ are uniformly equivalent, their curvature tensors are bounded in every $C^m$ norm, and the constants in the Dirichlet elliptic estimates \eqref{finalestimate} can also be chosen uniformly.

Moreover, since $F_k^*g_k$ is isometric to $g_k$,
\begin{equation*}
\zeta_{\Omega,F_k^*g_k}(-2)=\zeta_{\Omega,g_k}(-2),\qquad\zeta_{\Omega,F_k^*g_k}(-2m)=
\zeta_{\Omega,g_k}(-2m),\end{equation*}
and these quantities are Steklov spectral invariants.
Thus, for every $m\geq 1$,
\begin{equation*}
\|F_k^*\phi_k\|_{H^m(\Omega,F_k^*\hat g_k)}
\leq C_m'.
\end{equation*}
The equivalence of the above norms induced by $F_k^*\hat g_k$, together with Rellich compactness and Sobolev
embedding, therefore gives a subsequence such that
\begin{equation*}
F_k^*\phi_k\longrightarrow\phi_\infty
\qquad\text{in }C^\infty(\Omega).
\end{equation*}
Passing to a further subsequence if necessary, the $C^\infty$
precompactness of $\{F_k^*\hat g_k\}$ also gives
\begin{equation*}
F_k^*\hat g_k\longrightarrow\hat g_\infty
\qquad\text{in }C^\infty(\Omega).
\end{equation*}
Consequently,
\begin{equation*}
F_k^*g_k=e^{2F_k^*\phi_k}F_k^*\hat g_k
\longrightarrow
e^{2\phi_\infty}\hat g_\infty
\qquad\text{in }C^\infty(\Omega).
\end{equation*}
This proves the theorem.
\end{proof}

\section{Compactness theorem for planar domains}
In this section, we study the Steklov isospectral compactness problem for
planar domains, viewed as a special class of genus-zero flat surfaces. Let $(\Omega,g_E)$ be a
smooth planar domain with boundary
\[
M=M_1\cup M_2\cup\cdots\cup M_n,
\]
and suppose that its Steklov spectrum is known. We assume that
$n\geq 3$ and that $M_1$ is the outer boundary component enclosing all the
other boundary components.

\subsection{Boundary non-collapse}

We begin by placing this planar domain into the geometric framework used in
Section~\ref{traceineq}. Since the Steklov spectrum determines the multiset of boundary lengths, we choose labels so that the length of each component is fixed throughout the isospectral family. Let $L_{\max}$ and $L_{\min}$ denote the largest and smallest of these lengths. The diameter of the outer boundary is at most half its length, and $\Omega$ is contained in the convex hull of that boundary. Hence, after a translation, there exists $R>0$, depending only on $L_{\max}$, such that $\Omega\subset B_R$.

The metric on $B_{2R}$ is already Euclidean. We embed this disk isometrically into a closed surface $(\mathcal S_0,g)$ by attaching a fixed smooth cap outside $B_{2R}$, chosen so that $g$ has nonnegative Gaussian curvature.
By construction, $g$ coincides with the Euclidean metric on a neighborhood of
$\Omega_{\{1\}}$ in $\mathcal{S}_0$, while its curvature is non-negative on the disk $U_1$.
This allows us to apply the notation and the trace formulas developed in
Section~\ref{traceineq}.

We fix the following sign convention. Each boundary component is oriented so that $\Omega$ lies to its left, $\nu$ is the outward unit normal. Let $\tau$ be an arc-length parameter compatible with this orientation. The geodesic curvature of the boundary is defined by
\[
k_g=-g(\nabla_{\partial_\tau}\partial_\tau,\nu)
\]
With this convention, a convex outer boundary has positive geodesic curvature.

For boundary components $M_i$ and $M_j$, define
\[
\operatorname{dist}^{\perp}_{g_E}(M_i,M_j)
\]
to be the infimum of the lengths of  $g_E$-geodesic segments in $\Omega$ that join $M_i$ to $M_j$ and meet the boundary orthogonally at both endpoints. When $i=j$, the two endpoints are required to be distinct. If no such segment exists, we set $\operatorname{dist}^{\perp}_{g_E}(M_i,M_j)=+\infty$.
This distance measures two possible types of boundary degeneration in the interior of $\Omega$: mutual collapse of distinct boundary components and self-collapse of a single boundary component.

\begin{lemma}\label{lem:noncollapseineqj}
 Let $(\Omega,g_E)$ be a planar domain as above. Then there exists a constant
$\delta_{\mathrm{sep}}>0$, depending only on the Steklov spectrum of $(\Omega,g_E)$, such
that $\operatorname{dist}^{\perp}_{g_E}(M_i,M_j)\geq \delta_{\mathrm{sep}}$ for any $i\neq j$.
\end{lemma}
\begin{proof}
Fix $J\subset\{1,2,\dots,n\}$ with $\#J=2$. By Theorem~\ref{thm::OPS}, there exist a cylinder \[\mathcal{C}_J:=S^1\times[0,l_J]\] with the canonical product metric $g_{\mathrm{can}}$ and a function $\phi_J\in C^\infty(\mathcal C_J;\mathbb R)$, such that
\[
    (\Omega_J,g) \quad\text{is isometric to}\quad
    (\mathcal C_J,e^{2\phi_J}g_{\mathrm{can}}).
\]

By Corollary~\ref{cor::zetamonotonicity}, applied at $s=-2$, we have
\begin{equation}\label{eq::inzetacJ-2}
    \zeta_{\mathcal{C}_J,e^{2\phi_J}g_{\mathrm{can}}}(-2)=\zeta_{\Omega_J,g}(-2)
    \leq \zeta_{\Omega,g}(-2)=\zeta_{\Omega,g_E}(-2).
\end{equation}
On the other hand, the explicit formula for the spectral zeta function associated with a
conformal cylinder, see \cite[(4.10) and (4.16)]{annular}, gives
\begin{equation}\label{eq:jw2640.10}
    \zeta_{\mathcal{C}_J,e^{2\phi_J}g_{\mathrm{can}}}(-2)\geq \pi^2B_{l_J}(0)\left(\sum_{j\in J}\bigl(L_{g_E}(M_j)\bigr)^{-1}\right)^2\geq \frac{4\pi^2}{L_{\max}^2}B_{l_J}(0),
\end{equation}
where
\begin{equation*}
B_{l_J}(0):=\sum_{k\in\Z\setminus\{0\} }\left(k^2\tanh^2\frac{kl_J}{2}+k^2\coth^2\frac{kl_J}{2}-2k^2\right)+\frac{4}{l_J^2}.
\end{equation*}
The function $B_{l_J}(0)$ is strictly decreasing in ${l_J}$, and satisfies
\begin{equation}\label{eq::blj0}
    B_{l_J}(0)=\frac{4}{l_J^2}+O(e^{-2l_J}), \qquad {l_J}\to\infty.
\end{equation}
Moreover, $B_{l_J}(0)\to+\infty$ as ${l_J}\to 0^+$. The inequality \eqref{eq:jw2640.10} implies that $B_{l_J}(0)$ is bounded above by a constant depending
only on the Steklov spectrum of $(\Omega,g_E)$. Consequently, there exists
$l_{\min}>0$, depending only on the Steklov spectrum, such that
$l_J\geq l_{\min}$ for every such $J$.

We next obtain a uniform lower bound for the boundary values of $e^{\phi_J}$.
Equations (5.6) and (5.7) in \cite{annular} imply that
\begin{equation}\label{inequephiJ}
    e^{-\phi_J}|_{\partial\mathcal{C}_J}\leq r'_{l_J}\sqrt{\zeta_{\mathcal{C}_J,e^{2\phi_J}g_{\mathrm{can}}}(-2)},
\end{equation}
where
\[
r'_{l_J}=C_{\mathrm{emb}}\sqrt{\frac{4}{c_1}+r_{l_J}},\qquad r_{l_J}:=\max\left\{\frac{8}{B_{l_J}(1)},\,\frac{8}{B_{l_J}'(1)},\,\frac{4}{B_{l_J}(0)}
\right\}.
\]
Here $C_{\mathrm{emb}}$ is the embedding constant for
$H^{\frac32}(S^1)\hookrightarrow C(S^1)$, and the positive constant $c_1$ may be taken to be $c_1=\frac49$ (see \cite[proof of Lemma 5.3]{JOLLIVET20181712}). The quantities
$B_{l_J}(1)$ and $B_{l_J}'(1)$ are given by
\begin{align*}
    B_{l_J}(1)
:=\frac{4}{l_J}\coth\frac{l_J}{2}&+2\sum_{k=1}^{\infty}k(k+1)
\Big[\tanh\frac{k l_J}{2}\tanh\frac{(k+1)l_J}{2}\\&+\coth\frac{k l_J}{2}\coth\frac{(k+1)l_J}{2}-2\Big],
\end{align*}
and
\begin{align*}
B_{l_J}'(1):=\frac{4}{{l_J}}\tanh\frac {l_J}2&+2\sum_{k=1}^{\infty}k(k+1)
\Big[
\tanh\frac{k{l_J}}{2}\coth\frac{(k+1){l_J}}{2}\\&+\coth\frac{k{l_J}}{2}\tanh\frac{(k+1){l_J}}{2}-2
\Big].
\end{align*}
Both $B_{l_J}(1)$ and $B_{l_J}'(1)$ are strictly decreasing in ${l_J}$, and, as
${l_J}\to\infty$,
\begin{equation}\label{eq::Blj1}
    B_{l_J}(1) =\frac{4}{l_J}+O({e^{-l_J}}/{l_J}),\qquad B_{l_J}'(1) =\frac{4}{l_J}+O({e^{-l_J}}/{l_J}).
\end{equation}

Combining \eqref{eq::inzetacJ-2}, \eqref{eq::blj0}, \eqref{inequephiJ} and \eqref{eq::Blj1}, for $l_J\geq l_{\min}$, there exists a positive constant $\delta_{\mathrm{sep}}$ that depends only on the Steklov spectrum of $(\Omega,g_E)$ such that
\[
e^{\phi_J}|_{\partial\mathcal{C}_J}\geq \delta_{\mathrm{sep}}l_J^{-1}.
\]

It remains to extend this lower bound from the boundary to the whole cylinder.
The conformal factor $\phi_J$ satisfies
\[
-\Delta_{g_\mathrm{can}}\phi_J=e^{2\phi_J}K_g,
\]
where $K_g$ is the Gaussian curvature of
$(\mathcal C_J,e^{2\phi_J}g_{\mathrm{can}})$. Under the standing assumptions,
$K_g\geq 0$. Hence
\[
-\Delta_{g_{\mathrm{can}}}\phi_J\geq 0.
\]
By the minimum principle, on $\mathcal{C}_J$ we have
\[
e^{\phi_J}\geq \min_{\partial\mathcal{C}_J} e^{\phi_J}\geq \delta_{\mathrm{sep}}l_J^{-1}.
\]

Now write $J=\{j_1,j_2\}$. Every curve in $\mathcal C_J$ joining the two
boundary components has $g_{\mathrm{can}}$-length at least $l_J$.
Therefore, via the
isometry between $(\Omega_J,g)$ and
$(\mathcal C_J,e^{2\phi_J}g_{\mathrm{can}})$, we obtain
\[
\operatorname{dist}^{\perp}_{g_E}(M_{j_1},M_{j_2})\geq l_J\min_{\mathcal{C}_J} e^{\phi_J}\geq\delta_{\mathrm{sep}}.
\]
The constant $\delta_{\mathrm{sep}}$ is independent of $J$. This completes the
proof.
\end{proof}

We will use the following consequence of the proof of the Steklov
isospectral compactness theorem of Jollivet and Sharafutdinov
\cite[Section~5]{JOLLIVET20181712}. Although their theorem is stated for
isospectral families, the proof only requires uniform bounds for the
spectral quantities involved. More precisely, Edward's estimates \cite[Propositions~3 and~4]{Edward01011993}, used in
\cite[Lemma~5.2]{JOLLIVET20181712} require a fixed boundary length and
uniform bounds for $\zeta_U(-1)$ and $\zeta_U(-2)$, while the subsequent
bootstrap uses only uniform bounds for the zeta-invariants
$\zeta_U(-2m)$, $m\geq1$. Consequently, their argument gives the
following criterion.

\begin{theorem}\label{JS18}
Let $\mathcal{F}$ be a family of smooth flat topological disks with the
same boundary length. Suppose that there exists a constant $C_1$ such
that
\begin{equation*}
\zeta_U(-1)\leq C_1,
\end{equation*}
and, for every integer $m\geq1$, a constant $C_{2m}$ such that
\begin{equation*}
\zeta_U(-2m)\leq C_{2m}
\end{equation*}
for all $U\in\mathcal F$. Then $\mathcal F$ is compact in the
$C^\infty$ topology.
\end{theorem}
We will use this result to control the inner self-collapse of the boundary.
\begin{lemma}\label{lem:noncollapsei}
Let $(\Omega,g_E)$ be a planar domain as above. Then there exist constants
$\delta_{\mathrm{ret}},\bar k>0$, depending only on the Steklov spectrum of
$(\Omega,g_E)$, such that
\begin{equation*}
\operatorname{dist}^{\perp}_{g_E}(M_i,M_i)\geq \delta_{\mathrm{ret}}
\end{equation*}
for every $i\in\{1,\ldots,n\}$, and the geodesic curvature of
$M$ satisfies
\begin{equation*}
k_{g_E}\leq \bar k.
\end{equation*}
\end{lemma}
\begin{proof}
Fix $i\in\{1,\dots,n\}$. By construction, $(\Omega_{\{i\}},g)$ is a
topological disk and has nonnegative Gaussian curvature $K_g\geq 0$. Define a
conformal metric
\[
    g_{i,\mathrm{flat}}:=e^{2\phi_i}g
\]
on $\Omega_{\{i\}}$, where $\phi_i\in C^\infty(\Omega_{\{i\}};\mathbb R)$ is the
solution of
\[\left\{
\begin{aligned}
    -\Delta_g\phi_i&=-K_g, && \text{in } \Omega_{\{i\}},\\
    \phi_i&=0, & &\text{on } M_i.
\end{aligned}\right.
\]
Since
\[
    K_{g_{i,\mathrm{flat}}} =e^{-2\phi_i}(K_g-\Delta_g\phi_i)=0,
\]
the metric $g_{i,\mathrm{flat}}$ is flat. Moreover, $\Delta_g\phi_i=K_g\geq 0$
and $\phi_i|_{M_i}=0$, so the maximum principle gives
\[
    \phi_i\leq 0 \quad \text{on } \Omega_{\{i\}}.
\]
In particular, for the outward unit normal vector field $\nu$ along $M_i$, we have
\[
    \partial_\nu \phi_i\geq 0.
\]

Because $\phi_i$ vanishes on $M_i$, the Dirichlet-to-Neumann operators of
$(\Omega_{\{i\}},g)$ and $(\Omega_{\{i\}},g_{i,\mathrm{flat}})$ coincide on
$M_i$. Hence, by Corollary~\ref{cor::zetamonotonicity}, for every $s\leq -1$,
\begin{equation}\label{eq:giflatupperbound}
      \zeta_{\Omega_{\{i\}},g_{i,\mathrm{flat}}}(s)
  =\zeta_{\Omega_{\{i\}},g}(s)\leq\zeta_{\Omega,g}(s)-2\zeta_R(s)\sum_{j\in\{1,\dots,n\}\setminus\{i\}}
  \left(\frac{2\pi}{L_g(M_j)}\right)^{-s}.
\end{equation}
Since the boundary lengths $L_g(M_j)$ are Steklov spectrally determined and satisfy
\[
    L_{\min}\leq L_g(M_j)\leq L_{\max},
\]
the estimate \eqref{eq:giflatupperbound} gives spectral upper bounds for
$\zeta_{\Omega_{\{i\}},g_{i,\mathrm{flat}}}(s)$, $s\leq -1$. Therefore, by
Theorem~\ref{JS18}, the family of flat topological disks
$(\Omega_{\{i\}},g_{i,\mathrm{flat}})$ is compact in the $C^\infty$ topology.
Since there are only finitely many boundary components, this compactness is
uniform in $i$.

We next compare the boundary curvatures on $M_i$. Since $\phi_i|_{M_i}=0$, the
conformal transformation law for geodesic curvature gives
\[
    k_{g_{i,\mathrm{flat}}} = k_g+\partial_\nu\phi_i.
\]
Thus
\[
    k_g = k_{g_{i,\mathrm{flat}}}-\partial_\nu\phi_i
 \leq k_{g_{i,\mathrm{flat}}}.
\]
By the $C^\infty$-compactness of the flat disks 
$(\Omega_{\{i\}},g_{i,\mathrm{flat}})$, the boundary curvature
$k_{g_{i,\mathrm{flat}}}$ is uniformly bounded from above. Hence there exists a
constant $\bar k>0$, depending only on the Steklov spectrum of $(\Omega,g_E)$,
such that
\begin{equation}\label{eq:upperboundk}
  k_g\leq \bar k\quad \text{on } M_i.    
\end{equation}

It remains to exclude returning orthogeodesics of arbitrarily small length. Suppose, to the contrary, that there are domains in the isospectral family and returning Euclidean orthogeodesics $\gamma_k$, with endpoints $p_k,q_k\in M_i$, such that
\[
\ell_g(\gamma_k)\longrightarrow0.
\]
Since $\phi_i\leq0$, we have $g_{i,\mathrm{flat}}\leq g$, and hence the
$g_{i,\mathrm{flat}}$-distance between $p_k$ and $q_k$ tends to zero.
By the $C^\infty$-compactness of the flat disks, the intrinsic boundary
distance between $p_k$ and $q_k$ also tends to zero. Indeed, otherwise,
after pulling back a subsequence to a fixed disk, $C^\infty$ convergence
would give two coincident limit boundary points whose intrinsic boundary
distance stays bounded away from zero, a contradiction. Consequently, the
shorter boundary arc $\alpha_k\subset M_i$ joining $p_k$ to $q_k$ satisfies
\[
\ell_{g_{i,\mathrm{flat}}}(\alpha_k)\longrightarrow0.
\]
Since $\phi_i=0$ on $M_i$, also $\ell_g(\alpha_k)\to0$.

For all sufficiently large $k$, the simple loop $\alpha_k\cup\gamma_k$ bounds the local domain-side disk. Indeed, its diameter tends to zero, while Lemma~\ref{lem:noncollapseineqj} excludes any other boundary component from this loop. The segment $\gamma_k$ is Euclidean and meets $M_i$ orthogonally at both endpoints. Gauss--Bonnet, with the convention fixed above, gives
\[
\int_{\alpha_k}k_g\,d\tau+\frac\pi2+\frac\pi2=2\pi,
\qquad\text{so}\qquad
\int_{\alpha_k}k_g\,d\tau=\pi.
\]
On the other hand, \eqref{eq:upperboundk} implies
\[
\pi\leq\bar k\,\ell_g(\alpha_k),
\]
a contradiction. Thus all returning orthogeodesics have length at least a uniform constant $\delta_{\mathrm{ret}}>0$. Since $g=g_E$ on $\Omega$, we obtain
\[
k_{g_E}\leq\bar k,\qquad\operatorname{dist}^{\perp}_{g_E}(M_i,M_i)\geq\delta_{\mathrm{ret}}.
\]
This completes the proof.
\end{proof}

\subsection{Control of the hyperbolic boundary length}

The preceding non-collapse estimates provide uniform embedded Euclidean
collars along all boundary components. We now use these collars to control
the boundary length of the canonical hyperbolic representative and complete
the proof of the planar compactness theorem.

\begin{theorem}\label{thm::planar}
Any family of smooth Steklov isospectral planar domains with $n\geq3$ boundary components is compact in the $C^\infty$ topology.
\end{theorem}
\begin{proof}
Let $\mathcal{F}$ denote the family of smooth Steklov isospectral planar
domains. We set
\begin{equation*}
\delta:=\frac14
\min\{\delta_{\mathrm{sep}},\delta_{\mathrm{ret}},\bar k^{-1}\}.
\end{equation*}
We claim that, for every $(\Omega,g_E)\in\mathcal F$, the inward Euclidean
normal maps
\begin{equation*}
F_i:M_i\times[0,\delta]\longrightarrow\Omega,\qquad F_i(p,r)=p+r\nu_{\mathrm{in}}(p),
\end{equation*}
define pairwise disjoint smooth embedded collars.

Indeed,
\begin{equation*}
dF_i(\partial_\tau)
=
(1-rk_{g_E})\partial_\tau,
\end{equation*}
so no focal point occurs for $r\leq\delta<\bar k^{-1}$.

Suppose that the conclusion fails. Since there are no focal points, there
is a smallest $r\leq\delta$ for which some point $x\in\Omega$ has two
distinct nearest boundary points $p$ and $q$ at distance $r$.  
The segments from $p$ and $q$ to $x$  meet the
boundary orthogonally.
At $x$, these two segments must be tangent. Indeed, suppose that
\begin{equation*}
F_i(p,r)=F_j(q,r)=x,
\qquad p\neq q,
\end{equation*}
at the first intersection. If these two segments
intersected transversely, then the derivative with respect to $(p,q)$
of
\begin{equation*}
F_i(p,s)-F_j(q,s)
\end{equation*}
would be invertible at $s=r$. By the implicit function theorem, for
every $s<r$ sufficiently close to $r$, there would then exist nearby
boundary points $p_s$ and $q_s$ such that
\begin{equation*}
F_i(p_s,s)=F_j(q_s,s),
\end{equation*}
contradicting the minimality of $r$.
Hence the two inward normal directions at $p$ and $q$ are parallel. Namely,
\begin{equation*}
\nu_{\mathrm{in}}(q)=-\nu_{\mathrm{in}}(p).
\end{equation*}
Consequently, the two normal segments concatenate to form a Euclidean
geodesic segment from $p$ to $q$ of length $2r\leq2\delta$, orthogonal
to the boundary at both endpoints. This contradicts the definition of
$\delta_{\mathrm{sep}}$ if $p$ and $q$ lie on distinct boundary
components, and that of $\delta_{\mathrm{ret}}$ if they lie on the same
component. The claim follows.

By Theorem~\ref{thm::OPS}, there exists a unique hyperbolic metric
$\hat g\in[g_E]$ such that each boundary component $M_i$ is geodesic with
respect to $\hat g$. We now show that the total $\hat g$-length of the
boundary is bounded from above in terms of the Euclidean boundary lengths
$L_{g_E}(M_i)$ and $\delta$.

Denote by
\[
A_i:=F_i(M_i\times[0,\delta])
\]
the corresponding Euclidean collar. Again by Theorem~\ref{thm::OPS},
there exists a unique flat cylinder
\[
\mathcal C_i=S^1\times[0,l_i],\qquad |S^1|=2\pi,
\]
which is conformal to $A_i$. We write $(\theta,t)$ for the coordinates on
$\mathcal C_i$, and assume that $M_i$ corresponds to $S^1\times\{0\}$.

We first prove that
\[
L_{\hat g}(M_i)\leq \frac{\pi^2}{l_i}.
\]
Indeed, on $\mathcal C_i$ we may write
\[
\hat g=e^{2\phi_H}(d\theta^2+dt^2).
\]
Since $K_{\hat g}=-1$, and the background cylinder metric
$d\theta^2+dt^2$ is flat, we have
\[
\Delta\phi_H=e^{2\phi_H}.
\]
Moreover, $M_i$ is a $\hat g$-geodesic, while $S^1\times\{0\}$ is geodesic
with respect to the flat cylinder metric. Hence the conformal transformation
law for geodesic curvature gives
\[
\partial_t\phi_H|_{t=0}=0.
\]

Now consider the one-dimensional function
\[
v(t)=\log\left(\frac{\pi}{2l_i}
\sec\frac{\pi t}{2l_i}\right).
\]
Then
\[
v''=e^{2v},\qquad v'(0)=0,
\]
and $v(t)\to+\infty$ as $t\to l_i$. We claim that
\begin{equation}\label{eq:claimphih}
    \phi_H(\theta,t)\leq v(t)
\end{equation}
on $S^1\times[0,l_i)$.

To see this, set $w=\phi_H-v$. Since $\phi_H$ is smooth up to
$S^1\times\{l_i\}$ while $v(t)\to+\infty$ as $t\to l_i$, we have
$w<0$ near $S^1\times\{l_i\}$. If $w$ had a positive maximum at an interior
point, then at that point
\[
0\geq \Delta w
=\Delta\phi_H-\Delta v=e^{2\phi_H}-e^{2v}
>0,
\]
which is impossible.
A positive maximum on $S^1\times\{0\}$ is also impossible: since
$\partial_t w|_{t=0}=0$, the Hopf lemma would give
$-\partial_t w>0$ there, a contradiction. Therefore $w\leq0$, proving the claim \eqref{eq:claimphih}.

In particular,
\[
\phi_H(\theta,0)\leq v(0)=\log\frac{\pi}{2l_i}.
\]
Hence
\begin{equation}\label{eq:LgMiupperbound}
    L_{\hat g}(M_i)
=
\int_{S^1} e^{\phi_H(\theta,0)}\,d\theta
\leq
\int_{S^1}\frac{\pi}{2l_i}\,d\theta
=
\frac{\pi^2}{l_i}.
\end{equation}

It remains to estimate the cylinder height $l_i$ from below.
Let
\[
M_i(r):=F_i(M_i,r)
\]
be the parallel curve at distance $r$, and set
\[
L_i(r):=L_{g_E}(M_i(r)).
\]
By Gauss--Bonnet formula,
\begin{equation}\label{eq:parallelcurve}
    L_i(r)\leq L_{g_E}(M_i)+2\pi r
\leq L_{g_E}(M_i)+2\pi\delta,
\qquad 0\leq r\leq\delta.
\end{equation}

Let
\[
\Psi_i:S^1\times[0,l_i]\longrightarrow A_i
\]
be a conformal diffeomorphism such that
\[
\Psi_i(S^1\times\{0\})=M_i,\qquad\Psi_i(S^1\times\{l_i\})=M_i(\delta).
\]
Let $u_i$ be the harmonic function on $A_i$ with boundary values
\[
u_i|_{M_i}=0,\qquad u_i|_{M_i(\delta)}=1.
\]
Then $u_i\circ\Psi_i$ is harmonic on the flat cylinder
$S^1\times[0,l_i]$ and has boundary values $0$ and $1$ on the two boundary
components. By uniqueness of the Dirichlet problem,
\[
u_i\circ\Psi_i=\frac{t}{l_i}.
\]
Since the Dirichlet energy is conformally invariant in dimension two, it
follows that
\[
\int_{A_i}|\nabla u_i|_{g_E}^2\,dV_{g_E}=\int_{S^1\times[0,l_i]}
\left|\nabla\frac{t}{l_i}\right|^2\,d\theta dt
=\frac{2\pi}{l_i}.
\]

On the other hand, $u_i$ minimizes the Dirichlet energy among all functions
with the same boundary values. In the Euclidean collar coordinates
$(\tau,r)\in A_i$, consider the comparison function
\[
\eta_i(F_i(\tau,r))=\frac{r}{\delta}.
\]
Then $\eta_i|_{M_i}=0$ and $\eta_i|_{M_i(\delta)}=1$, and hence
\[
\frac{2\pi}{l_i}=\int_{A_i}|\nabla u_i|_{g_E}^2\,dA_{g_E}
\leq
\int_{A_i}|\nabla \eta_i|_{g_E}^2\,dA_{g_E}
=\frac{1}{\delta^2}\int_0^\delta L_i(r)\,dr.
\]
Using the bound \eqref{eq:parallelcurve} for $L_i(r)$, this gives
\[
\frac{2\pi}{l_i}
\leq
\frac{1}{\delta^2}\int_0^\delta
\left(L_{g_E}(M_i)+2\pi\delta\right)\,dr=\frac{L_{g_E}(M_i)+2\pi\delta}{\delta}.
\]
Therefore
\[
l_i\geq
\frac{2\pi\delta}{L_{g_E}(M_i)+2\pi\delta}.
\]
Combining this lower bound with \eqref{eq:LgMiupperbound}, we obtain
\[
L_{\hat g}(M_i)
\leq
\frac{\pi}{2\delta}
\left(L_{g_E}(M_i)+2\pi\delta\right).
\]
Summing over all boundary components yields
\[
L_{\hat g}(M)=\sum_{i=1}^n L_{\hat g}(M_i)
\leq\frac{\pi}{2\delta}
\sum_{i=1}^n L_{g_E}(M_i)
+n\pi^2=
\frac{\pi}{2\delta}L_{g_E}(M)+n\pi^2.
\]
This proves the desired uniform upper bound for the total $\hat g$-length of the boundary.

Finally, set $T=\frac{\pi}{2\delta}L_{g_E}(M)+n\pi^2$. By Theorem~\ref{thm::utcompactness}, the family
$\mathcal{F}\subset\widetilde{\mathcal{U}}_T$
is compact in the $C^\infty$ topology. This completes the proof.
\end{proof}

\appendix

\section{Proof of the heat trace estimate}\label{heattrace}
In this appendix, we prove Proposition~\ref{propheattrace}, which provides
a uniform all-time estimate for the Dirichlet heat trace on degenerating
hyperbolic surfaces with geodesic boundary. The proof starts from the classical heat trace formula \cite[Proposition 2.2]{10.1215/S0012-7094-86-05345-7} and recasts it in a form suited to the degeneration analysis in Section~\ref{sectiondegenration}.
\subsection{The heat trace formula and the short-time estimate}

Let $\theta(t;P)$ denote the heat trace of an eigenvalue problem $P$. In particular,
\[
\theta(t;X):=\operatorname{Tr}e^{-t\Delta_X^D}
\]
is the heat trace of the Dirichlet Laplacian on $X$.

The heat trace $\theta(t;X)$ for a compact hyperbolic surface with geodesic boundary is given by the following result.
\begin{proposition}[\cite{10.1215/S0012-7094-86-05345-7}, Proposition 2.2] Let $X$ be a compact orientable hyperbolic surface with geodesic boundary. Then
    \begin{align*}
\theta(t;X) =&-2\pi \chi (X) P(t;\mathbb{H}^2) \\
 & +\frac{e^{-\frac t4}}{4\sqrt{\pi t}}\left\{\sum_{i=1}^n 2b_i\sum_{k=1}^\infty\frac{e^{-\frac{(kb_i)^2}{4t}- \frac{kb_i}{2}}}{\sinh(kb_i)}\right.-\frac{\ell(\partial X)}{2} \\
 & +\sum_{\substack{\gamma\in [C]\\ n_\gamma\text{ is even}}}r_\gamma\sum_{k=1}^\infty\frac{e^{-\frac{(kr_\gamma)^2}{4t}}}{\sinh\bigl(\frac{kr_\gamma}{2}\bigr)} \\
 & +\sum_{\substack{\gamma\in [C]\\ n_\gamma\text{ is odd}}}\left.r_\gamma\sum_{k=1}^\infty\left(\frac{e^{-\frac{(kr_\gamma)^2}{t}}}{\sinh(kr_\gamma)}-\frac{e^{-\frac{((2k-1)r_\gamma)^2}{4t}}}{\cosh\bigl((k-\tfrac12)r_\gamma\bigr)}\right)\right\},
\end{align*}
where $P(t;\mathbb H^2)$ is the on-diagonal heat kernel of the hyperbolic plane.
\end{proposition}

For $0<t\leq 1$, the contributions that we absorb into the uniformly bounded remainder are the terms associated with non-short geodesics:
\begin{enumerate}
    \item the boundary-geodesic terms corresponding to $b_i\geq r^*$;
    \item the terms $\gamma\in[C]$ with $n_\gamma=0,2$ and $r_\gamma\geq r^*$;
    \item the terms $\gamma\in[C]$ with $n_\gamma\neq 0,2$.
\end{enumerate}

For every $a\geq r^*$,
\[
\frac{1}{\sqrt t}e^{-\frac{a^2}{4t}}
=\left(\frac{1}{\sqrt t}e^{-\frac{a^2}{8t}}\right)
e^{-\frac{a^2}{8t}}
\leq C_{r^*}e^{-\frac{a^2}{8}}.
\]
We apply this with $a=kb_i$, $a=kr_\gamma$, $a=2kr_\gamma$, and
$a=(2k-1)r_\gamma$ in the corresponding four sums. Since
\[
\sinh x=\frac12e^x(1-e^{-2x}),\qquad
\cosh x\geq\frac12e^x,
\]
each such contribution is bounded by a Gaussian series with
constants depending only on $r^*$. The boundary sum is uniformly
bounded since there are exactly $n$ boundary components. For the sums
over $\gamma\in[C]$, grouping the geodesics according to unit length
intervals and using \eqref{roughestimate}, the exponential growth in
their number is dominated by the Gaussian decay. 

Consequently, there exists a constant $B_{tr}'(n,\mathtt g)>0$ such
that, for $0<t\leq1$,
\begin{align}
    \bigg|\theta(t;X)-\Big(&-2\pi \chi (X) P(t;\mathbb{H}^2)-\frac{e^{-\frac t4}}{8\sqrt{\pi t}}\ell(\partial X)\notag \\ &+\frac{e^{-\frac t4}}{2\sqrt{\pi t}}\sum_{b_i<r^*} b_i\sum_{k=1}^\infty\frac{e^{-\frac{(kb_i)^2}{4t}- \frac{kb_i}{2}}}{\sinh(kb_i)}\notag\\
  &+\frac{e^{-\frac t4}}{4\sqrt{\pi t}}\sum_{\gamma\in\mathscr{C}_1\cup\mathscr{C}_2}r_\gamma\sum_{k=1}^\infty\frac{e^{-\frac{(kr_\gamma)^2}{4t}}}{\sinh\bigl(\frac{kr_\gamma}{2}\bigr)} \Big)\Bigg|\leq B_{tr}'(n,\mathtt{g}).\label{heattractless1}
\end{align}

\subsection{Long-time decomposition and model estimates}

For $t>1$, we adopt a method similar to that of \cite{Wolpert1987}, decomposing the hyperbolic surface into several collars together with the remaining region and computing the heat kernel on each part separately. The key difference is that our analysis is carried out on a hyperbolic surface $X$ with boundary.

Let $SC$ be the union of all standard subcollars on the double $\widetilde{X}$. Define
\[
R:=X\setminus SC
\]
to be the complement of $SC$ in $X$. The connected components of $SC\cap{X}$ can be classified into three types, according to the way in which the corresponding short closed geodesics interact with the boundary $\partial X$.
More precisely, for each $\gamma\in \mathscr{C}_j/\sim$, $j=1,2$, with length $r_\gamma$, we denote by $SC_j(\gamma)$ the connected component of $SC\cap X$ associated with $\gamma$. When only the length of the underlying geodesic matters, we write this component simply as $SC_j(r_\gamma)$. Similarly, if $\Gamma_i\subset \partial X$ is a short boundary component, we denote the corresponding boundary subcollar by $SC_b(\Gamma_i)$.
We then write
\[
SC_I,\qquad SC_{II},\qquad SC_B\subset SC\cap X
\]
for the unions of all connected components of the forms $SC_1(\gamma)$, $SC_2(\gamma)$, and $SC_b(\Gamma)$, respectively. Thus
\[
SC\cap X=SC_I\sqcup SC_{II}\sqcup SC_B.
\]

We now introduce the following notation:

\begin{description}
\item[$\lambda_k(P)$] the $k$-th eigenvalue of the eigenvalue problem $P$;
\item[$q(\Omega)$] the number of connected components of $\Omega$;
\item[$q'(R)$] the number of connected components of $R$ that do not meet
$\partial X$;
\item[$D\Omega$] the Dirichlet eigenvalue problem on $\Omega$;
\item[$N\Omega$] the eigenvalue problem on $\Omega$ with Dirichlet conditions
on $\partial\Omega\cap\partial X$ and Neumann conditions on the remaining
part of $\partial\Omega$;
\item[$P_1\sqcup P_2$] the disjoint union of the eigenvalue problems $P_1$
and $P_2$.
\end{description}

With this notation, $DSC_I$ now denotes the Dirichlet eigenvalue problem on the connected components of $SC_I$. Similarly, $NSC_I$ is defined componentwise according to the preceding convention: Dirichlet conditions are imposed on the portion inherited from $\partial X$, and Neumann conditions on the artificial cutting boundary. That is,
\begin{equation*}
    DSC_I:= \bigsqcup_{\gamma\in \mathscr{C}_1/\sim}DSC_1(\gamma),\quad NSC_I:= \bigsqcup_{\gamma\in \mathscr{C}_1/\sim}NSC_1(\gamma).
\end{equation*}
The same definitions apply to $SC_{II}$ and $SC_{B}$.

By a standard argument based on Rayleigh quotients, one obtains the comparison
\begin{align}
    \lambda_k(NR\sqcup NSC_I\sqcup &NSC_{II}\sqcup NSC_{B})\leq \lambda_k^D(X)\notag\\
    &\leq \lambda_k(DR\sqcup DSC_I\sqcup DSC_{II}\sqcup DSC_{B}).\label{eigenvlauescomparison}
\end{align}
Consequently, the heat trace $\theta(t;X)$ can be estimated by summing the heat traces associated with each decomposed model region.

We next determine the small eigenvalues, by which we mean the eigenvalues that tend to zero as the surface degenerates. Since constant functions are eigenfunctions of $NSC_1(r)$, we have $\lambda_1(NSC_1(r))=0$. For the second eigenvalue, we have
\begin{lemma}\label{lemeigennsc1}
After decreasing $r^*$ if necessary, for every $0<r<r^*$ one has 
\[
\lambda_2(NSC_1(r))=\lambda_1(NSC_b(r))\longrightarrow 0
\qquad \text{as } r\to 0.
\]
\end{lemma}
\begin{proof}
    Let $\gamma\in\mathscr{C}_1/\sim$ with length $r=r_{\gamma}$. By the conformal invariance of the energy, on
\[
SC_1(\gamma)
=(\mathbb R/r\mathbb Z)\times (2r,\pi-2r),
\qquad g=\frac{1}{\sin^2 v}(du^2+dv^2),
\]
the Laplacian Rayleigh quotient with respect to the metric $g$ is
\begin{equation}\label{laplacerayleigh}
\mathcal R(f):=
\frac{\int_{SC_1(\gamma)}(|\partial_u f|^2+|\partial_v f|^2)\,du\,dv}
{\int_{SC_1(\gamma)}|f|^2\sin^{-2}v\,du\,dv}.
\end{equation}
Expand $f$ in Fourier modes in the $u$-variable,
\[
f(u,v)=\sum_{n\in\mathbb Z} f_n(v)e^{\frac{2\pi i n u}{r}}.
\]
For $0<r<r^*$ and $v\in[2r,\pi-2r]$, one has $\sin v\geq\sin(2r)$. Hence every nonzero Fourier mode satisfies
\[
\mathcal R\bigl(f_n(v)e^{\frac{2\pi i n u}{r}}\bigr)
\geq
\left(\frac{2\pi |n|}{r}\right)^2\sin^2(2r)\geq64n^2.
\]
Thus the non-zero Fourier modes have a uniform positive lower bound.

On the other hand, for the mode $n=0$, the test function
\[
\epsilon(v):=\frac{1}{\log\cot r}\log\tan\frac v2,
\]
is orthogonal to constants and gives
\[
\mathcal R(\epsilon)\lesssim \frac{1}{(\log\cot r)^2}\longrightarrow 0,
\qquad \text{as } r\to 0.
\]
Hence, by the min--max principle, after decreasing $r^*$ once more if
necessary, the second Neumann eigenfunction is radial for $0<r<r^*$; write it
as $h(v)$. Moreover,
$\lambda_2(NSC_1(r))\to0$ as $r\to0$.

The corresponding one-dimensional problem is
\[
-h''(v)=\lambda \sin^{-2}v\,h(v),
\qquad
h'(2r)=h'(\pi-2r)=0.
\]
It is invariant under the reflection $v\mapsto \pi-v$. Since the second Neumann eigenvalue is simple and its eigenfunction has exactly one zero, $h(v)$ is odd with respect to this reflection. Thus,
\[
h(\pi-v)=-h(v),
\]
and, in particular, the second Neumann eigenfunction $\widetilde h(u,v):=h(v)$ on $SC_1(\gamma)$ vanishes on the central geodesic
\[
\gamma=\{v=\frac{\pi}{2}\}.
\]
It follows that, for $0<r<r^*$, the restriction of $\widetilde h$ to
\[
(\mathbb R/r\mathbb Z)\times \left(2r,\frac{\pi}{2}\right)
\]
is an eigenfunction associated with $\lambda_1(NSC_b(r))$. This completes the proof.
\end{proof}
By \cite[Section 3.5]{Wolpert1987}, we get
\begin{equation}\label{eq:lambda3nsc1}
    \lambda_3(NSC_1(r))\geq\lambda_1(DSC_1(r))>\frac{1}{4}.
\end{equation}
Since $SC_b(r)$ is exactly half of $SC_1(r)$, we also have 
\begin{equation}\label{eq:lambda2nscb}
    \lambda_2(NSC_b(r))>\frac14.
\end{equation}

\begin{lemma}
For every $0<r<r^*$, one has
\begin{equation}\label{eq:lambda1nsc2}
\lambda_1(NSC_2(r))\geq64.
\end{equation}
\end{lemma}
\begin{proof}
    We write $NSC_2(r)$ as
\[
(u,v)\in(0,\frac{r}{2})\times (2r,\pi-2r),\qquad g=\frac{1}{\sin^2 v}(du^2+dv^2),
\]
with Dirichlet boundary conditions on $\{u=0\}\cup\{u=\frac{r}{2}\}$ and Neumann boundary conditions on the remaining two sides.
For any admissible test function $f$, the Dirichlet conditions allow us to apply the one-dimensional
Poincar\'e inequality to $u\mapsto f(u,v)$ for each fixed $v$.
Integrating in $v$, we obtain
\[
\begin{aligned}
\int_{NSC_2(r)} |\nabla f|^2\,du\,dv
&\geq\int_{NSC_2(r)} |\partial_u f|^2\,du\,dv \\
&\geq\left(\frac{2\pi}{r}\right)^2
\int_{NSC_2(r)} |f|^2\,du\,dv.
\end{aligned}
\]
It follows from $\sin v\geq \sin(2r)$ on $(2r,\pi-2r)$ that
\[
\mathcal R(f)\geq\left(\frac{2\pi}{r}\right)^2\sin^2(2r).
\]
For $0<r<r^*<\frac{\pi}{4}$, we have
\[
\lambda_1(NSC_2(r))\geq\left(\frac{2\pi}{r}\right)^2\sin^2(2r)\geq64.
\]
\end{proof}

The multiplicity of $0$ in the spectrum of $NR$ is $q'(R)$. By Lemma~\ref{lem::thickgeometry}, $R$ has uniformly bounded
geometry. Consequently, there exists
$\rho_0=\rho_0(n,\mathtt g)>0$ such that
\[
\lambda_1(DR)\geq\rho_0,\qquad\lambda_{q'(R)+1}(NR)\geq\rho_0.
\]
For use below, we decrease $r^*$ once more so that, for every $0<r<r^*$,
\begin{equation}\label{eq:lambdansc1drnr}
    \lambda_2(NSC_1(r))\leq \rho_0\leq\min\{\lambda_1(DR),\lambda_{q'(R)+1}(NR)\}.
\end{equation}
This is possible because $\lambda_2(NSC_1(r))\to0$ as $r\to 0$. This is the first point at which the choice of $r^*$ depends on the topology. From now on, we allow $r^*=r^*(n,\mathtt g)$.

To remove the contribution to the trace from small eigenvalues that tend to zero during degeneration, we consider the truncated heat trace. For an eigenvalue problem $P$, define the truncated heat trace as
\begin{equation*}
    \theta_e(t;P):=\sum_{k>e(P)}e^{-\lambda_k(P) t},
\end{equation*}
where $e(P)$ denotes the number of small eigenvalues that are excluded from the sum. In our setting, the exclusion numbers are given by
\begin{gather*}
    e(NSC_1):=2,\qquad e(NSC_b):=1,\qquad
    e(NSC_2):=0,\qquad e(NR):=q'(R),
\end{gather*}
and
\begin{equation*}
    e(X)=e(DR\sqcup DSC_I\sqcup DSC_{II}\sqcup DSC_B):=2q(SC_I)+q(SC_B)+q'(R).
\end{equation*}

\begin{lemma}\label{heattracedeco}
 With the above notation, we have
\begin{align*}
    \theta_e(t;DR&\sqcup DSC_I\sqcup DSC_{II}\sqcup DSC_{B})\leq\theta_e(t;X)\\
    &\leq \theta_e(t;NR) +\sum_{j=1,2}\sum_{\gamma\in\mathscr C_j/\sim}\theta_e(t;NSC_j(r_\gamma))
+\sum_{b_i<r^*}\theta_e(t;NSC_b(\Gamma_i)).
\end{align*} Moreover, there exists a positive constant $\rho$, depending only on the genus $\mathtt{g}$ and the number of boundary components $n$, such that all eigenvalues appearing in the sums above, as well as those of $DR$, are bounded below by $\rho$.
\end{lemma}
\begin{proof}
Set
\begin{equation}\label{defrho}
    \rho:=\min\{\rho_0,\tfrac14\}.
\end{equation}
By \eqref{eq:lambda3nsc1}--\eqref{eq:lambdansc1drnr}, after the prescribed numbers of eigenvalues have been removed, every
eigenvalue in the decomposed problems is at least $\rho$. Since the same
total number of eigenvalues is excluded on both sides, the two inequalities
follow from \eqref{eigenvlauescomparison}.
\end{proof}

We collect several estimates from \cite{Wolpert1987} that will be used in the sequel. For convenience, we restate them in a form adapted to our notation.

\begin{lemma}[\cite{Wolpert1987}, Theorem 3.5]\label{wollem1}
The following estimate holds.
    \begin{equation*}
        \theta_e(t;NSC_1(r))\leq \theta(t;DSC_1(r))+4\sum_{k=1}^\infty e^{-8\pi^2k^2t}.
    \end{equation*}
\end{lemma}
\begin{lemma}[\cite{Wolpert1987}, Lemma 4.2]\label{wollem2}
The following estimate holds.
    \begin{equation*}
        \theta(t;DSC_1(r))\leq P(t;\mathbb{H}^2)\operatorname{Area}(SC_1(r))+\frac{re^{-\frac t4}}{2\sqrt{\pi t}}\sum_{k=1}^\infty \frac{e^{-\frac{r^2k^2}{4t}}}{\sinh(\frac{kr}{2})}.
    \end{equation*}
\end{lemma}
\begin{lemma}[\cite{Wolpert1987}, Lemma 4.3]\label{wollem3}
The following estimate holds:
\begin{equation*}
\theta(t;DSC_b(r))
\geq\frac{e^{-\frac{t}{4}}}{2}
\left(\frac{1}{\sqrt{\pi t}}\log\frac{\pi}{4r}-1\right).
\end{equation*}
\end{lemma}

Note that $SC_1(r)$ is the double of $SC_b(r)$ across the central
geodesic. Accordingly, the spectrum on $SC_1(r)$ decomposes into the
even and odd parts, corresponding respectively to Neumann and Dirichlet
conditions on the central geodesic. The odd part is precisely the
spectrum of $SC_b(r)$, while the eigenvalues in the even part are no
larger than the corresponding ones in the odd part. Moreover,
Lemma~\ref{lemeigennsc1} gives
\[
\lambda_2(NSC_1(r))=\lambda_1(NSC_b(r)).
\]
After removing the corresponding small eigenvalues, it follows that
\begin{equation}\label{compareSC1SCB}
    \theta_e(t;NSC_b(r))
    \leq \frac12 \theta_e(t;NSC_1(r)),
    \qquad \theta(t;DSC_b(r)) \leq \frac12 \theta(t;DSC_1(r)).
\end{equation}

It remains to derive heat trace estimates for the boundary value problems on $SC_2$. The following lemma shows that for fixed $t\geq 1$, the heat trace of $SC_2$ is bounded under the degeneration of $X$.

We introduce the following model cusp region. Let
\[
    Q_{\mathrm{cusp}}
    := \{(x,y):0<x<\tfrac12,\ y>2\},  \qquad g_c:=\frac{dx^2+dy^2}{y^2}.
\]
We impose Dirichlet boundary conditions on the two
vertical sides
\[
    \{x=0\}\quad \text{and}\quad \{x=\tfrac12\}.
\]
For $\alpha\in\{D,N\}$, we denote by $Q_{\mathrm{cusp}}^\alpha$ the
corresponding eigenvalue problem on $(Q_{\mathrm{cusp}},g_c)$, where the
boundary condition on the remaining boundary component $\{y=2\}$ is
Dirichlet if $\alpha=D$, and Neumann if $\alpha=N$.
\begin{lemma}\label{estimateDSC3}
For each fixed $t\geq 1$, as $r\to0$ one has
\[
    \theta(t;DSC_2(r)) = 2\theta(t;Q_{\mathrm{cusp}}^{D})+o(1),
\]
and
\[
    \theta(t;NSC_2(r))  = 2\theta(t;Q_{\mathrm{cusp}}^{N})+o(1).
\]
\end{lemma}
\begin{proof}
Set
\[
    Q_r:=[0,\frac{r}{2}]\times[2r,\frac{\pi}{2}],\qquad  g=\frac{1}{\sin^2 v}(du^2+dv^2),
\]
which is one half of $SC_2(r)$. We impose Dirichlet boundary conditions
on the two vertical sides $\{u=0\}$ and $\{u=\frac{r}{2}\}$. For $\alpha,\beta\in\{D,N\}$, let
$Q_r^{\alpha\beta}$ denote the corresponding problem on $Q_r$, with
boundary condition $\alpha$ on $\{v=2r\}$ and $\beta$ on
$\{v=\frac{\pi}{2}\}$.

Decomposition into the even and odd parts with respect to
$\{v=\frac{\pi}{2}\}$ gives
\[
    \theta(t;DSC_2(r))  = \theta(t;Q_r^{DD})+\theta(t;Q_r^{DN}),
\]
and similarly
\[
    \theta(t;NSC_2(r))  = \theta(t;Q_r^{ND})+\theta(t;Q_r^{NN}).
\]
It is therefore enough to prove that
\[
    \theta(t;Q_r^{\alpha\beta}) \longrightarrow \theta(t;Q_{\mathrm{cusp}}^\alpha)
\]
for every $\alpha,\beta\in\{D,N\}$.

Under the rescaling
\[
    u=rx,\qquad v=ry,
\]
the domain becomes
   $ [0,\tfrac12]\times[2,\tfrac{\pi}{2r}]$
and
\[
    g_r   =  \frac{r^2(dx^2+dy^2)}{\sin^2(ry)} =  \left(\frac{ry}{\sin(ry)}\right)^2 g_c.
\]
For every fixed
\[
    Q_L:=[0,\tfrac12]\times[2,L],
\]
we have
\[
    g_r=g_c+O_L(r^2)
    \quad\text{in }C^\infty(Q_L).
\]
Hence, if $Q_{r,L}^{\alpha\gamma}$ and
$Q_{c,L}^{\alpha\gamma}$ denote the corresponding problems on $Q_L$,
then
\begin{equation}\label{eq:compact-QL}
    \theta(t;Q_{r,L}^{\alpha\gamma}) = \theta(t;Q_{c,L}^{\alpha\gamma})
    +o_{t,L}(1)
\end{equation}
for $\gamma\in\{D,N\}$.

It remains to control the tails uniformly. Set
\[
    T_{r,L}:=[0,\tfrac12]\times[L,\tfrac{\pi}{2r}].
\]
Since the two vertical sides are Dirichlet,
\[
    \int_{T_{r,L}}|\partial_xf|^2\,dxdy
    \geq 4\pi^2\int_{T_{r,L}}|f|^2\,dxdy.
\]
Moreover, on $T_{r,L}$,
\[
    \frac1{y^2} \leq\frac{r^2}{\sin^2(ry)}\leq\frac{\pi^2}{4y^2}
    \leq \frac{\pi^2}{4L^2}.
\]
Therefore
\begin{equation}\label{eq:spectral-gap}
    \frac{\int_{T_{r,L}}|\nabla f|^2\,dxdy}
    {\int_{T_{r,L}}|f|^2\,dV_{g_r}}
    \geq cL^2
\end{equation}
for some universal $c>0$. The same estimate holds for
\[
    T_L:=[0,\tfrac12]\times[L,\infty)
\]
with the cusp metric $g_c$.

We next obtain a uniform fixed-time heat-trace bound on $T_L$.
Expanding in the Dirichlet eigenfunctions
$\sin(2\pi nx)$, $n\geq1$, in the $x$-variable, write the $n$-th mode
as $f_n(y)\sin(2\pi nx)$. Under the change of variable $y=Le^s$, set
\[
    \widetilde f_n(s):=(Le^s)^{-\frac12}f_n(Le^s).
\]
This gives a unitary identification with $L^2([0,\infty),ds)$, under
which the quadratic form of the $n$-th mode becomes
\[
    \int_0^\infty
    \left(
        \left|\widetilde f_n'+\frac12\widetilde f_n\right|^2
        +4\pi^2n^2L^2e^{2s}|\widetilde f_n|^2
    \right)\,ds.
\]
Divide $[0,\infty)$ into the intervals $[k,k+1]$ and impose Neumann
conditions at the points $s=k$, $k\geq1$. Let
$\lambda_j^{n,k}$, $j\geq0$, denote the eigenvalues associated with the
restriction of this quadratic form to $H^1([k,k+1])$. Since
Since
\[
    \left|\widetilde f_n'+\frac12\widetilde f_n\right|^2
    \geq\frac12|\widetilde f_n'|^2-\frac14|\widetilde f_n|^2,
\]
the min--max principle and the Neumann eigenvalues
$\pi^2j^2$ of an interval of length one give
\[
    \lambda_j^{n,k}\geq\frac{\pi^2}{2}j^2 +4\pi^2n^2L^2e^{2k} -\frac14.
\]
Hence
\[
    \sum_{j=0}^\infty e^{-t\lambda_j^{n,k}}\leq
C_t e^{-4\pi^2tn^2L^2e^{2k}}.
\]
Dirichlet--Neumann bracketing therefore yields
\[
    \operatorname{Tr}(e^{-t\Delta_{T_L}}) \leq
    C_t  \sum_{n=1}^\infty\sum_{k=0}^\infty e^{-4\pi^2tn^2L^2e^{2k}} \leq C_t,
\]
uniformly for $L\geq2$. The same estimate holds on
\[
[0,\tfrac12]\times[L,L']
\]
with the cusp metric $g_c$, uniformly in $L'>L$ and with either Dirichlet or
Neumann boundary condition imposed on $\{y=L\}$ or $\{y=L'\}$.

For $T_{r,L}$, the Dirichlet energies for $g_r$ and $g_c$ agree, while
\[
    dV_{g_c}   \leq  dV_{g_r} \leq \frac{\pi^2}{4}dV_{g_c}.
\]
Thus the min--max principle gives the same uniform fixed-time
heat-trace bound for $T_{r,L}$. 
For either $T_{r,L}$ or $T_L$, by \eqref{eq:spectral-gap} and the spectral theorem,
\[
    \left\|e^{-\frac t2\Delta}\right\|
    \leq e^{-\frac12 c_0tL^2}.
\]
Using the semigroup property and the trace ideal inequality, we obtain
\[
\begin{aligned}
    \operatorname{Tr}(e^{-t\Delta})
    =\left\|e^{-t\Delta}\right\|_1\leq\left\|e^{-\frac t2\Delta}\right\|
    \left\|e^{-\frac t2\Delta}\right\|_1
    \leq e^{-\frac12 c_0tL^2} \operatorname{Tr}\left(e^{-\frac t2\Delta}\right).
\end{aligned}
\]
The uniform fixed-time heat-trace bound at time $\frac{t}{2}$ therefore gives
\begin{equation}\label{eq:tail-QL}
    \operatorname{Tr}(e^{-t\Delta})
    =O_t(e^{-cL^2}),
\end{equation}
after changing $c>0$.

Finally, applying Dirichlet--Neumann bracketing at $\{y=L\}$ and using
\eqref{eq:compact-QL} and \eqref{eq:tail-QL}, we obtain
\[
    \theta(t;Q_r^{\alpha\beta})= \theta(t;Q_{\mathrm{cusp}}^\alpha)+o_{t,L}(1)+O_t(e^{-cL^2}).
\]
Letting first $r\to0$ and then $L\to\infty$ gives
\[
    \theta(t;Q_r^{\alpha\beta}) =\theta(t;Q_{\mathrm{cusp}}^\alpha)+o(1),
\]
independently of $\beta$. Substituting
$(\alpha,\beta)=(D,D),(D,N)$ and
$(N,D),(N,N)$ into the two decompositions at the beginning proves the
lemma.
\end{proof}

In particular, the heat traces at $t=1$ are uniformly bounded for $0<r<r^*$. Together with the uniform lower bound $\lambda_1(NSC_2(r))>1$ and Dirichlet--Neumann monotonicity, this implies that there exist constants $C,c>0$, depending only on the topology, such that
\[
\theta(t;DSC_2(r))+\theta(t;NSC_2(r))\leq Ce^{-ct},\qquad t\geq1,
\]
uniformly for $0<r<r^*$.

\begin{lemma}\label{lem::LC3}
There exists a constant $B_\partial(n,\mathtt g)>0$, depending only on the topology of $X$, such that
\[
\left|
\ell(\partial X)-
\sum_{\gamma\in\mathscr C_2/\sim}4\log\frac{4}{r_\gamma}
\right|
\leq B_\partial(n,\mathtt g).
\]
\end{lemma}
\begin{proof}
Decompose the boundary as
\[
\partial X
=(\partial X\cap SC_B)\sqcup(\partial X\cap SC_{II})\sqcup(\partial X\cap R).
\]
The normal neighborhood
\[
N_{\frac{r^*}{4}}:=\{x\in R\mid d(x,\partial X\cap R)<\frac{r^*}{4}\}
\]
is embedded by the definition of the standard subcollars. Hence
\[
\operatorname{Area}(N_{\frac{r^*}{4}})\geq\operatorname{Length}(\partial X\cap R)\sinh\frac{r^*}{4}.
\]
Since $\operatorname{Area}(R)\leq\operatorname{Area}(X)=-2\pi\chi(X)$, the length of $\partial X\cap R$ is uniformly bounded. The total length of $\partial X\cap SC_B$ is at most $q(SC_B)r^*$ and is therefore uniformly bounded as well.

For each $\gamma\in\mathscr C_2/\sim$, the portion of the boundary contained in the corresponding standard subcollar has length $4\log \cot r_\gamma$. Moreover,
\[
\log \cot r_\gamma-\log\frac{4}{r_\gamma}
\]
is uniformly bounded for $0<r_\gamma<r^*$. Since the number of standard subcollars is bounded by the topology, summing these estimates proves the lemma.
\end{proof}
\subsection{Completion of the proof}

\begin{proof}[Proof of Proposition~\ref{propheattrace}]
For $t\in(0,1]$, Lemma~\ref{lem:number-small-eigenvalues} shows that
the number of Dirichlet eigenvalues below $\frac14$ is uniformly bounded,
while each corresponding heat factor is at most $1$.
Therefore the desired estimate follows directly from
\eqref{heattractless1}. It remains to treat $t\geq1$.

The difference between $\theta_e(t;X)$ and
$\theta_{\geq\frac14}(t;X)$ consists of a uniformly bounded number of terms
whose eigenvalues lie in $[\rho,\frac14]$; it is therefore
$O(e^{-\rho t})$. Combining \eqref{compareSC1SCB},
Lemma~\ref{heattracedeco}, Lemmas~\ref{wollem1} and~\ref{wollem2}, and
Lemma~\ref{estimateDSC3}, we obtain the following upper bound for
$\theta_e(t;X)$:
\begin{align}
    \theta_e(t;X)&\leq \theta_e(t;NR)+\sum_{\gamma\in \mathscr{C}_1/\sim}\Big(P(t;\mathbb{H}^2)\operatorname{Area}(SC_1(r_\gamma))+\frac{r_\gamma e^{-\frac t4}}{2\sqrt{\pi t}}\sum_{k=1}^\infty \frac{e^{-\frac{r_\gamma^2k^2}{4t}}}{\sinh (\frac{kr_\gamma}{2})}\Big)\notag \\
    &\phantom{\leq} +\sum_{b_i<r^*}\Big(P(t;\mathbb{H}^2)\operatorname{Area}(SC_b(\Gamma_i))+\frac{b_i e^{-\frac t4}}{4\sqrt{\pi t}}\sum_{k=1}^\infty \frac{e^{-\frac{b_i^2k^2}{4t}}}{\sinh (\frac{kb_i}{2})}\Big)\notag \\
     &\phantom{\leq}+ (4q(SC_I)+ 2q(SC_{B}))\sum_{k=1}^\infty e^{-8\pi^2k^2t} +C e^{-ct}.\label{thetaeupperb}
\end{align}

Next, we analyze the asymptotic behavior of the series appearing above.
As $r\to 0^+$, we have
\begin{equation*}
    \frac{e^{-\frac{(kr)^2}{4t}}}{\sinh(\frac{kr}{2})}=\frac{2}{kr}+O(1).
\end{equation*}
Consequently, we have
\begin{align}
    r\sum_{k=1}^\infty \frac{e^{-\frac{(kr)^2}{4t}}}{\sinh(\frac{kr}{2})}&=\sum_{k\leq 1/r}\frac{2}{k}+\sum_{k>1/r}\frac{re^{-\frac{(kr)^2}{4t}}}{\sinh(\frac{kr}{2})}+O(1)\notag\\
    &=2\log\frac{1}{r}+O(1).\label{c1c3asy}
\end{align}

On the other hand, as $r\to 0^+$, we also have
\begin{equation*}
    \frac{e^{-\frac{(kr)^2}{4t}-\frac{kr}{2}}}{\sinh(kr)}=\frac{1}{kr}+O(1).
\end{equation*}
This implies that
\begin{align}
    r\sum_{k=1}^\infty \frac{e^{-\frac{(kr)^2}{4t}-\frac{kr}{2}}}{\sinh(kr)}&=\sum_{k\leq 1/r}\frac{1}{k}+\sum_{k>1/r}\frac{r e^{-\frac{(kr)^2}{4t}-\frac{kr}{2}}}{\sinh(kr)}+O(1)\notag \\
    &=\log\frac{1}{r}+O(1).\label{c2asy}
\end{align}
Thus the leading term in \eqref{c2asy} is one half of that in \eqref{c1c3asy}.

By Lemma~\ref{lem::LC3}, we have
\begin{align*}
   \left|-\frac{e^{-\frac{t}{4}}}{8\sqrt{\pi t}}\ell(\partial X) +\sum_{\gamma\in \mathscr{C}_2/\sim}\frac{e^{-\frac t4}}{2\sqrt{\pi t}}\log \frac{1}{r_\gamma}\right|\leq \frac{B_\partial(n,\mathtt g)}{8\sqrt{\pi t}}e^{-\frac{t}{4}}.
\end{align*}
Consequently, by \eqref{c1c3asy}, there exists a constant $C_\partial>0$, depending only on $\mathtt g$ and $n$, such that
\begin{equation}\label{lmest}
   \Big| -\frac{e^{-\frac t4}}{8\sqrt{\pi t}} \ell(\partial X)+\frac{e^{-\frac t4}}{4\sqrt{\pi t}}\sum_{\gamma\in\mathscr{C}_2/\sim}r_\gamma\sum_{k=1}^\infty \frac{e^{-\frac{(kr_\gamma)^2}{4t}}}{\sinh\bigl(\frac{kr_\gamma}{2}\bigr)}\Big|\leq C_\partial e^{-\frac t4}.
\end{equation}

The sum of the area terms in \eqref{thetaeupperb} is bounded above by
\begin{equation*}
    P(t;\mathbb{H}^2)\operatorname{Area}(X)=-2\pi\chi(X)P(t;\mathbb{H}^2).
\end{equation*}
By \eqref{c1c3asy} and \eqref{c2asy}, the contribution of the short boundary length terms can be rewritten as
\begin{equation*}
   \sum_{b_i<r^*} b_i\sum_{k=1}^\infty \frac{e^{-\frac{b_i^2k^2}{4t}}}{\sinh (\frac{kb_i}{2})}=2\sum_{b_i<r^*}b_i\sum_{k=1}^\infty\frac{e^{-\frac{(kb_i)^2}{4t}-\frac{kb_i}{2}}}{\sinh(kb_i)}+O(1).
\end{equation*}

In addition, the number of connected components of $SC$ is bounded in
terms of the topology of $X$. The truncated heat trace of $NR$ is bounded by
$Ce^{-\rho t}$, and every bounded error above is multiplied by
$e^{-\frac{t}{4}}/\sqrt t$. Consequently, all remaining contributions can be absorbed
into an exponentially decaying term in $t$.
It follows that there exist positive constants $B'_{tr}$ and $b'_{tr}$, depending only on $\mathtt{g}$ and $n$, such that
\begin{align*}
    \theta_e(t;X)&\leq -2\pi\chi(X)P(t;\mathbb{H}^2)-\frac{e^{-\frac t4}}{8\sqrt{\pi t}}\ell(\partial X)\\
    &\phantom{\leq}+\frac{e^{-\frac t4}}{2\sqrt{\pi t}}\sum_{b_i<r^*} b_i\sum_{k=1}^\infty\frac{e^{-\frac{(kb_i)^2}{4t}- \frac{kb_i}{2}}}{\sinh(kb_i)}\\
  &\phantom{\leq}+\frac{e^{-\frac t4}}{4\sqrt{\pi t}}\sum_{\gamma\in \mathscr{C}_1\cup \mathscr{C}_2}r_\gamma\sum_{k=1}^\infty\frac{e^{-\frac{(kr_\gamma)^2}{4t}}}{\sinh\bigl(\frac{kr_\gamma}{2}\bigr)}+ B'_{tr}e^{-b'_{tr}t}.
\end{align*}
This establishes the desired upper bound.

For the lower bound, we combine Lemma~\ref{heattracedeco}, Lemma~\ref{wollem3} and Lemma~\ref{estimateDSC3} to prove
\begin{align*}
    \theta_e(t;X)\geq& \sum_{\gamma\in\mathscr{C}_1/\sim}e^{-\frac t4}\Big(\frac{1}{\sqrt{\pi t}}\log \frac{\pi}{4r_\gamma}-1\Big)+\frac12\sum_{b_i<r^*}e^{-\frac t4}\Big(\frac{1}{\sqrt{\pi t}}\log \frac{\pi}{4b_i}-1\Big)\\&+2q(SC_{II})\theta(t;Q_{\text{cusp}}^D)-\bigl(2q(SC_I)+q(SC_B)+q'(R)\bigr)e^{-\rho t}-Ce^{-ct},
\end{align*}
where $\rho$ is defined in \eqref{defrho}. By applying the asymptotic estimates \eqref{c1c3asy}, \eqref{c2asy}, and \eqref{lmest}, we can rewrite the above expression in a form consistent with the upper bound. More precisely, there exist positive constants $B''_{tr}$ and $b''_{tr}$, depending only on $\mathtt{g}$ and $n$, such that
\begin{align*}
     \theta_e(t;X)&\geq -2\pi\chi(X)P(t;\mathbb{H}^2)-\frac{e^{-\frac t4}}{8\sqrt{\pi t}}\ell(\partial X)\\
    &\phantom{\geq}+\frac{e^{-\frac t4}}{2\sqrt{\pi t}}\sum_{b_i<r^*} b_i\sum_{k=1}^\infty\frac{e^{-\frac{(kb_i)^2}{4t}- \frac{kb_i}{2}}}{\sinh(kb_i)}\\
  &\phantom{\geq}+\frac{e^{-\frac t4}}{4\sqrt{\pi t}}\sum_{\gamma\in \mathscr{C}_1\cup \mathscr{C}_2}r_\gamma\sum_{k=1}^\infty\frac{e^{-\frac{(kr_\gamma)^2}{4t}}}{\sinh\bigl(\frac{kr_\gamma}{2}\bigr)}- B''_{tr}e^{-b''_{tr}t}.
\end{align*}
Combining the upper and lower bounds completes the proof.

\end{proof}

\section{Proofs of the graph comparison estimates}
\label{smalleigenvaluesDN}

In this appendix, we prove Proposition~\ref{prop:small-Neumann-product},
Corollary~\ref{corprodN}, and
Proposition~\ref{prop:small-Dirichlet-product}. We retain throughout the
notation and assumptions introduced in Section~\ref{sec::compactnessflat}.

Throughout this appendix, $C$ and $c$ denote positive constants depending
only on $n$ and $T$, and may change from line to line. Since $r^*$ is fixed
in terms of the topology, it will not be displayed in the constants.

For $f\in H^1(X)$, define
\begin{equation*}
    m_\alpha(f) :=
    \frac{1}{A_\alpha}\int_{P_\alpha}f\,dV,
    \qquad m(f) :=
    \bigl(m_\alpha(f)\bigr)_{\alpha\in\mathcal V}.
\end{equation*}
We shall also use the matrix
\begin{equation*}
    \mathsf A:=\operatorname{diag}(A_\alpha).
\end{equation*}

\subsection{Common averaging estimates}

The following estimates are the common analytic input for the Neumann and
Dirichlet comparison arguments; compare
\cite[Lemmas~5 and~6]{Burger1990}.

\begin{lemma}\label{lem:graph-averaging-estimates}
There exist constants $C,C_0>0$, depending only on $n$ and $T$, such that
for every $f\in H^1(X)$,
\begin{equation}\label{eq:common-graph-energy}
    Q(m(f),m(f))
    \leq C\int_X|\nabla f|^2\,dV,
\end{equation}
and
\begin{equation}\label{eq:common-L2-average}
    \int_X|f|^2\,dV
    \leq C_0\sum_{\alpha\in\mathcal V}
    A_\alpha|m_\alpha(f)|^2  + C_0\int_X|\nabla f|^2\,dV.
\end{equation}
\end{lemma}

\begin{proof}
By Lemma~\ref{lem::thickgeometry} and the assumption $b_i\leq T$, the
pieces $P_\alpha$ have uniformly controlled geometry. In particular, the
areas $A_\alpha$ are bounded above and below by positive constants depending
only on $n$ and $T$, and the trace and Poincar\'e inequalities on the
pieces $P_\alpha$ hold with uniform constants.

We first prove \eqref{eq:common-graph-energy}. Fix an interior standard
subcollar
\begin{equation*}
    SC_1(\gamma_i)
    = \mathbb R/(r_{\gamma_i}\mathbb Z)\times (2r_{\gamma_i},\pi-2r_{\gamma_i}),
\end{equation*}
and let $P_\alpha$ and $P_\beta$ be the two adjacent pieces, labeled so
that
\begin{equation*}
    \{v=2r_{\gamma_i}\}\subset\partial P_\alpha,
    \qquad \{v=\pi-2r_{\gamma_i}\}\subset\partial P_\beta.
\end{equation*}
Define the boundary averages
\begin{equation*}
    m_i^-(f) := \frac{1}{r_{\gamma_i}}
    \int_0^{r_{\gamma_i}}
    f(u,2r_{\gamma_i})\,du
\end{equation*}
and
\begin{equation*}
    m_i^+(f)
 :=\frac{1}{r_{\gamma_i}}
    \int_0^{r_{\gamma_i}}
    f(u,\pi-2r_{\gamma_i})\,du.
\end{equation*}
By Cauchy--Schwarz, the uniform trace inequality on $P_\alpha$, and the
Poincar\'e inequality applied to $f-m_\alpha(f)$,
\begin{equation*}
\begin{aligned}
    |m_i^-(f)-m_\alpha(f)|^2 &\leq
    \frac{1}{r_{\gamma_i}}
    \int_0^{r_{\gamma_i}}
 |f(u,2r_{\gamma_i})-m_\alpha(f)|^2\,du \\ &\leq C\|f-m_\alpha(f)\|_{H^1(P_\alpha)}^2 \\ &\leq C\int_{P_\alpha}|\nabla f|^2\,dV.
\end{aligned}
\end{equation*}
Similarly,
\begin{equation*}
    |m_\beta(f)-m_i^+(f)|^2
    \leq
    C\int_{P_\beta}|\nabla f|^2\,dV.
\end{equation*}

We next compare the two boundary averages across the collar. By the
fundamental theorem of calculus,
\begin{equation*}
    m_i^+(f)-m_i^-(f) =\frac{1}{r_{\gamma_i}}
    \int_0^{r_{\gamma_i}}
    \int_{2r_{\gamma_i}}^{\pi-2r_{\gamma_i}}
    \partial_vf(u,v)\,dv\,du.
\end{equation*}
Hence Cauchy--Schwarz gives
\begin{equation*}
    |m_i^+(f)-m_i^-(f)|^2 \leq\frac{\pi-4r_{\gamma_i}}{r_{\gamma_i}}
    \int_{SC_1(\gamma_i)}
    |\partial_vf|^2\,du\,dv.
\end{equation*}
Since
\begin{equation*}
    q_i=\frac{r_{\gamma_i}}{\pi-4r_{\gamma_i}},
\end{equation*}
and the Dirichlet energy is conformally invariant in dimension two,
\begin{equation*}
    q_i|m_i^+(f)-m_i^-(f)|^2 \leq\int_{SC_1(\gamma_i)}|\nabla f|^2\,dV.
\end{equation*}
Combining these estimates and using the uniform upper bound for $q_i$,
we obtain
\begin{equation*}
\begin{aligned}
    q_i|m_\alpha(f)-m_\beta(f)|^2 &\leq Cq_i|m_\alpha(f)-m_i^-(f)|^2 \\
    &\quad +Cq_i|m_i^-(f)-m_i^+(f)|^2 \\
    &\quad +Cq_i|m_i^+(f)-m_\beta(f)|^2 \\
    &\leq C\int_{P_\alpha\cup SC_1(\gamma_i)\cup P_\beta}
    |\nabla f|^2\,dV.
\end{aligned}
\end{equation*}
Summing over the interior collars proves
\eqref{eq:common-graph-energy}.

We now prove \eqref{eq:common-L2-average}. On each piece $P_\alpha$, the
uniform Poincar\'e inequality gives
\begin{equation*}
    \int_{P_\alpha}|f-m_\alpha(f)|^2\,dV \leq
    C\int_{P_\alpha}|\nabla f|^2\,dV,
\end{equation*}
and therefore
\begin{equation}\label{eq:piece-L2}
    \int_{P_\alpha}|f|^2\,dV  \leq
    CA_\alpha|m_\alpha(f)|^2
    +
 C\int_{P_\alpha}|\nabla f|^2\,dV.
\end{equation}

Next consider an interior standard subcollar $SC_1(\gamma_i)$ adjacent to
$P_\alpha$ and $P_\beta$. A one-dimensional estimate in the $v$-direction
gives
\begin{equation*}
\begin{aligned}
    \int_{SC_1(\gamma_i)}|f|^2\,dV &\leq \frac{C}{r_{\gamma_i}}
    \int_0^{r_{\gamma_i}}
    |f(u,2r_{\gamma_i})|^2\,du \\
    &\quad  +  \frac{C}{r_{\gamma_i}}
    \int_0^{r_{\gamma_i}}
    |f(u,\pi-2r_{\gamma_i})|^2\,du \\
    &\quad  + C\int_{SC_1(\gamma_i)} |\nabla f|^2\,dV.
\end{aligned}
\end{equation*}
The uniform trace and Poincar\'e inequalities on the adjacent pieces imply
\begin{equation*}
    \frac{1}{r_{\gamma_i}}
    \int_0^{r_{\gamma_i}}
    |f(u,2r_{\gamma_i})|^2\,du
    \leq CA_\alpha|m_\alpha(f)|^2 + C\int_{P_\alpha}|\nabla f|^2\,dV
\end{equation*}
and
\begin{equation*}
    \frac{1}{r_{\gamma_i}}
    \int_0^{r_{\gamma_i}}
    |f(u,\pi-2r_{\gamma_i})|^2\,du
    \leq   CA_\beta|m_\beta(f)|^2 +  C\int_{P_\beta}|\nabla f|^2\,dV.
\end{equation*}
Thus
\begin{equation}\label{eq:interior-collar-L2}
\begin{aligned}
    \int_{SC_1(\gamma_i)}|f|^2\,dV
    &\leq CA_\alpha|m_\alpha(f)|^2
    +  CA_\beta|m_\beta(f)|^2 \\
    &\quad  + C\int_{P_\alpha\cup SC_1(\gamma_i)\cup P_\beta}
    |\nabla f|^2\,dV.
\end{aligned}
\end{equation}

It remains to control the standard boundary subcollars. Let $b_i<r^*$,
and suppose that
\begin{equation*}
    SC_b(\Gamma_i)
    =  \mathbb R/(b_i\mathbb Z) \times (2b_i,\tfrac{\pi}{2})
\end{equation*}
is attached to $P_\alpha$. For fixed $u$, write
\begin{equation*}
    f(u,v) = f(u,2b_i)
    + \bigl(f(u,v)-f(u,2b_i)\bigr).
\end{equation*}
Since $\sin v\asymp v$ on $(0,\pi/2]$,
\begin{equation*}
    \int_{2b_i}^{\frac{\pi}{2}} \frac{|f(u,2b_i)|^2}{\sin^2v}\,dv  \leq\frac{C}{b_i}|f(u,2b_i)|^2.
\end{equation*}
For the second term, the one-dimensional Hardy inequality gives
\begin{equation*}
    \int_{2b_i}^{\frac{\pi}{2}}
    \frac{|f(u,v)-f(u,2b_i)|^2}{\sin^2v}\,dv \leq
    C\int_{2b_i}^{\frac{\pi}{2}} |\partial_vf(u,v)|^2\,dv.
\end{equation*}
Therefore
\begin{equation}\label{eq:boundary-collar-L2}
    \int_{SC_b(\Gamma_i)}|f|^2\,dV
    \leq\frac{C}{b_i}
    \int_0^{b_i}|f(u,2b_i)|^2\,du +
    C\int_{SC_b(\Gamma_i)}|\nabla f|^2\,dV.
\end{equation}
The trace and Poincar\'e inequalities on the adjacent piece imply
\begin{equation*}
    \frac{1}{b_i}
    \int_0^{b_i}|f(u,2b_i)|^2\,du
    \leq CA_\alpha|m_\alpha(f)|^2 + C\int_{P_\alpha}|\nabla f|^2\,dV.
\end{equation*}
Combining this with \eqref{eq:boundary-collar-L2}, and then summing
\eqref{eq:piece-L2}, \eqref{eq:interior-collar-L2}, and the corresponding
boundary-collar estimates over all pieces and collars, proves
\eqref{eq:common-L2-average}.
\end{proof}

\subsection{The Neumann estimates}

\begin{proof}[Proof of Proposition~\ref{prop:small-Neumann-product}]
We first prove the upper bound. Given a vertex function $s$, define a
function $S$ on $X$ by
\begin{equation*}
    S=s_\alpha \quad\text{on }P_\alpha.
\end{equation*}
Across each interior standard subcollar $SC_1(\gamma_i)$, extend $S$
linearly in the $v$-variable between the two adjacent values. Thus
\begin{equation*}
    \int_{SC_1(\gamma_i)}|\nabla S|^2\,dV
    =q_i|s_{\alpha(i)}-s_{\beta(i)}|^2.
\end{equation*}
On every boundary standard subcollar attached to $P_\alpha$, set
$S=s_\alpha$. Hence the boundary subcollars contribute no Dirichlet energy,
and
\begin{equation*}
    \int_X|\nabla S|^2\,dV = Q(s,s).
\end{equation*}
Moreover,
\begin{equation*}
    \int_X|S|^2\,dV
    \geq \sum_{\alpha\in\mathcal V}
    A_\alpha|s_\alpha|^2  = \|s\|_{\mathcal V}^2.
\end{equation*}
Taking $s$ in the span of the first $j+1$ graph eigenvectors and applying
the min--max principle gives
\begin{equation}\label{eq:Neu-upper}
    \lambda_j^N(X) \leq
    C\mu_j(\mathcal G_X), \qquad 1\leq j\leq m.
\end{equation}

We now prove the reverse estimate. First observe that the graph eigenvalues
are uniformly bounded above. Indeed, the number of vertices and edges is
bounded in terms of $n$, while $q_i\leq C$ and $A_\alpha\geq c$. Hence
\begin{equation*}
    Q(s,s) \leq C\sum_{\alpha\in\mathcal V}|s_\alpha|^2
    \leq C_1\sum_{\alpha\in\mathcal V} A_\alpha|s_\alpha|^2,
\end{equation*}
so that
\begin{equation*}
    \mu_j(\mathcal G_X)\leq C_1,\qquad 1\leq j\leq m.
\end{equation*}

Fix $1\leq j\leq m$. If
\begin{equation*}
    \lambda_j^N(X)\geq\frac{1}{2C_0},
\end{equation*}
then
\begin{equation*}
    \mu_j(\mathcal G_X) \leq C_1
    \leq  2C_0C_1\lambda_j^N(X),
\end{equation*}
and the desired reverse estimate follows.

Suppose now that
\begin{equation*}
    \lambda_j^N(X)<\frac{1}{2C_0}.
\end{equation*}
Let
\begin{equation*}
    E_j
    :=
    \operatorname{span}\{e_0,e_1,\ldots,e_j\},
\end{equation*}
where $e_k$ is a Neumann eigenfunction with eigenvalue
$\lambda_k^N(X)$. For every $f\in E_j$,
\begin{equation*}
    \int_X|\nabla f|^2\,dV\leq\lambda_j^N(X)\int_X|f|^2\,dV.
\end{equation*}
Applying \eqref{eq:common-L2-average} and using
$C_0\lambda_j^N(X)<\frac{1}{2}$, we obtain
\begin{equation*}
\begin{aligned}
    \int_X|f|^2\,dV
     &\leq C_0\sum_{\alpha\in\mathcal V} A_\alpha|m_\alpha(f)|^2
    + C_0\lambda_j^N(X)\int_X|f|^2\,dV \\
    &\leq C_0\sum_{\alpha\in\mathcal V} A_\alpha|m_\alpha(f)|^2
    + \frac12\int_X|f|^2\,dV.
\end{aligned}
\end{equation*}
Consequently,
\begin{equation}\label{eq:Neu-average-lower}
    \sum_{\alpha\in\mathcal V}A_\alpha|m_\alpha(f)|^2
    \geq c\int_X|f|^2\,dV,\qquad f\in E_j.
\end{equation}
In particular, the averaging map $f\mapsto m(f)$ is injective on $E_j$,
so $m(E_j)$ is a $(j+1)$-dimensional subspace of the graph space.

For $0\neq f\in E_j$, Lemma~\ref{lem:graph-averaging-estimates} and
\eqref{eq:Neu-average-lower} give
\begin{equation*}
    \frac{Q(m(f),m(f))}
         {\|m(f)\|_{\mathcal V}^2} \leq C
    \frac{\int_X|\nabla f|^2\,dV}
         {\int_X|f|^2\,dV}
    \leq C\lambda_j^N(X).
\end{equation*}
Applying the min--max principle to $m(E_j)$ gives the desired reverse
estimate in this case. Together with the preceding case,
\begin{equation}\label{eq:Neu-reverse}
    \mu_j(\mathcal G_X)
    \leq C\lambda_j^N(X),
    \qquad 1\leq j\leq m.
\end{equation}
Combining \eqref{eq:Neu-upper} and \eqref{eq:Neu-reverse}, we obtain
\begin{equation*}
    \lambda_j^N(X) \asymp_{n,T}
    \mu_j(\mathcal G_X),
    \qquad 1\leq j\leq m.
\end{equation*}

It remains to prove the uniform lower bound for the next eigenvalue.
Suppose, to the contrary, that
\begin{equation*}
    \lambda_{m+1}^N(X)<\frac{1}{2C_0}.
\end{equation*}
Applying the preceding argument to
\begin{equation*}
    E_{m+1} := \operatorname{span}\{e_0,e_1,\ldots,e_{m+1}\}
\end{equation*}
shows that the averaging map is injective on $E_{m+1}$. This is impossible,
since
\begin{equation*}
    \dim E_{m+1}=m+2,
\end{equation*}
whereas the graph space has dimension
\begin{equation*}
    \#\mathcal V=m+1.
\end{equation*}
Therefore
\begin{equation*}
    \lambda_{m+1}^N(X)\geq\frac{1}{2C_0},
\end{equation*}
which completes the proof.
\end{proof}

\begin{proof}[Proof of Corollary~\ref{corprodN}]
By Proposition~\ref{prop:small-Neumann-product},
\begin{equation*}
    \lambda_j^N(X)
    \asymp_{n,T} \mu_j(\mathcal G_X), \qquad 1\leq j\leq m,
\end{equation*}
and
\begin{equation*}
    \lambda_{m+1}^N(X)\geq c>0.
\end{equation*}
Hence only the first $m$ positive Neumann eigenvalues can tend to zero.
Moreover, by Lemma~\ref{lem:number-small-eigenvalues}, the number of
Neumann eigenvalues in $(0,\frac14)$ is uniformly bounded in terms of the
topology. Therefore every eigenvalue which occurs
in one of the products
\begin{equation*}
    \prod_{0<\lambda_j^N(X)<\frac14}\lambda_j^N(X), \qquad\prod_{j=1}^m\lambda_j^N(X),
\end{equation*}
but not in the other is bounded above and below by positive constants
depending only on $n$ and $T$. Thus
\begin{equation}\label{eq:Neu-small-product}
    \prod_{0<\lambda_j^N(X)<\frac14}\lambda_j^N(X) \asymp_{n,T}\prod_{j=1}^m\lambda_j^N(X).
\end{equation}

It remains to compute the product of the nonzero graph eigenvalues.
They are precisely the nonzero eigenvalues of
\begin{equation*}
    \mathsf A^{-\frac12} \mathsf Q \mathsf A^{-\frac12}.
\end{equation*}
For a vertex $\alpha$, let
\begin{equation*}
    \bigl(
    \mathsf A^{-\frac12}  \mathsf Q
    \mathsf A^{-\frac12}
    \bigr)^{(\alpha)}
\end{equation*}
and $\mathsf Q^{(\alpha)}$ denote the matrices obtained by deleting the
row and column corresponding to $\alpha$.

Since $\mathcal G_X$ is connected, the kernel of $\mathsf Q$ is spanned
by the constant vector $\mathbf 1$. Hence
\begin{equation*}
    \ker\bigl(
    \mathsf A^{-\frac12} \mathsf Q
    \mathsf A^{-\frac12}
    \bigr)
    = \operatorname{span}
    \{\mathsf A^{\frac12}\mathbf 1\}.
\end{equation*}
By the elementary cofactor identity for a rank-$m$ positive semidefinite
matrix, see for instance \cite[Section~0.8.2]{HornJohnson},
\begin{equation*}
    \det
    \bigl(
    \mathsf A^{-\frac12} \mathsf Q
    \mathsf A^{-\frac12}
    \bigr)^{(\alpha)}
    = \left(
    \prod_{j=1}^m\mu_j(\mathcal G_X)
    \right)
    \frac{A_\alpha}
         {\sum_{\beta\in\mathcal V}A_\beta}.
\end{equation*}
On the other hand,
\begin{equation*}
    \det
    \bigl(
    \mathsf A^{-\frac12}  \mathsf Q
    \mathsf A^{-\frac12}
    \bigr)^{(\alpha)}
    = \frac{\det\mathsf Q^{(\alpha)}}
         {\prod_{\beta\neq\alpha}A_\beta}.
\end{equation*}
Since $\mathcal G_X$ is a tree, the matrix-tree theorem gives
\begin{equation*}
    \det\mathsf Q^{(\alpha)}
    =  \prod_{i=1}^m q_i.
\end{equation*}
Combining these identities yields
\begin{equation*}
    \prod_{j=1}^m\mu_j(\mathcal G_X)
    = \frac{\sum_{\alpha\in\mathcal V}A_\alpha}
         {\prod_{\alpha\in\mathcal V}A_\alpha}
    \prod_{i=1}^m q_i.
\end{equation*}
Since the areas $A_\alpha$ are bounded above and below by positive
constants depending only on $n$ and $T$,
\begin{equation*}
    \prod_{j=1}^m\mu_j(\mathcal G_X) \asymp_{n,T} \prod_{i=1}^m q_i.
\end{equation*}
Consequently,
\begin{equation*}
    \prod_{j=1}^m\lambda_j^N(X)
    \asymp_{n,T}
    \prod_{i=1}^m q_i.
\end{equation*}
Finally,
\begin{equation*}
    q_i
    = \frac{r_{\gamma_i}}{\pi-4r_{\gamma_i}} \asymp r_{\gamma_i},
\end{equation*}
so
\begin{equation*}
    \prod_{j=1}^m\lambda_j^N(X) \asymp_{n,T} \prod_{i=1}^m q_i
    \asymp_{n,T}\prod_{i=1}^m r_{\gamma_i}.
\end{equation*}
Combining this with \eqref{eq:Neu-small-product} proves the corollary.
\end{proof}

\subsection{The Dirichlet estimates}

\begin{proof}[Proof of Proposition~\ref{prop:small-Dirichlet-product}]
We first prove the upper bound. Given a vertex function
$s=(s_\alpha)_{\alpha\in\mathcal V}$, construct
$S\in H_0^1(X)$ as follows. On each $P_\alpha$, set $S=s_\alpha$ away
from the original boundary. Across every interior standard subcollar
$SC_1(\gamma_i)$, extend $S$ linearly in the $v$-variable between the
two adjacent vertex values.

If $b_i<r^*$, then on
\begin{equation*}
    SC_b(\Gamma_i)=
  \mathbb R/(b_i\mathbb Z)
    \times
    (2b_i,\tfrac{\pi}{2}),
\end{equation*}
let $S$ be the linear interpolation between $s_{\alpha_b(i)}$ at
$v=2b_i$ and $0$ at $v=\pi/2$. Its energy is
\begin{equation*}
\begin{aligned}
    b_i
    \int_{2b_i}^{\frac{\pi}{2}} |\partial_vS|^2\,dv&=
    \frac{b_i}{\frac{\pi}{2}-2b_i} |s_{\alpha_b(i)}|^2 \\  &= p_i|s_{\alpha_b(i)}|^2.
\end{aligned}
\end{equation*}
If $b_i\geq r^*$, use a fixed cutoff near $\Gamma_i$ which vanishes on
$\Gamma_i$ and equals $s_{\alpha_b(i)}$ away from a uniformly controlled
collar. Since $p_i=1$ in this case, its energy is bounded by
\begin{equation*}
    Cp_i|s_{\alpha_b(i)}|^2.
\end{equation*}
The uniform geometry of the pieces therefore gives
\begin{equation*}
    \int_X|\nabla S|^2\,dV
    \leq
    C\left(
        Q(s,s)
        +  \sum_{i=1}^n  p_i|s_{\alpha_b(i)}|^2
    \right)
    =
    CQ_D(s,s).
\end{equation*}
Moreover,
\begin{equation*}
    \int_X|S|^2\,dV \geq
    c\sum_{\alpha\in\mathcal V}
    A_\alpha|s_\alpha|^2
    = c\|s\|_{\mathcal V}^2.
\end{equation*}
The min--max principle gives
\begin{equation}\label{eq:Dir-upper}
    \lambda_j^D(X)
    \leq C\nu_j(\mathcal G_X),
    \qquad 1\leq j\leq m+1.
\end{equation}

We now prove the reverse estimate. For $f\in H_0^1(X)$,
Lemma~\ref{lem:graph-averaging-estimates} already gives
\begin{equation*}
    Q(m(f),m(f)) \leq C\int_X|\nabla f|^2\,dV.
\end{equation*}
It therefore remains only to estimate the boundary potential.

Suppose first that $b_i<r^*$, and set
\begin{equation*}
    m_i^b(f) :=
    \frac{1}{b_i}\int_0^{b_i}f(u,2b_i)\,du.
\end{equation*}
By the uniform trace inequality on the adjacent piece
$P_{\alpha_b(i)}$ and the Poincar\'e inequality,
\begin{equation*}
    |m_{\alpha_b(i)}(f)-m_i^b(f)|^2 \leq C\int_{P_{\alpha_b(i)}}|\nabla f|^2\,dV.
\end{equation*}
Since $f=0$ on $\Gamma_i$, the fundamental theorem of calculus and
Cauchy--Schwarz give
\begin{equation*}
    |m_i^b(f)|^2
    \leq \frac{\frac{\pi}{2}-2b_i}{b_i}
    \int_{SC_b(\Gamma_i)} |\partial_vf|^2\,du\,dv.
\end{equation*}
Multiplying by $p_i$ and using its definition,
\begin{equation*}
    p_i|m_i^b(f)|^2
    \leq \int_{SC_b(\Gamma_i)}
    |\nabla f|^2\,dV.
\end{equation*}
Since the $p_i$ are uniformly bounded above,
\begin{equation*}
    p_i|m_{\alpha_b(i)}(f)|^2
    \leq   C\int_{P_{\alpha_b(i)}\cup SC_b(\Gamma_i)}
    |\nabla f|^2\,dV.
\end{equation*}

If $b_i\geq r^*$, then $\Gamma_i$ lies on the boundary of
$P_{\alpha_b(i)}$, and $f=0$ on $\Gamma_i$. The uniform Poincar\'e
inequality with a Dirichlet boundary part gives
\begin{equation*}
    |m_{\alpha_b(i)}(f)|^2 \leq
    C\int_{P_{\alpha_b(i)}}|\nabla f|^2\,dV.
\end{equation*}
Since $p_i=1$, the same estimate holds for
$p_i|m_{\alpha_b(i)}(f)|^2$. Summing over all boundary components gives
\begin{equation*}
    \sum_{i=1}^n
    p_i|m_{\alpha_b(i)}(f)|^2 \leq
    C\int_X|\nabla f|^2\,dV.
\end{equation*}
Together with \eqref{eq:common-graph-energy}, this proves
\begin{equation}\label{eq:Dir-graph-energy}
    Q_D(m(f),m(f))
    \leq C\int_X|\nabla f|^2\,dV.
\end{equation}

We now apply the same min--max argument as in the Neumann case. First,
the graph eigenvalues $\nu_j(\mathcal G_X)$ are uniformly bounded above,
since the number of vertices and edges is bounded in terms of $n$, while
$q_i$, $p_i$, and $A_\alpha^{-1}$ are uniformly bounded. Thus
\begin{equation*}
    \nu_j(\mathcal G_X)\leq C_1,
    \qquad 1\leq j\leq m+1.
\end{equation*}

Fix $1\leq j\leq m+1$. If
\begin{equation*}
    \lambda_j^D(X)\geq\frac{1}{2C_0},
\end{equation*}
the uniform upper bound for $\nu_j(\mathcal G_X)$ immediately gives the
desired reverse estimate. If
\begin{equation*}
    \lambda_j^D(X)<\frac{1}{2C_0},
\end{equation*}
let
\begin{equation*}
    E_j  :=\operatorname{span}\{e_1,\ldots,e_j\},
\end{equation*}
where $e_k$ is a Dirichlet eigenfunction with eigenvalue
$\lambda_k^D(X)$. For $f\in E_j$,
\begin{equation*}
    \int_X|\nabla f|^2\,dV
    \leq \lambda_j^D(X)\int_X|f|^2\,dV.
\end{equation*}
Equation~\eqref{eq:common-L2-average} then implies
\begin{equation*}
    \sum_{\alpha\in\mathcal V}
    A_\alpha|m_\alpha(f)|^2
    \geq c\int_X|f|^2\,dV.
\end{equation*}
Hence the averaging map is injective on $E_j$. Combining this estimate
with \eqref{eq:Dir-graph-energy} and applying the min--max principle to
$m(E_j)$ gives the desired reverse estimate in this case. Therefore
\begin{equation}\label{eq:Dir-reverse}
    \nu_j(\mathcal G_X)
    \leq   C\lambda_j^D(X),
    \qquad 1\leq j\leq m+1.
\end{equation}
Together with \eqref{eq:Dir-upper}, we obtain
\begin{equation*}
    \lambda_j^D(X)
    \asymp_{n,T} \nu_j(\mathcal G_X),
    \qquad  1\leq j\leq m+1.
\end{equation*}

The uniform lower bound for the next eigenvalue follows from the same
dimension argument. If
\begin{equation*}
    \lambda_{m+2}^D(X)<\frac{1}{2C_0},
\end{equation*}
then \eqref{eq:common-L2-average} would imply that the averaging map is
injective on
\begin{equation*}
    \operatorname{span}\{e_1,\ldots,e_{m+2}\}.
\end{equation*}
This is impossible, since this space has dimension $m+2$, whereas the
graph space has dimension $m+1$. Hence
\begin{equation*}
    \lambda_{m+2}^D(X)\geq\frac{1}{2C_0}.
\end{equation*}

It remains to compute the product. The eigenvalues
$\nu_j(\mathcal G_X)$ are the eigenvalues of
\begin{equation*}
    \mathsf A^{-\frac12}   (\mathsf Q+\mathsf B)
    \mathsf A^{-\frac12}.
\end{equation*}
Therefore
\begin{equation*}
\begin{aligned}
    \prod_{j=1}^{m+1}\nu_j(\mathcal G_X)
    &= \det\left(
        \mathsf A^{-\frac12}  (\mathsf Q+\mathsf B)
        \mathsf A^{-\frac12}
    \right) \\
    &= \frac{\det(\mathsf Q+\mathsf B)}
         {\prod_{\alpha\in\mathcal V}A_\alpha}.
\end{aligned}
\end{equation*}
Since the areas $A_\alpha$ are bounded above and below by positive
constants depending only on $n$ and $T$,
\begin{equation*}
    \prod_{j=1}^{m+1}\nu_j(\mathcal G_X)\asymp_{n,T} \det(\mathsf Q+\mathsf B).
\end{equation*}
Using the spectral comparison just proved,
\begin{equation*}
    \prod_{j=1}^{m+1}\lambda_j^D(X) \asymp_{n,T} \det(\mathsf Q+\mathsf B).
\end{equation*}

Finally, the lower bound for $\lambda_{m+2}^D(X)$ shows that only the
first $m+1$ Dirichlet eigenvalues can tend to zero. Moreover, by Lemma~\ref{lem:number-small-eigenvalues}, the number of
Dirichlet eigenvalues in $(0,\frac14)$ is uniformly bounded in terms of the
topology. Hence every eigenvalue which occurs in one of the products
\begin{equation*}
    \prod_{\lambda_j^D(X)<\frac14}\lambda_j^D(X),\qquad
    \prod_{j=1}^{m+1}\lambda_j^D(X),
\end{equation*}
but not in the other is bounded above and below by positive constants
depending only on $n$ and $T$. Consequently,
\begin{equation*}
    \prod_{\lambda_j^D(X)<\frac14}\lambda_j^D(X) \asymp_{n,T}\prod_{j=1}^{m+1}\lambda_j^D(X)
    \asymp_{n,T} \det(\mathsf Q+\mathsf B).
\end{equation*}
This completes the proof.
\end{proof}

\section*{Acknowledgments}
The first author would like to thank Professor Colin Guillarmou for helpful discussions.

\renewcommand{\refname}{References}

\end{document}